\documentclass{amsart}

\usepackage{a4wide}
\usepackage{amsmath,amssymb,amsfonts,amsthm}
\usepackage{moreverb,rotating,graphics}
\usepackage[hidelinks]{hyperref}
\usepackage[english]{babel} 
\usepackage{framed,fancybox}
\usepackage{tikz}
\usepackage{float}
\usepackage{caption, subcaption}
\usepackage{mathrsfs, esint, dsfont}
\usepackage{enumerate}

\numberwithin{equation}{section} 

\newtheorem{theorem}{Theorem}[section]

\newtheorem{remark}[theorem]{Remark}
\newtheorem{lemma}[theorem]{Lemma}

\newtheorem{definition}[theorem]{Definition}

\newcommand\1{{\ensuremath {\mathds 1} }} 

\newcommand\CC{{\mathbb C}}

\newcommand\N{{\mathbb N}}

\newcommand\R{{\mathbb R}}

\newcommand\Z{{\mathbb Z}}

\newcommand\bk{{\bf k}}
\newcommand\bK{{\bf K}}

\newcommand\br{{\bf r}}

\newcommand\by{{\bf y}}
\newcommand\bz{{\bf z}}

\newcommand\bnull{{\bf 0}}

\newcommand\rd{{\mathrm{d}}}
\newcommand\re{{\mathrm{e}}}
\newcommand\ri{{\mathrm{i}}}

\newcommand\cB{{\mathcal B}}
\newcommand\cC{{\mathcal C}}
\newcommand\cD{{\mathcal D}}
\newcommand\cE{{\mathcal E}}

\newcommand\cM{{\mathcal M}}
\newcommand\cN{{\mathcal N}}

\newcommand\cR{{\mathcal R}}
\newcommand\cS{{\mathcal S}}
\newcommand\cT{{\mathcal T}}
\newcommand\cV{{\mathcal V}}

\newcommand\cZ{{\mathcal Z}}

\newcommand\fS{{\mathfrak S}}

\newcommand{\sB}{\mathscr{B}}
\newcommand{\sC}{\mathscr{C}}

\newcommand\Tr{{\rm Tr}}
\newcommand\VTr{ \underline{ \rm Tr \,}}

\newcommand{\bra}{\langle}
\newcommand{\ket}{\rangle}

\newcommand{\BZ}{{\Gamma^\ast}} 
\newcommand{\WS}{\Gamma} 
\newcommand{\RLat}{\mathcal{R}^*} 
\newcommand{\Lat}{\mathcal{R}} 
\renewcommand{\div}{\rm div}

\newcommand{\per}{\mathrm{per}}

\newcommand\sqw{\hbox{\rlap{\leavevmode\raise.3ex\hbox{$\sqcap$}}$%
\sqcup$}}
\newcommand\cqfd{\ifmmode\sqw\else{\ifhmode\unskip\fi\nobreak\hfil
\penalty50\hskip1em\null\nobreak\hfil\sqw
\parfillskip=0pt\finalhyphendemerits=0\endgraf}\fi}

\hypersetup{
    colorlinks,
    linkcolor={blue!50!blue},
    citecolor={red}
}
\renewcommand{\eqref}[1]{(\ref{#1})}

\renewcommand{\Re}{{\rm Re}}
\renewcommand{\Im}{{\rm Im}}

\usepackage{float}
\usepackage{todonotes}
\usepackage{cancel}
\usepackage{ulem}
\usepackage{soul}

\DeclareMathOperator{\dist}{dist}
\DeclareMathOperator{\Spec}{Spec}

\newcommand{\eps}{\varepsilon}
\renewcommand{\BZ}{{\mathcal B}}
\newcommand{\fermi}{{\eps_F}}

\def\<{\langle}
\def\>{\rangle}

\newcommand{\rme}{\mathrm{e}}
\renewcommand{\leq}{\leqslant}
\renewcommand{\le}{\leqslant}
\renewcommand{\geq}{\geqslant}
\renewcommand{\ge}{\geqslant}

\newcommand{\dps}{\displaystyle}

\newcommand{\ie}{{\sl i.\,e.\ }}   

\newcommand{\REV}[1]{{#1}} 
\newcommand{\old}[1]{}

\usetikzlibrary{shapes.misc}
\tikzset{cross/.style={cross out, draw=black, minimum size=2*(#1-\pgflinewidth), inner sep=0pt, outer sep=0pt},cross/.default={3pt}}

\title{Numerical quadrature in the Brillouin zone \\ for periodic Schr\"odinger operators}
\author{\'Eric Canc\`es, Virginie Ehrlacher, David Gontier, Antoine Levitt, Damiano Lombardi}
\date{\today}

\begin{document}

\begin{abstract}
As a consequence of Bloch's theorem, the numerical computation of the 
fermionic ground state density matrices and energies of periodic
Schr\"odinger operators involves integrals over the Brillouin zone. These integrals are difficult to compute
numerically in metals due to discontinuities in the integrand. We
perform an error analysis of several widely-used quadrature rules
and smearing methods for Brillouin zone integration. We precisely identify the assumptions implicit in these
methods and rigorously prove error bounds. Numerical results for two-dimensional periodic systems are also provided.
Our results shed light on
the properties of these numerical schemes, and provide guidance as to the
appropriate choice of numerical parameters.
\end{abstract}

\maketitle

\section{Introduction}
The computation of the electronic properties of a $d$-dimensional perfect crystal in a
mean-field setting (e.g. Kohn-Sham density functional theory) formally requires to solve a periodic problem with infinitely many electrons.
In practice, a truncation to a finite supercell composed of $L^d$
crystal unit cells with periodic boundary conditions is necessary for the actual
computation, and $L$ is increased until an acceptable accuracy is
achieved. Bloch's theorem allows for a tremendous reduction of
computational costs by an explicit block-diagonalization of the
Hamiltonian operator, transforming an electronic problem for $L^d N$
one-body wave functions, where $N$ is the number of electron pairs per unit cell, to
$L^d$ electronic problems for $N$ one-body wave functions. In the infinite-$L$
limit, the theorem states that properties of the perfect crystal can be
obtained as an integral over the Brillouin zone (a $d$-dimensional torus) 
of properties of a parametrized system of $N$ electron pairs. The
truncation to a supercell of $L^{d}$ unit cells can then be seen as a
numerical quadrature of this integral. This leads to the famous Monkhorst-Pack numerical scheme~\cite{monkhorst1976special}.

\medskip

Mathematically, the natural question is that of the speed of
convergence of a given electronic property as $L \to \infty$. There appears
a distinction between \textit{insulators}, characterized by a band
gap, and \textit{metals}, with no band gap. In a sense that will be
made precise later, electrons are localized in insulators, but
delocalized in metals. Accordingly, the convergence of electronic
properties is much faster for insulators than for metals. This
translates to quantities of interest being very smooth across the
Brillouin zone in insulators, so that the quadrature is very
efficient: see for instance \cite{Gontier2016_M2AN}, which
proves the exponential convergence with $L$ of a number of properties of
interest for insulators in the reduced Hartree-Fock model. In
this paper, we aim to extend these results to the case of metals,
under natural genericity assumptions on the band structure at the
Fermi level (see Assumptions 1 and 2 in Section \ref{sec:assumptions}). To
reduce the technical content of the paper, we limit ourselves to the
case of independent electrons modeled by a single-particle Hamiltonian
$H=-\frac 12 \Delta + V$ on $L^2(\R^d)$, where $V$ is a given (non
self-consistent) periodic potential.

\medskip

For metals, because of the absence of a band gap, quantities of
interest are discontinuous when the electronic bands $\eps_{n\bk}$
cross the Fermi level $\eps_{F}$, and specific quadrature rules have
to be used to locate this singular set (the \textit{Fermi surface})
and recover an acceptable convergence speed. 
In the simple case when the Fermi level intersects a single isolated band $\BZ \ni \bk \mapsto f(\bk):=\eps_{n_0,\bk} \in \R$ (which can be the case for a metal with $2n_0+1$ electrons per unit cell), the problem of computing the ground-state energy boils down to evaluating
$$
E:=E(\eps_{\rm F})
$$
where the function $\R \ni \eps \mapsto E(\eps) \in \R$ is defined by 
\begin{equation}\label{eq:single_band_1}
E(\eps) = \int_{\BZ}  f(\bk) \, \1(f(\bk) \le \eps) \, d\bk
\end{equation}
and the Fermi level $\eps_{\rm F}$ by the constraint
\begin{equation}\label{eq:single_band_2}
\frac{1}{|\BZ|} \int_{\BZ}   \1(f(\bk) \le \eps_{\rm F}) \, d\bk = \frac 12.
\end{equation}
The Fermi surface then is the level set $\{\bk \in \BZ \, | \, f(\bk) = \eps_{\rm F}\}$ of the function $f$, and the Fermi level is chosen such that the volume of the set $\Omega:=\{\bk \in \BZ \, | \, f(\bk) < \eps_{\rm F}\}$ is half the one of the Brillouin zone. Similar quadrature problems are encountered in the level set method introduced by Osher and Sethian~\cite{OsherSethian1988}. However, the Brillouin zone integration problem encountered in electronic structure calculation has some specificities. First, for one of the most important quantities of interest, namely the ground-state energy, the function to be integrated on $\Omega$ is precisely the level set function $f$ (see \eqref{eq:single_band_1}), which requires a specific analysis. Second, the shape of the Fermi surface can be very complicated for real materials, and the required accuracy is much higher than in standard applications of the level set methods, where linear approximations of the boundary of $\Omega$ from a fixed uniform grid are usually sufficient~\cite{Osher2001,Sethian1999}. Additional technical difficulties appear when the Fermi level intersects several bands, and when the quantity of interest is not the ground-state energy, but some observable involving the Bloch eigenfunctions of $H$ and non only the Bloch eigenvalues $\eps_{n\bk}$, such as e.g. the ground-state density.

\medskip

The most famous Brillouin zone integration method is
the linear tetrahedron method and its improved version by
Bl\"ochl~\cite{Blochl1994} (the Bl\"ochl scheme is not covered
by the results in this paper, and we plan to investigate it in a
forthcoming paper). Other numerical quadratures have been
proposed~\cite{morgan2018efficiency, pickard1999extrapolative} (see
also~\cite{henk2001integration} for an adaptive numerical scheme). In
this paper, we study these quadrature rules, and prove that an
interpolation of quantities of interest to order $p$ coupled to a
reconstruction of the Fermi surface with a method of order $q$ leads,
in general, to a total error of order $L^{-(\min(p+1,q+1))}$: this is
the content of Theorem \ref{thm:err_A}. On the other hand, the error
made on the ground state energy is, to leading order, proportional to
the error made on the number of electrons, which is kept fixed by
varying the Fermi level: therefore, the energy is less sensitive to
the location of the Fermi surface, and the leading order contribution
to the error vanishes, leaving a total error of order
$L^{-(\min(p+1,2q+2))}$.


\medskip

Another way to improve the convergence, and the most widely used method to
compute properties of metals, is the \textit{smearing} method~\cite{methfessel1989high}. This amounts
to regularizing the discontinuities of the occupation
numbers, restoring smoothness across the Brillouin zone. The smearing
parameter $\sigma > 0$, which has the dimension of an energy, should be chosen small enough so that it does
not change the properties of interest too much, but large enough so
that the quadrature is efficient. In numerical codes, this choice is left to
the users, who must use their expertise to select an appropriate value
for the parameter $\sigma$. This is a complex task, and rules of thumb
provide suboptimal choices of $\sigma$.

\medskip

We show in this paper that, up to sub-exponential factors, the total
error for a given smearing parameter $\sigma$ and supercell of size
$L$ is bounded by $C(\sigma^{p+1} + \re^{-\eta \sigma L})$ for some
$C \in \R_+$ and $\eta > 0$, where $p \geq 0$ is the order of the
smearing method used (Theorem \ref{esigmaL-e}). This leads to the
conclusion that $\sigma$ should in practice be varied at the same time
as $L$ to balance the two sources of error. We also investigate the
precise convergence with respect to $L$ at $\sigma > 0$ fixed, and
find the surprising result that, while the convergence is exponential
when the Fermi-Dirac smearing is used, it is super-exponential
(bounded by $C\re^{-\eta L^{4/3}}$) when Gaussian-based smearing
methods are used, due to the different complex-analytic properties of
these functions. Such a
phenomenon has already been observed in
\cite{suryanarayana2013spectral}, in the context of the locality of
the density matrix of metals with Gaussian smearing.

\medskip

The structure of the paper is as follows. We introduce our notation
and recall the basic properties of the periodic Schr\"odinger operator 
$H = - \frac 12 \Delta + V$ in
Section~\ref{sec:notations}. We carefully study the band structure
of this operator in the vicinity of the Fermi level in Section \ref{sec:geometry}. We analyze
interpolation methods in Section~\ref{sec:interp} and smearing methods in
Section~\ref{sec:smearing-methods}. Technical results on the complex-analytic
properties of the integrand in smearing methods are proved in
Appendix \ref{sec:appendixFanalytic}. 
Two-dimensional numerical results illustrating our theoretical results
are presented in Section~\ref{sec:resnum}. Some tests are also given
where our assumptions on the band structure are violated (in the
presence of a van Hove singularity or an eigenvalue crossing at the
Fermi level).

\section{Notation and model}
\label{sec:notations}
In this section, we set our notation, and define the different quantities of interest.

\medskip

Let $d \in \{1,2,3\}$ denote the dimension of the crystal, and $\Lat \subset \R^d$ the crystalline lattice. We denote
by $\RLat$ the dual (or reciprocal) lattice, by $\WS$ the fundamental unit cell of $\Lat$, and we let $\BZ$ be either the first Brillouin
zone, or the fundamental unit cell of $\RLat$. Our results being independent of this choice, we will call $\cB$ ``the Brillouin zone'' for
simplicity. The periodicities in $\Lat$ and $\RLat$ equip the sets $\WS$ and $\BZ$ with the topology of a $d$-dimensional torus.

\medskip

Throughout this paper, $\cC^k(E,F)$ denotes the usual class of $k$ times continuously differentiable functions from $E$ to $F$, and
$L^p(\R^d)$ (resp. $H^s(\R^d)$) denotes the usual Lebesgue (resp. Sobolev) space on $\R^d$, while $L^p_\per$ (resp. $H^s_\per$) denote
spaces of $\Lat$-periodic complex-valued functions on the torus $\WS$. In particular, the space $L^2_\per$ is a Hilbert space when endowed with its natural
inner product
\begin{equation*}
\forall f,g \in L^2_\per, \quad \bra f,g \ket := \int_\WS \overline{f}(\br) g(\br) \rd \br.
\end{equation*}

\medskip

We use the notation $\sB := \sB(L^2_\per)$ to denote the space of bounded operators on $L^2_\per$, and $\fS_p := \fS_p(L^2_\per)$ to denote the Schatten class of compact operators on $L^2_\per$ with finite norms $\| A \|_{\fS_p} := \left( \Tr_{L^2_\per} | A |^p \right)^{1/p}$. In particular, $\fS_1$ is the space of trace-class operators on $L^2_\per$, while $\fS_2$ is the space of Hilbert-Schmidt operators on $L^2_\per$. We finally introduce, for $s \ge 0$, the space
$$
\fS_{1,s} := \left\{ \gamma \in \fS_1 \; | \; \gamma^\ast=\gamma, \; (1-\Delta)^{s/2} \gamma (1-\Delta)^{s/2} \in \fS_1 \right\},
$$
which we endow with the norm 
$$
\|\gamma\|_{\fS_{1,s}} := \| (1-\Delta)^{s/2} \gamma (1-\Delta)^{s/2} \|_{\fS_1}.
$$

\medskip

Let $\mathcal O$ be an operator acting on $L^2(\R^d)$ which commutes with $\Lat$-translations. Thanks to the Bloch-Floquet transform $\cZ$
\cite[Chapter XIII]{ReedSimon4}, we have the decomposition
\begin{equation*}
\cZ^* \mathcal O \cZ = \fint_\BZ^\oplus \mathcal O_\bk \rd \bk,
\end{equation*}
where we denote by $\fint_\BZ := \frac{1}{|\BZ|} \int_\BZ$, and each
of the operators $\mathcal O_\bk$, the Bloch fibers of $\mathcal O$,
acts on $L^{2}_{\per}$. If $\mathcal O$ is locally trace-class, its
{\em trace per unit cell} is then given by
\begin{equation*}
\VTr(\mathcal O) := \fint_\BZ \Tr_{L^2_\per} \left( \mathcal O_\bk \right) \rd \bk.
\end{equation*}

\medskip

We consider the one-body electronic Hamiltonian
\begin{equation*}
H := -\frac12 \Delta + V \quad \textrm{acting on} \quad L^2(\R^d),
\end{equation*}
where $V \in L^\infty_\per(\R^d)$ 
is a real-valued $\Lat$-periodic potential. This hypothesis could be relaxed and a more
general class of potentials could be considered; this choice
simplifies some technical proofs in the Appendix
by ensuring that $V$ is bounded as an
operator on $L^{2}_{\per}$.
It is well-known that $H$ is a bounded from below self-adjoint operator on $L^2(\R^d)$ with domain $H^2(\R^d)$, whose
spectrum is purely absolutely continuous (see e.g. \cite[Theorem~XIII.100]{ReedSimon4}). Since the
operator $H$ commutes with $\Lat$-translations, we can consider its
Bloch-Floquet transform. For $\bk \in \BZ$, the fiber $H_{\bk}$ is
given by
\begin{equation} \label{eq:def:Hk}
H_\bk = \frac12 \left(-\ri \nabla + \bk \right)^2 + V = - \frac12 \Delta - \ri \bk
\cdot \nabla + \frac{\bk^2}{2} + V ,
\end{equation}
which is a self-adjoint operator on $L^2_\per$ with domain $H^2_\per$.  It is bounded from
below, and with compact resolvent. We denote by
$\varepsilon_{1 \bk} \leq \varepsilon_{2 \bk} \leq \cdots$ its eigenvalues ranked in
increasing order, counting multiplicities, and by
$\left( u_{n \bk}\right)_{n \in \N^*} \in \left( H^2_\per \right)^{\N^*}$ a corresponding
$L^2_\per$-orthonormal basis of eigenvectors, so that
\begin{equation*}
H_\bk u_{n \bk} = \varepsilon_{n \bk} u_{n \bk}, \quad \bra u_{n \bk}, u_{m \bk} \ket = \delta_{nm}.
\end{equation*}
Seeing $H_\bk$ as a bounded perturbation of the operator $\frac12 \left(-\ri \nabla + \bk \right)^2$, standard
min-max arguments show that there exist
$\underline{C_{1}}, \overline{C_{1}} \in \R, \underline{C_{2}},
\overline{C_{2}} > 0$ such that
\begin{align}
\label{eq:comparison_eigenvalues}
\underline{C_{1}} + \underline{C_{2}} n^{2/d} \leq \eps_{n\bk} \leq \overline{C_{1}} + \overline{C_{2}} n^{2/d}.
\end{align}

Let us now introduce several physical observables. A fundamental quantity in our
study is the {\em one-body density matrix} at level $\eps \in \R$,
which is the bounded non-negative self-adjoint operator acting on $L^2(\R^{d})$ and defined by
\begin{equation*}
\gamma(\varepsilon) := \1 ( H \leq \varepsilon).
\end{equation*}
Its Bloch-Floquet decomposition is simply
\begin{equation*}
\gamma_\bk(\varepsilon) :=  \1 ( H_\bk \leq \varepsilon) =  \sum_{n \in \N^*} \1(\varepsilon_{n\bk} \leq \varepsilon) |u_{n\bk} \ket \bra u_{n \bk} |.
\end{equation*}
The {\em integrated density of states}  is the function $\cN$ from $\R$ to $\R_+$ defined by
\begin{equation} \label{eq:IDoS}
\forall \varepsilon \in \R, \quad \cN(\varepsilon) := \VTr ( \gamma(\varepsilon) ) = \sum_{n \in \N^*}
\fint_\BZ \1 (\varepsilon_{n \bk} \leq \varepsilon) \rd \bk.
\end{equation}
The function $\cN$ is a non-decreasing continuous function, with $\cN(-\infty) = 0$ and $\cN(+\infty) = +\infty$. In particular, if $N$
denotes the number of electron pairs per unit cell, then $\cN^{-1}(\{N\})$ is a non-empty interval, of the form $[\varepsilon_-,
\varepsilon_+]$. If $\varepsilon_- < \varepsilon_+$, the system is an \textit{insulator}. In this case, supercell methods are very efficient
to compute numerically the properties of the crystals (see for instance~\cite{monkhorst1976special,Gontier2016_M2AN}). In this article, we focus on the
\textit{metallic} case $\varepsilon_- = \varepsilon_+$. In this case, the {\em Fermi level} of the system is the unique number
$\varepsilon_F := \varepsilon_- = \varepsilon_+$.

\medskip

We then introduce the {\em integrated density of energy} $E:\R \to \R$ defined by
\begin{equation}
\label{eq:IDoE}
E(\varepsilon) := \VTr(H \gamma(\varepsilon)) = \sum_{n \in \N^*}  \fint_{\BZ} \varepsilon_{n\bk} \1(\varepsilon_{n\bk} \le \varepsilon) \rd \bk,
\end{equation}
and the (zero-temperature) {\em ground state energy} (per unit cell) $E := E(\varepsilon_F)$.
Finally, the {\em electronic density} up to level $\varepsilon$ is defined as the density of the locally trace-class $\cR$-periodic operator $\gamma(\epsilon)$, that is the real-valued function $\rho_\varepsilon \in L^1_\per$ characterized by
\begin{equation*}
\forall v \in L^\infty_\per, \quad \int_{\WS} \rho_\varepsilon(\br) v(\br) \rd \br = \VTr \left( v \gamma(\varepsilon) \right) =
\sum_{n \in \N^*}  \fint_{\BZ}
\bra u_{n\bk} |  v | u_{n\bk} \ket \1(\varepsilon_{n\bk} \le \varepsilon) \rd \bk.
\end{equation*}
We therefore have
\begin{equation}
\label{eq:totalDensity}
\forall \br \in \WS, \quad \rho_\varepsilon(\br) := \sum_{n \in \N^*} \fint_{\BZ} | u_{n\bk} |^2 (\br) \1 (\varepsilon_{n \bk} \le \varepsilon)  \rd \bk.
\end{equation}
The (zero-temperature) {\em ground state electronic density} then is $\rho := \rho_{\varepsilon_F}$.

\medskip

\begin{remark}[Observables]
In this article, we focus on the numerical calculation of the
integrated density of states $\cN$, the Fermi level $\varepsilon_F$,
the ground state energy per unit cell $E$ and the ground state
electronic density $\rho$ of the system. It is possible to extend
our results to a broader class of observables, but precising the
complete set of assumptions needed to formulate our results is
cumbersome and we will not proceed further in this direction.
\end{remark}

\medskip

\begin{remark}[Discretization errors]
  \label{rem:discretization}
The goal of this paper is to study various numerical schemes to
compute Brillouin zone integrals of the form above. In particular,
we assume that the eigenvalues $\eps_{n\bk}$ and eigenvectors
$u_{n \bk}$ are perfectly known on some mesh of the Brillouin zone
$\cB$, and we study the numerical errors coming from the
discretization of the Brillouin zone in \eqref{eq:IDoS},
\eqref{eq:IDoE} and \eqref{eq:totalDensity}. We do not study the
effects of numerical errors in the computation of the $\eps_{n\bk}$
and $u_{n\bk}$ themselves. We also do not study more complicated
nonlinear models such as the periodic Kohn-Sham model.

It is however interesting to note that the discretization of the
eigenvalue problem $H_{\bk} u_{n\bk} = \varepsilon_{n\bk} u_{n\bk}$
in a $\bk$-consistent manner is not trivial: since $H_{\bk}$ is not
equal but unitarily equivalent to $H_{\bk+\bK}$ for $\bK \in \RLat$,
a fixed Galerkin space will yield eigenvalues $\eps_{n\bk}$ that are
not $\RLat$-periodic. Conversely, popular choices such as a $\bk$-dependent
Galerkin space $V_{\bk}$ consisting of all the plane waves
$\re^{\ri \bK \cdot \br}$ such that
$\frac 1 2 |\bk+\bK|^{2} \leq E_{\rm cut}$ will yield eigenvalues
that are periodic, but not continuous as a function of $\bk$. It is
possible to restore continuity by using a smooth cutoff; we plan to
explore this possibility from a numerical analysis viewpoint in a forthcoming paper.
\end{remark}

\section{Properties of the band structure at the Fermi level}
\label{sec:geometry}
We first partition the Brillouin zone $\BZ$ into several sets, whose
definitions are gathered here for the sake of clarity. For a given energy level
$\varepsilon \in \R$, we introduce, for $n \in \N$ (and with
the convention that  $\varepsilon_{0 \bk} = -\infty$), the sets
\begin{align*}
\cB_{n}(\varepsilon) & := \{\bk \in \BZ, \ \varepsilon_{n \bk} < \varepsilon <
\varepsilon_{n+1,\bk}\} & \text{$n$-th component of the Brillouin zone,} \\
\cS_{n}(\varepsilon) & := \{\bk \in \BZ, \ \varepsilon_{n \bk} = \varepsilon\}
&  \text{$n$-th sheet of the \REV{\old{Fermi surface}level set}},
\\
\cS(\eps) & := \bigcup_{n \in \N} \cS_{n}(\eps) = \{\bk \in \BZ, \ \exists
n \in \N,\, \varepsilon_{n \bk} = \varepsilon\} & \text{\REV{\old{Fermi surface}}level set}.
\end{align*}
\REV{The set $\cS(\varepsilon_F)$ is called the Fermi surface.} For all $\varepsilon \in \R$, it holds that
\begin{equation} \label{eq:decomp_B}
\cB = \cS(\varepsilon) \cup \left( \bigcup_{n \in \N} \cB_n(\varepsilon) \right).
\end{equation}
From~\eqref{eq:comparison_eigenvalues}, the unions in the above definition of $\cS(\eps)$ and in~\eqref{eq:decomp_B} are finite. The sets $\cS_1(\eps), \, \cS_2(\eps), \cdots $ are 
pairwise disjoint outside of \textit{band crossings} where $\eps_{n \bk} = \eps_{n+1,\bk} = \eps$ for some $n$. The boundary of $\cB_{n}(\eps)$ is
$\partial\cB_{n}(\eps) = \cS_{n}(\eps) \cup \cS_{n+1}(\eps)$. A typical example of the sets $\cB_n(\eps)$ and $\cS_n(\eps)$ is
represented in Figure~\ref{fig:BZ}.
\begin{figure}[h!]
\centering
\includegraphics[width=.3\textwidth]{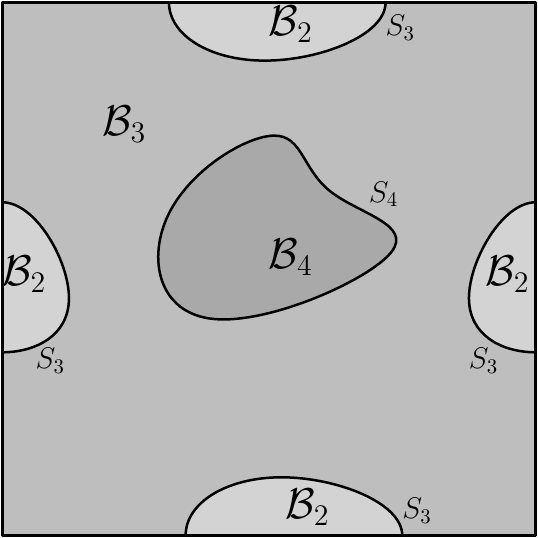}
\caption{Partitioning of the Brillouin zone into the sets
$\BZ_{n}(\varepsilon)$ and $\cS_{n}(\varepsilon)$, for a given
value of $\varepsilon$.}
\label{fig:BZ}
\end{figure}

As we shall see in Lemma~\ref{lem:coverBZ}, various spectral quantities are
smooth in $\BZ_{n}(\eps)$ and on $\cS_{n}(\eps)$. It will be useful in our analysis
to extend this smoothness to the following neighborhoods of these
sets (see Figure~\ref{fig:sets_in_1d}): for $\delta > 0$, we set
\begin{align*}
\widetilde{\cB}_n(\eps, \delta) & := \bigcup_{\varepsilon' \in (\varepsilon
- \delta, \varepsilon + \delta)} \cB_n(\varepsilon')  = \left\{ \bk \in \cB, \ \exists \varepsilon' \in
(\varepsilon - \delta, \varepsilon + \delta), \ \varepsilon_{n \bk} < \varepsilon' <
\varepsilon_{n+1,\bk} \right\}, \\
\widetilde{\cS}_n(\varepsilon, \delta) & := \bigcup_{\varepsilon' \in (\varepsilon -
\delta, \varepsilon + \delta)}
\cS_n(\varepsilon') = \left\{ \bk \in \cB, \ \varepsilon_{n \bk} \in (\varepsilon - \delta,
\varepsilon + \delta) \right\}.
\end{align*}
\REV{Here and thereafter, smooth means infinitely differentiable.}

\begin{figure}
\centering

\begin{tikzpicture}
\draw[->] (-6, -3) -> (-6, 2) node[right] {$\varepsilon$};
\draw[->] (-6.5, -2.7) -> (6.5, -2.7) node[right] {$\bk$};

\draw[line width =1, red] (-6.2,-0.5) node[left] {$\varepsilon_F$}-- (6, -0.5) ;
\draw (-6.2,-0.7) -- (6, -0.7) ;  \draw  (-6.2,-0.3) -- (6, -0.3); \draw[<->] (-6.1,
-0.5) -- (-6.1, -0.7) node[left] {$\delta$};


    \draw[domain=-6.2:-5.24,smooth,variable=\t, purple] plot ({\t},{-\t*\t/7-\t/3+2});
    \draw[domain=-5.24:4.2,smooth,variable=\t, purple] plot ({\t},{\t*\t*\t*\t/200-\t*\t/6-\t/8});
    \draw[domain=4.2:4.5,smooth,variable=\t, purple] plot ({\t},{-\t*\t/7-\t/3+2})
    node[right] {$\varepsilon_{n\bk}$};

    \draw[domain=-6.2:-5.24,smooth,variable=\t, blue] plot ({\t},{\t*\t*\t*\t/200-\t*\t/6-\t/8});
    \draw[domain=-5.24:4.2,smooth,variable=\t, blue] plot ({\t}, {-\t*\t/7-\t/3+2});
    \draw[domain=4.2:5,smooth,variable=\t, blue] plot
    ({\t},{\t*\t*\t*\t/200-\t*\t/6-\t/8}) node[right] {$\varepsilon_{n+1, \bk}$};

    \draw[line width=0.3, dashed] (-5.65, -2.7) -- (-5.65, -0.7);
    \draw[line width=0.3, dashed] (-5.35, -2.7) -- (-5.35, -0.3);
    \draw[line width=0.3, dashed] (-5.1, -2.7) -- (-5.1, -0.3);
    \draw[line width=0.3, dashed] (-4.65, -2.7) -- (-4.65, -0.7);
    \draw[line width=0.3, dashed] (-2.9, -2.7) -- (-2.9, -0.7);
    \draw[line width=0.3, dashed] (-1.9, -2.7) -- (-1.9, -0.3);
    \draw[line width=0.3, dashed] (1.05, -2.7) -- (1.05, -0.3);
    \draw[line width=0.3, dashed] (1.8, -2.7) -- (1.8, -0.7);

    \draw[line width=2, purple] (-5.65, -2.6) -- (-5.35, -2.6);
    \draw[line width=2, purple] (-5.1, -2.6) -- (-4.65, -2.6);
    \draw[line width=2, purple] (-2.9, -2.6) -- (-1.9, -2.6);
    \draw[line width=2, purple] (1.05, -2.6) -- (1.8, -2.6);

    \draw[line width=0.3, dashed] (3, -2.7) -- (3, -0.3);
    \draw[line width=0.3, dashed] (3.35, -2.7) -- (3.35, -0.7);

    \draw[line width=2, blue] (3, -2.6) -- (3.35, -2.6);

    \draw[line width=2, densely dotted, orange] (-5.65, -3) -- (-4.65, -3);
    \draw[line width=2, densely dotted, orange] (-2.9, -3) -- (1.8, -3);

    \draw[line width=2, densely dotted, purple] (-6.2, -2.9) -- (-5.35, -2.9);
    \draw[line width=2, densely dotted, purple] (-5.1, -2.9) -- (-1.9, -2.9);
    \draw[line width=2, densely dotted, purple] (1.05, -2.9) -- (3.35, -2.9);
    \draw[line width=2, densely dotted, blue] (3, -2.8) -- (4.5, -2.8);

    \draw[line width=2, purple] (3.5, 2.3) -- (4, 2.3) node[right]
    {$\widetilde{\cS}_{n}(\varepsilon_F, \delta)$};
    \draw[line width=2, blue] (3.5, 1.8) -- (4, 1.8) node[right]
    {$\widetilde{\cS}_{n+1}(\varepsilon_F, \delta)$};
    \draw[line width=2, densely dotted, orange] (3.5, 1.3) -- (4, 1.3) node[right]
    {$\widetilde{\cB}_{n-1}(\varepsilon_F, \delta)$};
    \draw[line width=2, densely dotted, purple] (3.5, 0.8) -- (4, 0.8) node[right]
    {$\widetilde{\cB}_{n}(\varepsilon_F, \delta)$};
    \draw[line width=2, densely dotted, blue] (3.5, 0.3) -- (4, 0.3) node[right]
    {$\widetilde{\cB}_{n+1}(\varepsilon_F, \delta)$};

    \end{tikzpicture}

    \caption{A schematic view of the sets $\widetilde{\cB}_n$ and $\widetilde{\cS}_n$.}
    \label{fig:sets_in_1d}
    \end{figure}


\subsection{Assumptions on the Fermi level}
\label{sec:assumptions}
Recall that we are studying metallic systems, so that the Fermi level
$\varepsilon_F$ is uniquely defined. In particular, the Fermi surface
$\cS(\varepsilon_F)$ is non empty. We make the following two
assumptions to ensure a good mathematical structure of the Fermi surface:
\begin{align*}
&\textbf{Assumption 1 (no band crossings at $\varepsilon_F$): } & \quad \forall n \neq m, \quad
\cS_n(\varepsilon_F) \cap \cS_m(\eps_F) = \emptyset; \\
&\textbf{Assumption 2 (no van Hove singularities at $\varepsilon_F$): } & \quad \forall n \in \N^*, \ \forall \bk \in \cS_n(\varepsilon_F),
\quad  \nabla_\bk \varepsilon_{n \bk} \neq \bnull.
\end{align*}
From Assumption 2, we see that, for all $n \in \N^*$, the map
$\varepsilon_n : \bk \mapsto \varepsilon_{n \bk}$ is a submersion
near the Fermi surface, so that
$\cS_{n}(\varepsilon_F) = \eps_{n}^{-1}(\eps_{F})$ is either empty or a smooth compact co-dimension~1
submanifold of the torus $\BZ$. From
Assumption 1, $\cS(\varepsilon_F)$ is itself a smooth compact
manifold, as the finite disjoint union of the
$\cS_{n}(\varepsilon_F)$.
\begin{remark}[Genericity of hypotheses]
It is an interesting question to know whether such assumptions hold
generically, i.e. for almost every potential $V$. For a generic
smooth family $H_{\bk}$ of self-adjoint operators on $L^2_\per$ with
compact resolvents, eigenvalue crossings happen on a set of
codimension $3$ \cite{vonNeumann1993}, and band degeneracies on
isolated points. \REV{\old{There is therefore no reason to expect neither band
crossings nor degeneracies at the Fermi level in the physical cases
$d \leq 3$, and we would naturally expect that both these
assumptions are generically true.} Band crossings or degeneracies thus do not appear in general at the Fermi level in the physical cases $d \leq 3$, and we would naturally expect that both these assumptions are generically true. }There are however two important
caveats: first, many natural conjectures on the genericity of
properties of the band structure still remain open in general (see
\cite{kuchment2016overview} for an overview), and second, symmetries may force crossings or degeneracies. For instance,
Assumption 1 is violated in the case of the free electron gas, or in
graphene \cite{Fefferman2012}, due to the high symmetries of these
systems. \REV{We will treat the case of the graphene in future work. In the sequel, the quality of Assumption 1 and Assumption 2 are measured by quantities $\delta_0 > 0$ and $C_\nabla > 0$ respectively (see Lemma~\ref{lem:coverBZ} below). For instance, for systems with Fermi level close to van-Hove singularity, $C_\nabla$ will be small.}
\end{remark}
\medskip

Let us define by 
$$
\underline{M}:= \min\{ n\in \N^*, \quad \cS_n(\varepsilon_F)\neq \emptyset \} \quad \mbox{ and } \overline{M}:= \max\{ n\in \N^*, \quad \cS_n(\varepsilon_F) \neq \emptyset \}.
$$
The existence of ${\underline{M}}$ and ${\overline{M}}$ comes from~\eqref{eq:comparison_eigenvalues}, and it naturally holds that $\underline{M} \leq \overline{M}$.

\medskip

In the next lemma, we collect a number of properties of the Fermi
surface and of spectral quantities on the sets
$\widetilde \BZ_{n}(\eps_{F},\delta)$ and
$\widetilde \cS_{n}(\eps_{F},\delta)$. In order to state these results,
we introduce the density matrices
\begin{equation} \label{eq:gammak}
\gamma_{n \bk} := \sum_{m=1}^{n}|u_{m \bk} \ket \bra u_{m \bk} | \quad \text{acting on} \quad L^2_\per,
\end{equation}
which are well-defined operators whenever $\varepsilon_{n \bk} < \varepsilon_{n+1,\bk}$, and the associated densities
\begin{equation} \label{eq:def:rhonk}
\rho_{n \bk}=\sum_{m=1}^n |u_{m\bk}|^2 \in L^1_\per.
\end{equation}
We recall that a smooth map $F : \R^{d} \to E$ where $E$ is a Banach space is
real-analytic if it is locally equal to its Taylor series.

\medskip

\begin{lemma}
\label{lem:coverBZ}
Under Assumption 1 and 2, there exists $\delta_0 > 0$ and $C_\nabla > 0$ such that
\begin{enumerate}[(i)]
\item For any $n \in \N^{*}$ and for all $0<\delta \leq \delta_0$, $\widetilde{\cS}_n(\eps_F, \delta) \neq
\emptyset$ if and only if $\underline M \leq n \leq \overline M$; \label{lem:coverBZ_0}
\item for all ${\underline{M}} \le m < n \le {\overline{M}}, \; \widetilde{\cS}_m(\varepsilon_F, \delta_0) \cap \widetilde{\cS}_n(\varepsilon_F, \delta_0) = \emptyset$ ;   \label{lem:coverBZ_1}
\item for all ${\underline{M}} \le n \le {\overline{M}}$ and all $\bk \in \widetilde{\cS}_n(\varepsilon_F, \delta_0), \; | \nabla_\bk \varepsilon_{n \bk} | \geq C_{\nabla}$ ; \label{lem:coverBZ_2}
\item for all ${\underline{M}} \le n \le {\overline{M}}$ and all
$\varepsilon \in (\eps_F - \delta_0, \eps_F + \delta_0)$,
$\cS_n(\varepsilon)$  is a non-empty smooth compact manifold
of co-dimension 1, with non-zero Haussdorf measure
$\left| \cS_n(\eps) \right|_{\rm Hauss} >
0$. The same properties hold for
$\cS(\varepsilon)$; \label{lem:coverBZ_3}
\item assume in addition that $V \in H^s_\per$ for some $s \ge 0$. Then, for all ${\underline{M}} \le n \le {\overline{M}}$,
\begin{itemize}
	\item the map $\bk \mapsto \gamma_{n \bk}$ is real-analytic from $\widetilde{\cB}_n(\eps_F, \delta_{0})$ to $\fS_{1,s+2}$;
	\item the map $\bk \mapsto \rho_{n \bk}$ is real-analytic from $\widetilde{\cB}_n(\eps_F, \delta_{0})$ to $H^{s+2}_\per$;
	\item the map $\bk \mapsto \varepsilon_{n \bk}$ is real-analytic from $ \widetilde{\cS}_n(\eps_F, \delta_{0})$ to $\R$.
\end{itemize}
    \label{lem:coverBZ_4}
\end{enumerate}
\end{lemma}

\begin{proof}
Assertions {\em(\ref{lem:coverBZ_0})} and {\em(\ref{lem:coverBZ_1})} come from Assumption 1 and the continuity
of the maps $\bk \mapsto \varepsilon_{n \bk}$. We now prove {\em
(\ref{lem:coverBZ_4})}. Using {\em (\ref{lem:coverBZ_4})} and
Assumption {\em (\ref{lem:coverBZ_1})} we will deduce {\em
(\ref{lem:coverBZ_2})}, which implies {\em (\ref{lem:coverBZ_3})}.

Let ${\underline{M}} \le n \le {\overline{M}}$ and $\delta_{0} > 0$ small enough so that {\em(\ref{lem:coverBZ_0})} and {\em (\ref{lem:coverBZ_1})} hold true.  

The map $\bk \mapsto \varepsilon_{n \bk}$ is continuous on $\widetilde{\cB}_n(\varepsilon_F, \delta_0)$ and for all $\bk \in \widetilde{\cB}_n(\varepsilon_F, \delta_0)$, we have $\varepsilon_{n+1, \bk} - \varepsilon_{n \bk} > 0$. Therefore, there exists  $0 < \delta_1 < \delta_0$ and $g > 0$ such that
\[
\forall \bk \in \widetilde{\cB}_n(\varepsilon_F, \delta_1), \quad \varepsilon_{n+1, \bk} - \varepsilon_{n \bk} \ge g.
\]
We first prove the analyticity of $\bk \mapsto \gamma_{n \bk}$ on $\widetilde{\cB}_n(\varepsilon_F, \delta_1)$. Let $\bk_0 \in \widetilde{\cB}_n(\varepsilon_F, \delta_1)$.
We then set $\varepsilon := \varepsilon_{n \bk_{0}} + g/2$, and consider the positively oriented loop (we denote by
$\Sigma := \min \sigma(H)$)
\begin{equation*}
\sC := [\Sigma - 1 - \ri, \varepsilon - \ri] \cup [\varepsilon - \ri, \varepsilon + \ri] \cup [\varepsilon + \ri, \Sigma - 1 +\ri] \cup [\Sigma - 1 + \ri, \Sigma - 1 -\ri].
\end{equation*}
\begin{figure}[!ht]
\centering
\begin{tikzpicture}
\draw[->] (-5, 0) -> (5,0); \node at (5.1, 0.2) {$\R$};
\draw[line width = 2] (-3,-0.2) -- (-3, 0.2); \node at (-3, 0.5) {$\Sigma$};
\foreach \x in {-2.5, -2, -1.7, -1.5, -1, 0, 1.2, 3} {
\draw[red, line width = 1.5] (\x, -0.15) -- (\x, 0.15);
}
\node[red] at (-2.5, -0.35) {$\varepsilon_{1 \bk}$};
\node[red] at (-1.5, -0.35) {$\cdots$};
\node[red] at (0, -0.35) {$\varepsilon_{n \bk}$};
\node[red] at (1.3, -0.35) {$\varepsilon_{n+1, \bk}$};

\draw[blue, line width = 1] (-3.5, -0.7) -- (-3.5, 0.7) -- (0.6, 0.7) -- (0.6,
-0.7) -- (-3.5, -0.7);
\draw[blue] (-1, 0.5) -- (-1.2, 0.7) -- (-1, 0.9);
\node[blue] at (0.9, 0.7) {$\sC$};

\node[blue] at (0.75, 0.12) {$\varepsilon$};
\end{tikzpicture}
\caption{The contour $\sC$.}
\label{fig:contourC}
\end{figure}

From the definition of $\sC$, we see that there exists $0 < \delta_2 < \delta_1$ such that
\[
\forall \bk \in \BZ \ \text{s.t.} \ | \bk - \bk_0 | \le \delta_2, \quad \forall \lambda \in \sC, \quad | \lambda - H_\bk | \ge g/4.
\]
In particular, we see that Cauchy's residual formula
\begin{equation} \label{eq:gammaAnalytic}
{\gamma_{n \bk}} = \dfrac{1}{2 \pi \ri} \oint_{\sC} \dfrac{\rd \lambda}{\lambda - H_{\bk}},
\end{equation}
holds for $\bk$ in a neighborhood of $\bk_0$. In addition, 
$$
H_{\bk} = H_{\bk_0} + (\bk - \bk_{0}) \cdot (- \ri
\nabla+\bk_0) + \frac{|\bk - \bk_{0}|^2}2,
$$
so that, for $\bk - \bk_{0}$ small enough,
$$
(\lambda - H_{\bk})^{-1} = (\lambda-H_{\bk_0})^{-1} \left( 1 - \left[
\left( (\bk - \bk_{0}) \cdot (- \ri \nabla+\bk_0) + \frac{|\bk - \bk_{0}|^2}2 \right) (\lambda-H_{\bk_0})^{-1} \right] \right)^{-1}.
$$
For all $0 \le s' \le s$, the linear operator
$(\lambda - H_{\bk_0})^{-1}$ is continuous from $H^{s'}_\per$ to
$H^{s'+2}_\per$ by classical elliptic regularity results, and the operator in brackets is bounded on $H^{s'}$. Therefore
we obtain, by expanding in Neumann series, that the map
$\bk \mapsto \gamma_{n \bk}$ is real-analytic from a neighborhood of
$\bk_0$ to $\sB(H^{s'}_\per,H^{s'+2}_\per)$. Using the fact that
$\gamma_{n\bk} = \gamma_{n\bk}^{2}$ and a bootstrap argument, this
implies that the map $\bk \mapsto \gamma_{n \bk}$ is real-analytic
from a neighborhood of $\bk_0$ to $\sB(L^2_\per,H^{s+2}_\per)$.

\medskip

Let $(v_{1\bk_0},\cdots, v_{n\bk_0})$ be an $L^2_\per$-orthonormal basis of $\mbox{Ran}(\gamma_{n\bk_0})$, $\widetilde v_{j\bk} = \gamma_{n\bk}v_{j\bk_0}$ and 
$$
v_{i\bk}= \sum_{j=1}^n \widetilde v_{j\bk} [S_\bk^{-1/2}]_{ji},
$$
where $S_\bk$ is the overlap matrix defined by $[S_\bk]_{ji} = \langle
\widetilde v_{j\bk}, \widetilde v_{i\bk} \rangle$, so that
$(v_{1\bk},\dots,v_{n\bk})$ is an $L^2_\per$-orthonormal basis of ${\rm Ran}\, \gamma_{n\bk}$. Of course, it holds that $\widetilde{v}_{j \bk_0} = v_{j \bk_0}$ and $S_{\bk_0} = {\rm Id}_n$. It is easily
checked that the map $\bk \mapsto (v_{1 \bk},\cdots, v_{n \bk}) \in
(H^{s+2}_\per)^n$ is well-defined and real-analytic  in a neighborhood
of $\bk_0$. In particular, we have, in a neighborhood of $\bk_0$, that
$$
(1-\Delta)^{(s+2)/2} \gamma_{n\bk} (1-\Delta)^{(s+2)/2} = \sum_{j=1}^n | w_{j\bk} \rangle \langle w_{j\bk}|,
$$
where the maps $\bk \mapsto w_{j\bk}:= (1-\Delta)^{(s+2)/2}v_{j\bk} \in L^2_\per$ are real-analytic in a neighborhood of $\bk_0$. It follows that the map 
$\bk \mapsto \gamma_{n\bk}$ is analytic from a neighborhood of $\bk_0$ to $\fS_{1,s+2}$, hence from $\widetilde{\cB}_n(\varepsilon_F, \delta_1)$ to $\fS_{1,s+2}$.

\medskip

On the other hand, whenever $s+2 > \frac{d}{2}$ (which is the case whenever $s \ge 0$ and $d \le 3$), it holds that $H^{s+2}_\per$ is an algebra. 
As a result, from the analyticity of the map  $\bk \mapsto (v_{1 \bk},\cdots, v_{n \bk}) \in  (H^{s+2}_\per)^n$, we deduce the analyticity of the maps 
$\bk \mapsto \rho_{n\bk} :=  \sum_{i=1}^{n}\left| u_{i \bk} \right|^2  = \sum_{i=1}^{n}\left| v_{i \bk} \right|^2  \in H^{s+2}_\per$.

\medskip

To prove the analyticity of $\bk \mapsto \varepsilon_{n \bk}$
on $\widetilde{\cS}_n(\varepsilon_F, \delta_1)$, we notice that
$\widetilde{\cS}_n(\varepsilon_F, \delta_1) = \widetilde{\cB}_{n-1}(\varepsilon_F,
\delta_1) \cap \widetilde{\cB}_n(\varepsilon_F, \delta_1)$. As a result,
both $\bk \mapsto \gamma_{n-1,\bk}$ and $\bk \mapsto \gamma_{n \bk}$ are analytic from $\widetilde{\cS}_n(\varepsilon_F, \delta_1)$ to $\fS_{1,s+2}$. 
Therefore, on $\widetilde{\cS}_n(\varepsilon_F, \delta_1)$, it holds that
\begin{align*}
\varepsilon_{n \bk} &= \Tr_{L^2_\per} \left[ H_\bk ( \gamma_{n \bk} - \gamma_{n-1,\bk} ) \right] \\ &= \Tr_{L^2_\per} \left[ (1-\Delta)^{-1/2} H_\bk (1-\Delta)^{-1/2} (1-\Delta)^{1/2} ( \gamma_{n \bk} - \gamma_{n-1,\bk} ) (1-\Delta)^{1/2} \right] .
\end{align*}
Since the maps $\widetilde{\cS}_n(\varepsilon_F, \delta_1) \ni \bk \mapsto (1-\Delta)^{-1/2} H_\bk (1-\Delta)^{-1/2}  \in \sB$ and $\widetilde{\cS}_n(\varepsilon_F, \delta_1) \ni \bk  \mapsto (1-\Delta)^{1/2} 
( \gamma_{n \bk} - \gamma_{n-1,\bk} ) (1-\Delta)^{1/2}  \in \fS_1$  are real-analytic, this proves the real-analyticity of the map $\bk \mapsto \varepsilon_{n \bk}$ on $\widetilde{\cS}_n(\varepsilon_F, \delta_1)$.
\end{proof}


\subsection{Density of states}

We are now concerned with the properties of the density of states $\cD(\eps)$, defined as
the derivative of the integrated density of states $\cN(\eps)$ defined in~\eqref{eq:IDoS}. Since $\cN$ is a 
non-decreasing continuous function which is increasing on $\sigma(H)$ and
constant outside $\sigma(H)$, $\cD$ is a positive measure on $\R$ whose support is exactly
$\sigma(H)$.  Our goal in this section is to establish that, under Assumptions 1 and 2, both
$\cN$ and $\cD$ are smooth around $\eps_{F}$.

\medskip

We recall some tools of differential geometry. In the sequel, if $S$ is a
(smooth) hypersurface of $\BZ$, we denote by $\rd \sigma_S$ the Haussdorf measure on $S$. We write
$\rd \sigma$ instead of $\rd \sigma_S$ if there is no risk of confusion. We first recall the co-area formula, which allows the integration of a function
$g : \BZ \to \R$ along the level sets of another function $\cE : \BZ \to \R$.

\begin{lemma}[co-area formula \cite{Federer1959curvature}]
Let $\cE : \BZ \to \R$ be a Lipschitz function and $g \in L^1(\BZ)$, then
\begin{equation*}
\int_{\BZ} g(\bk) | \nabla \cE (\bk)| \rd \bk = \int_{\R} \left( \int_{\cE^{-1} \{ \varepsilon \}} g(\bk) \rd \sigma(\bk) \right) \rd \varepsilon.
\end{equation*}
\end{lemma}
Setting $f = g | \nabla \cE|$, we deduce that for all $f: \BZ \to \R$ such that
$\bk \mapsto \frac{f(\bk)}{|\nabla \cE(\bk)|} \in L^{1}(\BZ)$, then
\begin{equation} \label{eq:coarea2}
\int_{\BZ} f(\bk) \rd \bk = \int_{\R} \left( \int_{\cE^{-1} \{
\varepsilon \}} \frac{f(\bk)}{|\nabla \cE(\bk)|} \rd
\sigma(\bk) \right) \rd \varepsilon.
\end{equation}
This allows us to differentiate functions defined as integrals on
the sets $\cS_n(\eps)$ and $\BZ_n(\eps)$ with respect to the energy
level $\eps$.
\begin{lemma}
\label{lemma:diff_integral_param}
Let $f \in \cC^p(\BZ, \R)$. Under Assumptions 1 and 2 and with the notation of Lemma \ref{lem:coverBZ}, the maps $F_n: (\varepsilon_F - \delta_0, \varepsilon_F + \delta_0) \to \R$, ${\underline{M}} \le n \le {\overline{M}}$, defined by
\begin{equation*}
\forall \varepsilon \in (\varepsilon_F - \delta_0, \varepsilon_F + \delta_0), \quad	F_n(\varepsilon) := \int_{\cS_n(\varepsilon)} f(\bk) \rd \sigma (\bk),
\end{equation*}
are of class $\cC^p$, and we have
\begin{equation}
\label{eq:F'_eps}
\forall \varepsilon \in (\varepsilon_F - \delta_0, \varepsilon_F + \delta_0), \quad	F_n'(\varepsilon) = \int_{\cS_n(\varepsilon)} \dfrac{\div \left( f(\bk) \frac{\nabla \varepsilon_{n \bk}}{|\nabla \varepsilon_{n \bk} |} \right) }{| \nabla \varepsilon_{n \bk} |} \rd \sigma (\bk).
\end{equation}
\end{lemma}

\begin{proof}
Using suitable cut-off functions, it is sufficient to prove the result for $f$ compactly
supported in $\widetilde \cS_{n}(\eps_{F},\delta_{0})$. The outgoing normal unit vector of
$\cS_n(\varepsilon)$ at $\bk$ (oriented so that its interior is $\{\bk \in \BZ, \eps_{n
\bk} < \eps\}$) is $\nu_\bk := \nabla \varepsilon_{n \bk} / | \nabla \varepsilon_{n
\bk} |$. Using the divergence theorem, we get
\begin{align*}
F_n(\varepsilon) = \int_{\cS_n(\varepsilon)} f(\bk) \rd \sigma (\bk) & = \int_{\cS_n(\varepsilon)} f(\bk) \frac{\nabla \varepsilon_{n \bk}}{| \nabla \varepsilon_{n \bk} | } \cdot \nu_\bk \rd \sigma(\bk)
= \int_{ \{ \bk \in \BZ, \ \varepsilon_{n \bk} < \varepsilon \} } \div \left(  f(\bk) \frac{\nabla \varepsilon_{n \bk}}{| \nabla \varepsilon_{n \bk} | } \right) \rd \bk.
\end{align*}
We now use the co-area formula~\eqref{eq:coarea2} with $\widetilde{f}(\bk) := \1
(\varepsilon_{n \bk} \le \varepsilon) \,\div \left(  f(\bk)
\frac{\nabla \varepsilon_{n \bk}}{| \nabla \varepsilon_{n
\bk} | } \right) $ and $\cE(\bk) = \eps_{n \bk}$, so that
\begin{align*}
F_n(\varepsilon)  & = \int_{-\infty}^{\varepsilon} \left( \int_{\cS_n(\varepsilon')} \dfrac{\div \left( f(\bk) \frac{\nabla \varepsilon_{n \bk}}{| \nabla \varepsilon_{n \bk} |} \right) }{| \nabla \varepsilon_{n \bk} |} \rd \sigma (\bk) \right) \rd \varepsilon'.
\end{align*}
Differentiating this expression leads to \eqref{eq:F'_eps}, and
iterating $p$ times leads to the result.
\end{proof}

Formally, using the co-area formula on the integrated density of
states $\cN(\eps)$ would yield
\begin{align*}
\cN(\eps) = \frac{1}{|\BZ|}\sum_{n \in \N^*} \int_{-\infty}^{\eps} \left(\int_{\cS_{n}(\eps')}\frac 1 {|\nabla\eps_{n\bk}|} \rd \sigma(\bk)\right) \rd \eps',
\end{align*}
but the integrand may not be well-defined for all $\eps' < \eps$. This argument is however
valid close to the Fermi surface, and we therefore have the following result.
\begin{lemma}
\label{lemma:DOS} Under Assumptions 1 and 2 and with the notation of Lemma \ref{lem:coverBZ},
the integrated density of states $\mathcal N$ is smooth on
$(\eps_{F} - \delta_{0}, \eps_{F} + \delta_{0})$. Moreover, we have
\begin{align*}
\forall \varepsilon \in (\varepsilon_F - \delta_0, \varepsilon_F + \delta_0), \quad  \mathcal D(\eps) := \mathcal N'(\eps) = \frac {1} {|\BZ|}\sum_{n={\underline{M}}}^{{\overline{M}}}\int_{\cS_{n}(\eps)}
\frac{1}{|\nabla \eps_{n \bk}|} \rd\sigma(\bk) \quad > 0.
\end{align*}
  \end{lemma}
  \begin{proof}
  Applying the co-area formula with $f(\bk) = \1(\eps_{n \bk} \le \eps_F)$, we have
  \begin{align*}
  \mathcal N(\eps) := \sum_{n={\underline{M}}}^{{\overline{M}}} \fint_{\BZ} \1(\eps_{n \bk}
  \leq \eps) \rd \bk
  &= \mathcal N(\eps_{F} - \delta_{0}) + \frac 1 {|\BZ|}
  \sum_{n={\underline{M}}}^{{\overline{M}}} \int_{\eps_{F} - \delta_{0}}^{\eps} \rd \eps'
  \int_{\cS_{n}(\eps')} \frac 1 {|\nabla \eps_{n \bk}|} \rd \sigma(\bk),
  \end{align*}
  from which we get
  \begin{equation} \label{eq:def:Deps}
  \mathcal D(\eps) := \mathcal N'(\eps) = \frac 1 {|\BZ|}
  \sum_{n={\underline{M}}}^{{\overline{M}}} \int_{\cS_{n}(\eps)} \frac 1 {|\nabla \eps_{n \bk}|} \rd \sigma(\bk).
  \end{equation}
  By Lemmas \ref{lem:coverBZ} and \ref{lemma:diff_integral_param}, this function is
  smooth and positive.
  \end{proof}

  The function $\cD$ appearing in~\eqref{eq:def:Deps} is called the {\em density of states}.
  This lemma justifies our Assumptions 1 and 2 as natural assumptions to ensure a
  smooth density of states at the Fermi level. The presence of crossings at the Fermi level
  may indeed yield singularities of the density of states, as is well-known for instance in
  graphene. Similarly, a zero of the band gradient (leading to so-called ``flat bands'') produces \textit{van Hove singularities} in the density of
  states.

    \medskip
    
  Following the same steps as in Lemma~\ref{lemma:DOS}, we obtain that the integrated density of energy defined in~\eqref{eq:IDoE} is also smooth on $(\eps_{F} - \delta_{0}, \eps_{F} + \delta_{0})$, and that its derivative (the {\em density of energy}) satisfies
  \begin{equation} \label{eq:DoE}
      \forall \varepsilon \in (\varepsilon_F - \delta_0, \varepsilon_F + \delta_0), \quad  E'(\varepsilon) = \varepsilon \cD(\varepsilon).
  \end{equation}

  We also record here the following technical lemma on the volume of the
  sets $\widetilde{\cS_n}(\varepsilon, \delta)$, which will be used in
  our analysis. It is an easy consequence of the co-area formula.
  \begin{lemma}
  \label{lem:vol_Sndelta} Under Assumptions 1 and 2 and with the notation of Lemma \ref{lem:coverBZ},
  there exists $C \in \R_+$ such that, for all $\varepsilon \in (\varepsilon_F -
  \delta_0, \varepsilon_F + \delta_{0})$ and all $0 \le \delta <
  \delta_0$ such that $(\varepsilon-\delta, \varepsilon+\delta) \subset
  (\varepsilon_F - \delta_0, \varepsilon_F + \delta_0)$, it holds that $|
  \widetilde{\cS_n}(\varepsilon, \delta) | \le C
  \delta$ for all ${\underline{M}} \le n \le {\overline{M}}$.
  \end{lemma}

  \begin{proof}
  We apply the co-area formula~\eqref{eq:coarea2} with $f(\bk) := \1(\varepsilon -
  \delta \le \varepsilon_{n \bk} \le \varepsilon + \delta)$ and $\cE(\bk) = \varepsilon_{n\bk}$, and get that
  \begin{equation*}
  | \widetilde{\cS}_n(\varepsilon, \delta) | = \int_{\varepsilon - \delta}^{\varepsilon + \delta} \left( \int_{\cS_n(\varepsilon')} \dfrac{1}{| \nabla \varepsilon_{n \bk} |} \rd \sigma (\bk) \right) \rd \varepsilon'.
  \end{equation*}
  From Lemma \ref{lem:coverBZ}, the map
  $\varepsilon' \mapsto \int_{\cS_n(\varepsilon')} | \nabla \varepsilon_{n \bk}
  |^{-1} \rd \sigma (\bk)$ is continuous and bounded on
  $(\eps_{F}-\delta_{0},\eps_{F}+\delta_{0})$. The proof follows.
  \end{proof}

\section{Interpolation methods}
\label{sec:interp}
In this section, we investigate methods based on the local interpolation of the functions
$\bk \mapsto \eps_{n \bk}$.  We consider families of linear
interpolation operators $\Pi^L : \cC^0(\BZ,\R) \to \cC^0(\BZ,\R)$ indexed by $L \in \N^*$, which strongly converge to the identity operator when $L$ goes to infinity, i.e.
$$
\forall f \in \cC^0(\BZ,\R), \quad \Pi^L f \mathop{\longrightarrow}_{L \to \infty} f \mbox{ in } \cC^0(\BZ,\R).
$$
We
say that $\left( \Pi^L \right)_{L \in \N^*}$ is of order $(p+1) \in \N$ if, for all
$\eta > 0$, there exists $C_\Pi^\eta \in \R_+$ such that, for all $p' \in \N$, all open sets
$\Omega \subset \cB$, and all $f \in \cC^{p'+1}(\Omega_\eta, \R)$, where
$\Omega_\eta := \left\{\bk \in \BZ, \, d(\bk, \Omega) \leq \eta\right\}$ is the
$\eta$-neighborhood of $\Omega$, it holds that
\begin{equation} \label{eq:f-Pif}
\sup_{\bk \in \Omega} | f(\bk) - \Pi^L f(\bk) | \le \dfrac{C_\Pi^\eta}{L^{\min(p,p')+1}}  \sup_{\bk \in \Omega_\eta} | f^{(p'+1)}(\bk) |.
\end{equation}
One of the most used interpolation operator is the \textit{linear
tetrahedron method} (and its {\em improved} version~\cite{Blochl1994} -see also~\cite{kawamura2014improved}- that will be studied in a future work). In this case, we choose a
sequence of uniform tetrahedral meshes $(\cT^L)_{L \in \N}$, and we define
$\Pi^L f$ as the piecewise linear function (linear on each tetrahedron
$T \in \cT^L$) interpolating $f$ at the vertices of $\cT^L$. In three
dimensions, a linear tetrahedron method constructed from a regular
$L \times L \times L$ mesh of the torus $\cB$ is of order $2$. Similarly, the quadratic method described in~\cite{methfessel1983analytic,boon1986singular}
and the cubic tetrahedron method described
in~\cite{Zaharioudakis2005quadratic} are of order $3$ and $4$
respectively. We chose this convention so that $p$ denotes the degree
of the polynomial in a usual polynomial interpolation (of order
$p+1$).

\begin{remark}[Local interpolation]
The approximation property above is local, in the sense that the
quality of the approximation at a point depends only on the
smoothness of the function near this point. This is necessary to
interpolate efficiently the functions $\eps_{n \bk}$, which are not smooth
across the whole Brillouin zone. By contrast, a Fourier
interpolation on the whole Brillouin zone would not satisfy this
condition: discontinuities in the interpolated function produce
Gibbs oscillations, which slow down the convergence of the Fourier
series even far from the point of discontinuity.
\end{remark}

\subsection{Error on the Integrated density of states and on the Fermi level}

Recall that the integrated density of states $\cN(\varepsilon)$ is defined in~\eqref{eq:IDoS}. In practice, we cannot compute $\cN(\varepsilon)$, but only an approximation of it. We therefore introduce
\begin{equation} \label{eq:IDoS^L}
\forall \varepsilon \in \R, \quad \cN^{L,q}(\varepsilon) := \sum_{n \in \N^*} \fint_\BZ \1 ( \varepsilon_{n \bk}^{L,q}\leq \varepsilon) \rd \bk \quad \textrm{with} \quad \varepsilon_{n \bk}^{L,q} := \Pi^{L,q} \left[ \varepsilon_{n \bk} \right],
\end{equation}
where $\left( \Pi^{L,q} \right)_{L \in \N^*}$ is a family of interpolation operators of order $q+1$. In practice, the
integral in~\eqref{eq:IDoS^L} is performed analytically (hence at low computational cost). Because of the smoothness of $\eps_{n\bk}$ near $\cS_{n}(\eps_{F})$, we are able to control the error on this function.
\begin{lemma}[Error on the integrated density of states] \label{lem:errorN}
Under Assumptions 1 and 2, there exist $C \in \R_+$ and $\delta > 0$ such that
\begin{equation*}
\forall L \in \N^*, \ \max_{\varepsilon \in [\varepsilon_F - \delta, \varepsilon_F +
\delta]} \left| \cN(\varepsilon) - {\cN}^{L,q}(\varepsilon) \right| \le \frac{C}{L^{q+1}}
\end{equation*}
\end{lemma}
\begin{proof}
Let $\delta^{L,q}_\BZ$ be the maximum error between $\eps_{n \bk}$ and $\eps_{n \bk}^{L,q}$ on
the whole Brillouin zone, i.e.
\begin{equation*}
\delta^{L,q}_{\BZ} := \max_{n \le {\overline{M}}} \max_{\bk \in \BZ} |\varepsilon_{n \bk}^{L,q} - \varepsilon_{n \bk} |.
\end{equation*}
From the fact that $\bk \mapsto \varepsilon_{n \bk}$ is Lipschitz and~\eqref{eq:f-Pif} in the case $p' = 0$, we deduce that
$\lim_{L \to \infty} \delta^{L,q}_{\BZ} = 0$. For $L$ large enough, $\delta_{\BZ}^{L,q} < \delta_{0}/2$, where $\delta_0$ was defined in
Lemma \ref{lem:coverBZ}. Let $\eps \in [\eps_{F} - \delta_{0}/2, \eps_{F} + \delta_{0}/2]$. We have
\begin{equation} \label{eq:errorN1} \cN(\varepsilon) - \cN^{L,q}(\varepsilon) = \sum_{n \in \N^*} \fint_{\BZ} (\1 (\varepsilon_{n \bk} \leq
\varepsilon) - \1 ( \varepsilon_{n \bk}^{L,q} \leq \varepsilon)) \rd \bk.
\end{equation} The integrand in~\eqref{eq:errorN1} can only be nonzero in the \textit{discrepancy regions} where $\varepsilon_{n \bk} \le
\varepsilon < \varepsilon_{n \bk}^{L,q}$ or $\varepsilon^{L,q}_{n \bk} \le \varepsilon < \varepsilon_{n \bk}$. In these regions, it holds
that $| \varepsilon_{n \bk} - \varepsilon | \le | \varepsilon_{n \bk} - \varepsilon^{L,q}_{n \bk} |$, so that they are included in
$\widetilde \cS_{n}(\eps,\delta^{L,q}_{\BZ})$. We can therefore rewrite~\eqref{eq:errorN1} as
\begin{equation} \label{eq:errorcNs} \cN(\varepsilon) - \cN^{L,q}(\varepsilon) = \frac 1 {|\BZ|}\sum_{n \in \N^*} \int_{\widetilde
\cS_{n}(\eps,\delta^{L,q}_{\BZ})} (\1 (\varepsilon_{n \bk} \leq \varepsilon) - \1 ( \varepsilon_{n \bk}^{L,q} \leq \varepsilon))\rd \bk.
\end{equation} From Lemma \ref{lem:vol_Sndelta}, we easily deduce that $|\cN(\varepsilon) - \cN^{L,q}(\varepsilon)| \le C
\delta^{L,q}_{\BZ}$. This is however a very crude approximation, since $\delta^{L,q}_{\BZ}$ only decays as $L^{-1}$ (and not as $L^{-(q+1)}$ as
wanted). This comes from the fact that $\varepsilon_{n \bk}$ is Lipschitz
but not $\cC^1$ on $\cB$.  However, according to Lemma~\ref{lem:coverBZ},
$\varepsilon_{n \bk}$ is analytic on $\widetilde \cS_n(\varepsilon, \delta^{L,q}_\BZ)$. Hence,
by setting
\begin{equation*}
\delta^{L,q}_{\cS} := \max_{n \le {\overline{M}}} \max_{\bk \in \widetilde{\cS}_n(\varepsilon_F, \delta_0)} |\varepsilon_{n \bk}^{L,q} - \varepsilon_{n \bk} |,
\end{equation*}
we first deduce from~\eqref{eq:f-Pif} that $\delta^{L,q}_\cS = O(L^{-(q+1)})$, then that the integrand in~\eqref{eq:errorcNs} is non-zero
only on $\widetilde{\cS}_n(\varepsilon, \delta^{L,q}_\cS)$, which again from Lemma~\ref{lem:vol_Sndelta} is of Lebesgue measure
$O(L^{-(q+1)})$. 
\end{proof}

For all $L \in \N$, the function $\cN^{L,q}$ is continuous and non-decreasing from $\cN^{L,q}(-\infty) = 0$ to $\cN^{L,q}(+\infty) = +\infty$. However, it is not necessarily increasing, and the non-empty set $(\cN^{L,q})^{-1}(\{N\})$ may
contain more than one point. Our results however will be independent of the choice of the
approximated Fermi level $\varepsilon_F^{L,q} \in (\cN^{L,q})^{-1}(\{N\})$. We state the next lemma in a very general setting.

\begin{lemma}[Error on the Fermi level]~\label{lem:error_epsF} Under Assumptions 1 and 2, there is $\delta_{1} > 0$
    such that, for all $0 < \delta \leq \delta_{1}$ and all continuous function $\widetilde{\cN}_\delta: \R \to \R$ satisfying
    \begin{equation}\label{eq:deltaprop}
    \max_{\eps \in [\varepsilon_F - \delta, \varepsilon_F +
        \delta]}\left|
    \cN(\varepsilon) - \widetilde{\cN}_\delta(\varepsilon) \right| \le
    \dfrac{\cD(\varepsilon_F) }{2} \delta,
    \end{equation}
    the equation $\widetilde{\cN}_\delta(\varepsilon) = N$ has at least one
    solution $\widetilde{\eps_{F}}$ in the range
    $[\varepsilon_F - \delta, \varepsilon_F + \delta]$, and any such
    solution satisfies
    \begin{equation} \label{eq:controlEps}
    \left| \varepsilon_F - \widetilde{\varepsilon_F} \right| \le
    \dfrac{2}{\cD(\varepsilon_F)} \max_{\varepsilon \in [\varepsilon_F - \delta,
        \varepsilon_F + \delta]} \left| {\cN}(\varepsilon) - \tilde{\cN}_\delta(\varepsilon) \right|
    \quad \left(  \le \delta \right).
    \end{equation}
    Together with Lemma~\ref{lem:errorN}, we deduce that there exists $C \in \R^+$ such that
    \[
    \forall L \in \N^*, \quad \left| \varepsilon_F - \varepsilon_F^{L,q}\right| \le \dfrac{C}{L^{q+1}}.
    \]
\end{lemma}
\begin{remark}[Failure of hypotheses]
    Lemma~\ref{lem:error_epsF} fails when $\cD(\varepsilon_F) = 0$, i.e.
    when the density of states is zero at the Fermi level (i.e. in
    semimetals such as graphene). In that case, the bound depends on the local behavior
    of $\cD$ around~$\fermi$.
\end{remark}

\begin{proof}[Proof of Lemma~\ref{lem:error_epsF}]
    Thanks to Lemma \ref{lemma:DOS}, there exists $\delta_{1} > 0$ such that
    $\inf_{\eps \in [\fermi-\delta_{1},\fermi+\delta_{1}]} \cD(\eps) > \cD(\fermi)/2$. In
    particular, for all   $\eps \in (\fermi-\delta_{1},\fermi+\delta_{1})$  and $0 \le
    \delta < \delta_1$ such that $[\eps - \delta , \eps + \delta] \subset [ \varepsilon_F -
    \delta_1, \varepsilon_F + \delta_1]$, we have
    \begin{equation}
    \label{eq:cNeps+delta}
    \cN(\varepsilon + \delta) = \cN(\varepsilon) + \int_{\varepsilon}^{\varepsilon +
        \delta} \cD(\varepsilon) \rd \varepsilon \ge \cN(\varepsilon) +
    \dfrac{\cD(\varepsilon_F)}{2} \delta,
    \end{equation}
    and similarly, $\cN(\varepsilon - \delta) \le \cN(\varepsilon) - \dfrac{\cD(\varepsilon_F)}{2} \delta$. Let now $\widetilde{\cN}_\delta: \R \to \R$ be a continuous 
    function satisfying~(\ref{eq:deltaprop}). Then, it holds that
    \begin{equation*}
    \widetilde \cN_\delta(\eps_{F}-\delta) \leq N \leq \widetilde \cN_\delta(\eps_{F}+\delta).
    \end{equation*}
    Hence, by continuity of $\widetilde{\cN}_\delta$, the equation $\widetilde{\cN}_\delta(\varepsilon) = N$ has at least one solution in
    $[ \varepsilon_F - \delta, \varepsilon_F + \delta]$.  Let $\widetilde{\varepsilon_F}$ be such a solution. We denote by
    $\kappa_{\cN,\widetilde{\cN}_\delta} := \max_{\varepsilon \in [\varepsilon_F - \delta, \varepsilon_F + \delta]} \left| \cN(\varepsilon) -
    \widetilde{\cN}_\delta(\varepsilon) \right|$. Since $\kappa_{\cN,\widetilde{\cN}_\delta} \le \frac{\cD(\varepsilon_F)}{2} \delta$, we can again
    use~\eqref{eq:cNeps+delta} and get
    \begin{equation*}
    N = \cN(\fermi) = \widetilde{\cN}_\delta(\widetilde{\varepsilon_F})  \le \cN(\widetilde{\varepsilon_F}) + \kappa_{\cN,\widetilde{\cN}_\delta} \le
    \cN \left( \widetilde{\varepsilon_F} +
    \dfrac{2}{{\cD}(\fermi)} \kappa_{\cN,\widetilde{\cN}_\delta} \right),
    \end{equation*}
    and, similarly,
    $N = \cN(\varepsilon_F) \ge \cN \left( \widetilde{\varepsilon_F} - \dfrac{2}{{\cD}(\fermi)} \kappa_{\cN,\widetilde{\cN}_\delta} \right)$.  From the fact that $\cN$ is non-decreasing, and the inequality
    \[
    \cN \left( \widetilde{\varepsilon_F} - \dfrac{2}{{\cD}(\fermi)} \kappa_{\cN,\widetilde{\cN}_\delta} \right) \le \cN(\varepsilon_F) \le \cN \left( \widetilde{\varepsilon_F} + \dfrac{2}{{\cD}(\fermi)} \kappa_{\cN,\widetilde{\cN}_\delta} \right),
    \]
    we obtain~\eqref{eq:controlEps}.
\end{proof}

\subsection{Error on the ground state energy and density}

We now focus on the calculations of the ground state energy~\eqref{eq:IDoE} and density~\eqref{eq:totalDensity}. Let $\Pi^{L,p}$ and
$\Pi^{L,q}$ be interpolation operators of order $(p+1)$ and $(q+1)$ respectively. For the total energy, we introduce
\begin{equation*}
\varepsilon_{n \bk}^{L,p} := \Pi^{L,p} \left( \varepsilon_{n \bk} \right) \quad \text{and} \quad \varepsilon^{L,q}_{n \bk} := \Pi^{L,q} \left( \varepsilon_{n \bk} \right).
\end{equation*}
Using two different interpolation operators allows us to identify the error coming from the inexact approximation of $\varepsilon_{n\bk}$ everywhere in the Brillouin zone (\textit{bulk error}) and the one coming
from the inexact calculation of the Fermi energy (\textit{surface error}). We assume that the Fermi level is approximated by $\varepsilon_F^{L,q}$ as in the previous section,
using the same interpolation operator $\Pi^{L,q}$.  Altogether, we  compare the ground state energy $E = E(\varepsilon_F)$ defined in~\eqref{eq:IDoE} with the approximate ground state energy
\begin{equation} \label{eq:Lpq}
E^{L,p,q} :=  \sum_{n \in \N^\ast} \fint_{\BZ} \varepsilon_{n \bk}^{L,p} \1(\varepsilon_{n \bk}^{L,q} \le \varepsilon_F^{L,q}) \rd \bk,
\end{equation}
and the ground state electronic density $\rho=\rho_{\varepsilon_F}$, where $\rho_\eps$ is defined in~\eqref{eq:totalDensity}, with the approximate ground state density
\begin{equation*}
\rho^{L,p,q}(\br):= \sum_{n \in \N^\ast} \fint_{\BZ}  \Pi^{L,p} \left( |u_{n \bk}(\br)|^2\right) \1(\varepsilon_{n \bk}^{L,q} \le \varepsilon_F^{L,q}) \rd \bk.
\end{equation*}

\medskip

The main theorem of this section is the following.

\begin{theorem}
\label{thm:err_A} Assume $V \in H^s_\per$ for some $s \ge 0$.
Under Assumptions 1 and 2, there exists $C \in \R_{+}$ such that, for all $L \in \N$,
\begin{align*}
\left\| \rho - \rho^{L,p,q} \right\|_{H^{s+2}_\per} &\le C \left( \dfrac{1}{L^{p+1}}  +
\dfrac{1}{L^{q+1}}\right),\\
\left| E - E^{L,p,q} \right| &\le C \left( \dfrac{1}{L^{p+1}}  +  \dfrac{1}{L^{2q+2}} \right).
\end{align*}
\end{theorem}

\begin{proof}
We start with the density. Let $W \in H^{-(s+2)}_\per$ and introduce
\begin{equation*}
W_{n\bk} := \left\langle W, |u_{n\bk}|^2\right\rangle_{H^{-(s+2)}_\per, H^{s+2}_\per}  \quad \text{and} \quad
W_{n\bk}^{L, p} := \Pi^{L,p} \left( W_{n\bk} \right),
\end{equation*}
so that the error is $ \displaystyle \left\| \rho - \rho^{L,p,q} \right\|_{H^{s+2}_\per} = \sup_{W \in H^{-(s+2)}_\per, \| W \|_{H^{-(s+2)}_\per} = 1} e^{L, p, q}(W)$, where we set
\begin{equation*}
e^{L, p, q}(W) := \left\langle W, 
\rho - \rho^{L, p, q} \right\rangle_{H^{-(s+2)}_\per, H^{s+2}_\per} = \sum_{n \in \N^*} \fint_{\BZ} \left( W_{n \bk}  \1(\varepsilon_{n \bk} \le
\varepsilon_F) - W_{n \bk}^{L, p}   \1(\varepsilon_{n \bk}^{L,q} \le
\varepsilon_F^{L,q}) \right) \rd \bk.
\end{equation*}

We decompose the error into two contributions: the bulk error and the surface error. We write
$e^{L,p,q}(W) = e^{L,p}_{\rm bulk}(W) + e^{L,p,q}_{\rm surf}(W)$ with
\begin{equation} \label{eq:bulkError}
e_{\rm bulk}^L(W) := \sum_{n \le {\overline{M}}} \fint_{\BZ}
\left( W_{n \bk} - W_{n \bk}^{L,p} \right) \1 (\varepsilon_{n \bk} \le
\varepsilon_F) \rd \bk
\end{equation}
and
\begin{equation} \label{eq:surfaceError}
e^{L,p,q}_{\rm surf}(W) := \sum_{n \le {\overline{M}}} \fint_{\BZ} W_{n \bk}^{L,p} \left[
\1(\varepsilon_{n \bk} \le \varepsilon_F) - \1(\varepsilon_{n \bk}^{L,q} \le
\varepsilon_F^{L,q}) \right] \rd \bk .
\end{equation}
The bulk error \eqref{eq:bulkError} is spread over the whole Brillouin zone, while the surface error \eqref{eq:surfaceError} is localized
near the Fermi surface $\cS$. In order to control these two terms, we use Lemma~\ref{lem:coverBZ} which shows that
$\bk \mapsto \sum_{m=1}^n W_{m \bk} = \left\langle W ,\rho_{n \bk}\right\rangle_{H^{-(s+2)}_\per, H^{s+2}_\per}$ is smooth on $\bk \mapsto
\widetilde{\cB}_n(\varepsilon_F, \delta_0)$, while the map
$\bk \mapsto \varepsilon_{n \bk}$ is smooth on $\widetilde{\cS_n}(\varepsilon_F, \delta_0)$.

\paragraph{{\bf Bulk error}}  We have
\begin{align*}
e^{L,p}_{\rm bulk}(W) & =
\frac{1}{|\BZ|}\sum_{n \le {\overline{M}}}\int_{\BZ_{n}(\varepsilon_F)}\sum_{m=1}^{n}(W_{m \bk} - W_{m \bk}^{L,p})
\rd \bk  \\
& =
\frac 1 {|\BZ|}\sum_{n \le {\overline{M}}} \int_{\BZ_{n}(\varepsilon_F)}  \left( \left\langle W ,\rho_{n \bk}\right\rangle_{H^{-(s+2)}_\per, H^{s+2}_\per} - \Pi^{L,p} \left[ \left\langle W ,\rho_{n \bk}\right\rangle_{H^{-(s+2)}_\per, H^{s+2}_\per} \right] \right) \rd \bk.
\end{align*}
\REV{Let us introduce the maps 
$$
F_{n,W}:\left\{
\begin{array}{ccc}
\BZ & \to & \mathbb{R}\\
\bk &\mapsto &\left\langle W ,\rho_{n \bk}\right\rangle_{H^{-(s+2)}_\per, H^{s+2}_\per}\\
\end{array}
\right. .
$$
}
According to Lemma~\ref{lem:coverBZ}, the maps \REV{\old{$F_{n,W}:\bk \mapsto \left\langle W ,\rho_{n \bk}\right\rangle_{H^{-(s+2)}_\per, H^{s+2}_\per}$} $F_{n,W}$} are analytic on $\widetilde{\cB}_n(\varepsilon_F, \delta_0/2)$, and it holds that
\[
\forall n \le \overline{M}, \ \forall \bk \in \widetilde{\cB}_n(\varepsilon_F, \delta_0/2), \quad
\left|  F_{n,W}^{(p+1)} \right| \le \| W \|_{H^{-(s+2)}_\per} \sup_{\bk \in \widetilde{\cB}_n(\varepsilon_F, \delta_0/2)} \left\| \partial_{\bk}^{(p+1)} \rho_{n \bk} \right\|_{H^{s+2}_\per}.
\] 
Together with~\eqref{eq:f-Pif}, we deduce that there exists $C \in \R_+$ such that
\begin{equation*}
\left| e^{L,p}_{\rm bulk}(W) \right| \le \dfrac{C}{L^{p+1}}  \| W \|_{H^{-(s+2)}_\per}.
\end{equation*}

\paragraph{\bf Surface error.}
For the integrand in \eqref{eq:surfaceError} to be non-zero, it must hold that
\begin{equation*}
\varepsilon_{n\bk} \leq
\varepsilon_F \textrm{ and } \varepsilon_{n
\bk}^{L,q} > \varepsilon_F^{L,q} \quad \text{or} \quad \varepsilon_{n\bk} >
\varepsilon_F \textrm{ and } \varepsilon_{n
\bk}^{L,q} \leq \varepsilon_F^{L,q}.
\end{equation*}
In the former case for instance, we have
\begin{equation*}
0 \le \varepsilon_F - \varepsilon_{n \bk} = \left( \varepsilon_F - \varepsilon_F^{L,q} \right)
+ \left( \varepsilon_F^{L,q} - \varepsilon_{n \bk}^{L,q} \right) + \left( \varepsilon_{n \bk}^{L,q}
- \varepsilon_F \right).
\end{equation*}
The middle term being negative, we deduce that
$\left| \varepsilon_F - \varepsilon_{n \bk} \right| \le \left| \varepsilon_F - \varepsilon_F^L \right| + \left| \varepsilon_{n \bk}^L -
\varepsilon_{n \bk} \right|$. The other case is similar. In particular, as in the proof of Lemma~\ref{lem:error_epsF}, for $L$ large
enough, we can first restrict the integral in~\eqref{eq:surfaceError} to $\widetilde \cS_{n}(\varepsilon_F, \delta_0/2)$, then to some
$\widetilde \cS_n (\varepsilon_F, C L^{-(q+1)})$. Finally, since the maps
$\bk \mapsto W_{n\bk} = \left\langle W ,\rho_{n+1, \bk} - \rho_{n\bk}  \right\rangle_{H^{-(s+2)}_\per, H^{s+2}_\per}$ are smooth on
$\widetilde \cS_{n}(\varepsilon_F, \delta_0)$, we deduce that there exists $C_q \in \R_+$ such that
\begin{equation*}
\left| e^{L,p,q}_{\rm surf}(W) \right| \le \dfrac{C_q}{L^{(q+1)}} \| W \|_{H^{-(s+2)}_\per},
\end{equation*}
and the result follows.

\REV{\begin{remark}
  As we see from the proof, the surface error behaves as $L^{-q-1}$, while the bulk error behaves as $L^{-p-1}$. For the case of insulators (no Fermi surface), the surface error vanishes, and the bulk error can be exponentially small with good choice of interpolants. However, for metallic systems, we believe that our estimates are optimal, hence much larger due to Gibbs oscillations. This is illustrated in our numerical simulations in Section~\ref{sec:resnum}.
\end{remark}}

\paragraph{\bf Case of the energy.} We now focus on the energy. We follow the same lines as above, and decompose the error into a bulk error and
a surface error. The bulk error is bounded as above. We focus on the surface error, which reads
\begin{align*}
e^{L,p,q}_{\rm surf} & := \sum_{n \le {\overline{M}}} \fint_{\BZ} \eps_{n \bk}^{L,p}
\left[	\1(\varepsilon_{n \bk} \le
\varepsilon_F) - \1(\varepsilon_{n \bk}^{L,q}
\le \varepsilon_F^{L,q}) \right] \rd \bk\\
&= \sum_{n \le {\overline{M}}} \fint_{\BZ} (\eps_{n \bk}^{L,p} - \eps_{F})
\left[	\1(\varepsilon_{n \bk} \le
\varepsilon_F) - \1(\varepsilon_{n \bk}^{L,q}
\le \varepsilon_F^{L,q}) \right] \rd \bk +
\underbrace{\eps_{F} (\mathcal N(\eps_{F}) - \mathcal
N^{L,q}(\eps_{F}^{L,q}))}_{= 0}.
\end{align*}
Again, the integrand of $e^{L,p,q}_{\rm surf}$ is supported on sets of the form $\widetilde{\cS}_{n}(\varepsilon_F, C L^{-(q+1)})$. On
these sets, it holds that
\begin{equation*}
\left| \varepsilon_{n \bk}^{L,p} - \varepsilon_F \right| \le |\eps_{n \bk}^{L,p} -
\eps_{n \bk}| + |\eps_{n \bk} - \eps_{F}| = O(L^{-(p+1)} + L^{-(q+1)}).
\end{equation*}
We easily deduce that
\begin{align*}
\left| e^{L,p,q}_{\rm surf} \right| & \leq \frac 1 {|\BZ|}\sum_{n \le {\overline{M}}}
\int_{\widetilde \cS_{n}(\varepsilon_F, C
L^{-(q+1)})} C  \left(
L^{-(p+1)} + L^{-(q+1)} \right) \rd\bk \leq C \left(\frac 1 {L^{p+q+2}} + \frac 1 {L^{2q+2}}\right),
\end{align*}
and the proof follows.
\end{proof}

\begin{remark}[Order gain on the energy] \label{rem:numericalIntegration} The interest of choosing two different interpolation operators $\Pi^{L,p}$ and $\Pi^{L,q}$ for the calculation of $E^{L,p,q}$ is
now clear: the integration zone $\1(\varepsilon_{n \bk}^{L,q} \le \varepsilon_F^{L,q})$ can be approximated with a
lower order term (i.e. $q = \lceil \frac{p-1}2 \rceil$) with no loss of order. For instance, when using a cubic method ($p = 3$), it is enough to evaluate the integral of
cubic functions on tetrahedra ($q = 1$), and not on
complicated intersections of cubic surfaces ($q=3$). 

This is surprising at first: since the computation of the approximate
energy $E^{L,p,q}$ involves the approximate Fermi level
$\varepsilon_{F}^{L,q}$, how can the energy be more accurate than the
Fermi level? The answer, as shown above, is that, to leading order,
the variations of the energy caused by an error in the determination
of the Fermi surface is proportional to $\varepsilon_{F}$ times the
error on the number of the particles. Since this number is kept fixed
to $N$ for all $L$, this leading order error vanishes. This means
that, even if the exact Fermi level is known, it is still numerically
advantageous to keep it determined implicity through the equation
$\cN^{L,q}(\varepsilon_{F}^{L,q}) = N$.
\end{remark}



\section{Smearing methods}
\label{sec:smearing-methods}

We now focus on smearing methods. Let $A$ denote either the integrated
density of states $\cN$, the ground state density $\rho$ or the ground
state energy $E$. We want to approximate $A$ by $A^L$, where $A^L$ is
obtained by replacing the integral
in~\eqref{eq:IDoS}-\eqref{eq:IDoE}-\eqref{eq:totalDensity} by a
corresponding Riemann sum on a regular grid with $L^d$ points.
However, since the step function $f(x) := \1(x \le 0)$ appearing in
the integrand is discontinuous, we expect the convergence to be slow.
The idea of smearing methods is to replace this step function by a
\textit{smeared} function $f^{\sigma}$ that is smooth: we define
\begin{align}
  \label{eq:def_Asigma}
  A^{\sigma}(\varepsilon) = \fint_{\BZ} \sum_{n \in \N^{*}} A_{n\bk} f^{\sigma}(\varepsilon_{n \bk} - \varepsilon) \rd \bk,
\end{align}
where $f^{\sigma}$ is a smooth approximation to $f$, as we will
discuss below. This approximate quantity $A^\sigma$ can then be efficiently
computed by a Riemann sum. We introduce, for $L \in \N^*$, the uniform
grid
\begin{equation*}
\BZ_L := \cB \cap L^{-1} \RLat,
\end{equation*}
where we see here $\cB$ as a torus, so that there are $L^d$
points in $\BZ_L$. We then define
\begin{align}
  A^{\sigma,L}(\varepsilon) & := \dfrac{1}{L^d} \sum_{\bk \in \BZ_L} \sum_{n \in \N^*} A_{n \bk} f^{\sigma}  (\varepsilon_{n\bk} - \varepsilon).
\end{align}
We define $\varepsilon_{F}^{\sigma}$ and $\varepsilon_{F}^{\sigma,L}$ to be the (a priori non-unique)
solutions of the equations
\begin{align}
  \label{eq:def_eps_F_sigma_L}
  \cN^{\sigma}(\varepsilon_{F}^{\sigma}) = N, \quad \cN^{\sigma,L}(\varepsilon_{F}^{\sigma,L}) = N,
\end{align}
and we finally set
\begin{align}
  A^{\sigma} := A^{\sigma}(\varepsilon_{F}^{\sigma}), \quad A^{\sigma,L} := A^{\sigma,L}(\varepsilon_{F}^{\sigma,L}).
\end{align}
The quantities $A^{\sigma,L}$ are the ones that we can compute numerically.
Our goal is to compute the error between $A^{\sigma,L}$ and $A$. We
first estimate the error between ${A}^\sigma$ and $ A$ in
Section~\ref{sec:error-between-exact}. Then, we provide in
Section~\ref{sec:error-betw-smear} error estimates for the
discretization error ${A}^{\sigma,L} - {A}^{\sigma}$. The combination
of the two provides the total error estimates for smearing methods.

\subsection{Smearing functions}
\label{sec:smearing-functions}

In this section, we explain how smearing functions are constructed. 

\REV{
\begin{definition}[Smearing mollifier]
\label{def:smearing}
 We say a function $\delta^1: \R \to \R$ is a smearing mollifier if it satisfies the two following properties:
\begin{enumerate}
	\item[(P1)] $\delta^1 \in \cS(\R)$, where $\cS(\R)$ denotes the Schwartz space of fast decaying functions;
	\item[(P2)] $\int_{\R} \delta^1 = 1$.
\end{enumerate}
We say that a smearing mollifier is of order at least $p \in \N$ if
\begin{equation*}
\int_{\R} P(x) \delta^1(x) \rd x = P(0) \quad \textrm{for all polynomials $P$ with $\deg(P) \le p$}.
\end{equation*}
that is $M_{0}(\delta^{1})=1$ and $M_n(\delta^1)=0$ for $1 \le n \le p$,
where the $M_n(\phi)$ is the $n$-th momentum of the function $\phi \in \cS(\R)$:
$$
\forall n \in \N, \quad M_n(\phi) = \int_\R x^n \phi(x) \, \rd x.
$$
The \textit{order} of a smearing method is the \REV{\old{smallest}largest} $p$ such that
the above property holds.
We say that $f^1: \R \to \R$  is a smearing function if there exists a smearing mollifier $\delta^1$ such that 
$$
f^1(x) = \int_{-\infty}^x \delta^1(y)\,dy.
$$
\end{definition}
}


\REV{For any smearing mollifier $\delta^1$, }we set $\delta^\sigma (x) := \sigma^{-1} \delta^1(\sigma^{-1} x)$
and
\begin{equation} \label{eq:def:fsigma}
f^\sigma (x) = 1 - \int_{-\infty}^{x} \delta^\sigma(y) \rd y, \quad \text{so that} \quad f^\sigma(x) = f^1(\sigma^{-1}x).
\end{equation}
Note that $\delta^\sigma = -(f^\sigma)'$, and that we have in a distributional sense $\delta^{\sigma} \to \delta$ and $f^{\sigma} \to f$ as $\sigma \to 0$.

\medskip


The true step function $f$ is non-increasing (which implies that the set of possible Fermi
levels $\mathcal N^{-1}(\{N\})$ is an interval), and has values equal to either~0 or~1. In particular, $f_{n\bk} = f(\varepsilon_{n\bk} - \fermi)$ is interpreted as the occupation number of the Bloch modes with energy $\varepsilon_{n\bk}$. By
contrast, this interpretation is not valid for a smearing method of order $p \ge 2$, since smearing functions of order $p \ge 2$ necessarily have values outside the
range $[0,1]$ (otherwise, $\int(f^{1}-f)x$ would be positive).

%
\medskip

Let us mention some possible choices encountered in the literature and used in practice:
\begin{itemize}
	\item the Fermi-Dirac smearing~\cite{Dirac1926, Fermi1926, Mermin1965}:
	\begin{equation}
	\label{smearing_function_1}
	f_{\rm FD}^1(x):= \frac{1}{1 + \re^{x}}, \quad  \delta_{\rm FD}^1(x) := \dfrac{1}{2 + \re^{x} + \re^{-x}}.
	\end{equation}
	This method is of order $1$ and $f^1_{\rm FD}$ is
	decreasing from $1$ to $0$;

	\item the Gaussian smearing~\cite{DeVita1991}:
	\begin{equation}
	\tag{\ref*{smearing_function_1}$'$}
	f^1_{\rm G}(x):= \frac{1}{2}\left( 1 - {\rm erf}(x)\right), \quad   \delta_{\rm G}^1(x) := \dfrac{1}{\sqrt{\pi}} \re^{-x^2}.
	\end{equation}
	This method is of order $1$ and $f^1_{\rm G}$ is decreasing from $1$ to $0$;

	\item the Methfessel-Paxton smearing~\cite{methfessel1989high}: this method is defined by the sequence of
	functions $(f_{\rm MP,N}^1)_{N \in \N}$ given by
	\begin{equation}
	\tag{\ref*{smearing_function_1}$''$}
	f_{{\rm MP},N}^1(x):=  f^1_{\rm G}(x) + \sum_{n=1}^N A_n H_{2n-1}(x)
	\re^{-x^2},
	\quad
	\delta_{{\rm MP},N}^1(x) := \sum_{n=0}^N A_n H_{2n}(x) \re^{-x^2}.
	\end{equation}
	Here, the functions $(H_n)_{n \in \N}$ are the Hermite
        polynomials (defined as $H_{0}(x)=1$,
        $H_{n+1}(x) = 2x H_{n}(x)-H_{n}'(x)$), and the coefficients
        $A_n := \dfrac{(-1)^n}{n! 4^n \sqrt{\pi}}$ are chosen such
        that the method is of order $2N+1$. For $N \ge 1$,
        $f_{{\rm MP}, N}^1$ is not monotone, and has negative
        occupation numbers;
	\item the Marzari-Vanderbilt cold smearing~\cite{Marzari1999}:
	\begin{equation}
	\tag{\ref*{smearing_function_1}$'''$}
	f_{\rm cs}^1(x) := f^1_{\rm G}(x) + \dfrac{1}{4 \sqrt{\pi}} ( - a H_2(x) +
	H_1(x) ) \re^{-x^2},
	\end{equation}
	corresponding to
	\begin{equation*}
	\delta^1_{\rm cs}(x) = \dfrac{1}{\sqrt{\pi}} \left( ax^3 - x^2 - \frac32 a x + \frac32 \right) \re^{-x^2},
	\end{equation*}
	where $a$ is a free parameter, usually chosen so that
        $f_{\rm cs}^{1}$ is always non-negative (avoiding negative
        occupation numbers). This method, like the $N=1$ case of the
        Methfessel-Paxton scheme above, is of order $3$.
\end{itemize}

\begin{figure}[h!]
\centering
\includegraphics[width=0.49\textwidth]{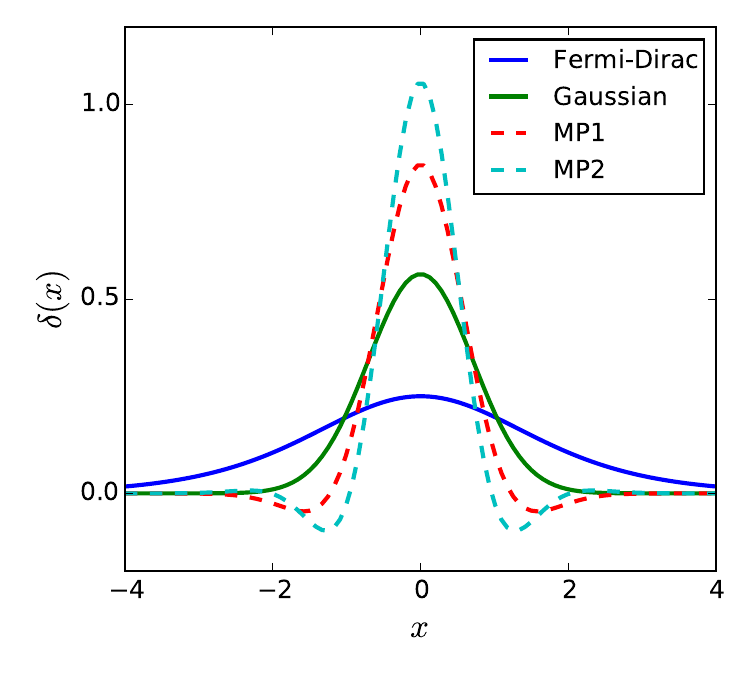}  \includegraphics[width=0.49\textwidth]{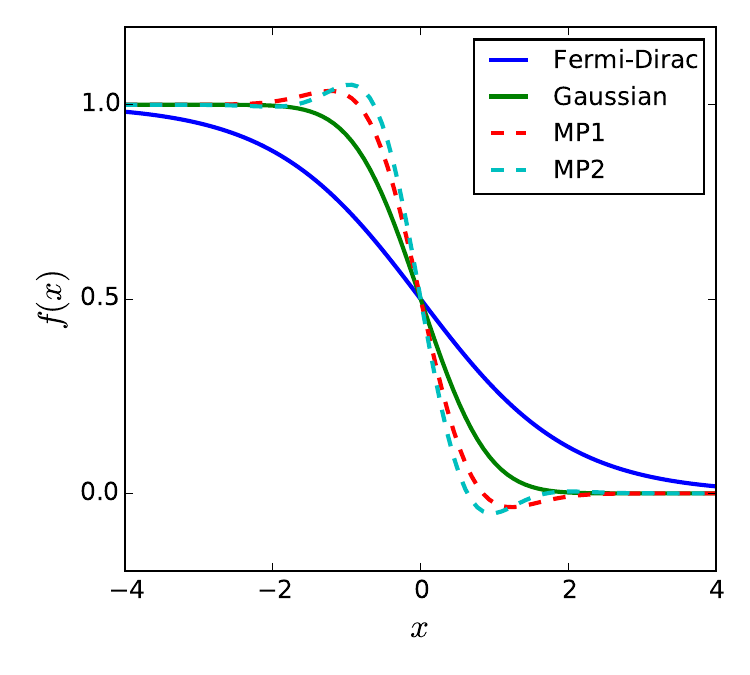}
\caption{Some smearing functions. Approximation to the Dirac
function $f^{1}$ (left), and occupation numbers $\delta^{1}$ (right).}
\label{fig:smearing_functions}
\end{figure}

\begin{remark}[Temperature]
The Fermi-Dirac distribution is used to model electronic systems at
a finite temperature. In this case $\sigma = k_{B} T$, where $k_{B}$
is the Boltzmann constant. The other smearing functions do not have
such a physical interpretation and are only chosen for their mathematical properties.
  \end{remark}

  The Fermi-Dirac function is \REV{\old{holomorphic}meromorphic}, but has poles at
  the imaginary energies $(2 \Z + 1)\ri \pi$ (called \textit{Matsubara
  frequencies} in the context of field theory). By contrast, the other
  smearing functions introduced above (called \textit{Gaussian-type} in
  the following) are entire. This will have an impact on our estimates.

\subsection{Error between exact and smeared quantities}
\label{sec:error-between-exact}

In the remaining of this section, we fix a smearing \REV{\old{method}mollifier $\delta^1$} of
order $p$, and we are interested in the behavior of $A^{\sigma} - A$
as $\sigma$ goes to $0$. 

\medskip

Consider a quantity of interest of the form
\begin{equation} \label{eq:def:Aeps}
    A(\varepsilon)= \sum_{n \in \N^{*}}\fint_{\BZ} A_{n\bk} \1(\varepsilon_{n\bk} \le \varepsilon) \, \rd \bk.
\end{equation}
Then, formally (we will justify this computation case by case for
various $A$ later)
\begin{align}
  \sum_{n \in \N^{*}}\fint_{\BZ} A_{n\bk} f^1\left(\frac{\varepsilon_{n\bk}-\varepsilon}\sigma\right) \, \rd \bk =& \sum_{n \in \N^{*}}\fint_{\BZ} A_{n\bk} \left( \int_{\frac{\varepsilon_{n\bk}-\varepsilon}\sigma}^{\infty}\delta^1(x) \, \rd x \right) \, \rd\bk \nonumber  \\ 
  =& \frac 1 \sigma \sum_{n \in \N^{*}}\fint_{\BZ} A_{n\bk} \left( \int_{\varepsilon_{n\bk}}^{\infty}\delta^1\left(\frac{\varepsilon'-\varepsilon}\sigma\right) \, \rd\varepsilon' \right) \,  \rd\bk \nonumber  \\
  =& \frac 1 \sigma \int_{\R} \left( \sum_{n \in \N^{*}}\fint_{\BZ} A_{n\bk} \1(\varepsilon_{n\bk} \leq \varepsilon) \, \rd\bk \, \right) \delta^1\left(\frac{\varepsilon'-\varepsilon}\sigma\right) \, \rd\varepsilon' \nonumber  \\ 
  =& \frac 1 \sigma \int_{\R}  \, A(\varepsilon) \delta^1\left(\frac{\varepsilon'-\varepsilon}\sigma\right)  \, \rd\varepsilon' =  (A \ast \delta^{\sigma})(\varepsilon). \label{eq:Amusigma}
\end{align}
In other words, the effect of smearing is to smooth the
function $A(\varepsilon)$ by a convolution with $\delta^{\sigma}$. In order to understand the properties of the
smearing method, we therefore have to study the asymptotic behavior of
integrals of the form~\eqref{eq:Amusigma} for $\sigma \to 0$.

\medskip

To make this precise, we introduce the mollification operator. \REV{Let us denote by $\cS'(\R)$ the set of tempered distributions on $\R$}. For $g
\in \cS'(\R)$ and $\phi \in \cS(\R)$, we define (in the sequel, $g \sim A$, and $\phi \sim \delta^1$ is a mollifier)
\begin{equation} \label{eq:def:cMAdelta}
   \forall \varepsilon \in \R, \ \forall \sigma \in \R^*, \quad
    \cM_{g,\phi}(\varepsilon,\sigma) := \left\langle g,\frac 1 \sigma \phi\left(\frac{\cdot - \varepsilon}{\sigma}\right) \right\rangle_{\cS',\cS}
   \quad \text{and} \quad
   \cM_{g,\phi}(\varepsilon, 0) := g(\varepsilon)M_0(\phi).
\end{equation}
Note that we extended $\cM$ to $\sigma$ negative. Due to the change of variables in~\eqref{eq:Amusigma}, this does not correspond to taking a negative smearing parameter in the original definition~\eqref{eq:def_Asigma} of $A^{\sigma}$. The main idea of this section is that if $g$ is smooth, then we can write that
\[
    \cM_{g,\phi}(\varepsilon,\sigma)  = \frac 1 \sigma \int_\R g(\varepsilon) \, \phi\left(\frac{\varepsilon'-\varepsilon}\sigma \right) \,  \rd \varepsilon' 
    = \int_\R g(\varepsilon + \sigma x) \, \phi(x) \,  \rd x.
\]
In particular, if $\phi$ is a smearing \REV{\old{method}mollifier} of order $p$, then by Taylor-expanding $g$
around $\varepsilon$, this quantity is also $g(\varepsilon) + O(\sigma^{p+1})$. We make this statement rigorous in the next Lemma, whose proof is postponed until the end of the section. 

\begin{lemma}\label{lem:propg} Let $g \in \cS'(\R)$ be such that $g|_{(\varepsilon_F-\delta_0,\varepsilon_F+\delta_0)}$ is a function of class $C^k$, and let $\phi \in \cS(\R)$. Then, the function $\cM_{g,\phi}(\varepsilon,\sigma)$ is of class $C^k$
  on $(\varepsilon_F-\delta_0,\varepsilon_F+\delta_0) \times \R$, and
  we have, for all
  $(m,n) \in \N \times \N \mbox{ such that } m+n \le k,$
$$
\frac{\partial^{m+n}\cM_{g,\phi}}{\partial \varepsilon^m \partial
  \sigma^n}(\varepsilon,\sigma)
= \left\{  \begin{array}{ll} \dps \dfrac{1}{\sigma^{n+1}}  \left\langle g,\phi_{m,n}\left(\frac{\cdot-\varepsilon}{\sigma}\right) \right\rangle_{\cS',\cS} & \mbox{ if } \sigma \neq 0, \\
           \dps g^{(m+n)}(\varepsilon) M_n(\phi) & \mbox{ if } \sigma =
                                           0, \end{array} \right.
$$
where $\phi_{m,n} \in \cS(\R)$ is defined by
$$
\forall t \in \R, \quad \phi_{m,n}(t) := (-1)^{m+n}
\frac{\rd^{m+n}}{\rd t^{m+n}} \left( t^n \phi(t) \right).
$$
\end{lemma}

\subsubsection{Error on the integrated density of states and on the Fermi level} 

By choosing $A_{n \bk} = 1$ in \eqref{eq:def:Aeps} and using the decay at
infinity of $\delta^{1}$ to justify the exchange of integrals in~\eqref{eq:Amusigma}, we have
$\cN^{\sigma}(\varepsilon) = \widetilde{\cN}(\varepsilon, \sigma)$, where we set for clarity $\widetilde{\cN} := \cM_{\cN,\delta^1}$. 
\begin{lemma} \label{lem:Nsigma}
    For any smearing \REV{\old{method}mollifier $\delta^1$} of order $p \ge 1$, the function $\widetilde{\cN}$ is smooth on $(\varepsilon_F - \delta_0, \varepsilon_F + \delta_0) \times \R$, and satisfies 
    \begin{equation*}
    \frac{\partial \widetilde\cN}{\partial\varepsilon}(\varepsilon,0) = \cD(\varepsilon) >0, \quad  
    \forall 1 \le n \le p, \ \frac{\partial^n \widetilde\cN}{\partial^n \sigma}(\varepsilon,0) =   0,
    \quad
    \frac{\partial^{p+1} \widetilde\cN}{\partial^{p+1} \sigma}(\varepsilon,0) = \cD^{(p)}(\varepsilon) M_{p+1}(\delta^1).
    \end{equation*}
    In particular, there exists $C \in \R^+$ such that
    \begin{equation} \label{eq:N-Nsigma}
    \max_{\varepsilon \in (\varepsilon_F - \delta_0, \varepsilon_F + \delta_0)} \left| \cN(\varepsilon) - \cN^\sigma(\varepsilon) \right| \le C \sigma^{p+1}.
    \end{equation}
\end{lemma}

\begin{proof} Applying Lemma~\ref{lem:propg} with $g=\cN$, which is smooth on $(\varepsilon_F-\delta_0,\varepsilon_F+\delta_0)$, and $\phi=\delta^1$, we obtain that $\widetilde\cN$ is smooth on $(\varepsilon_F-\delta_0,\varepsilon_F+\delta_0) \times\R$. The proof is then a consequence of the fact that $M_0(\delta^1)=1$ and $M_n(\delta^1)=0$ for $1 \le n \le p$, together with the fact that $\cN'(\varepsilon) = \cD(\varepsilon)$ by definition (see~\eqref{eq:def:Deps}).
\end{proof}

From~\eqref{eq:def_eps_F_sigma_L}, $\varepsilon_F^\sigma$ is solution to $\widetilde{\cN}(\varepsilon_F^\sigma,\sigma) = N$. From the previous proposition together with the implicit function theorem we directly get the following result.

\begin{lemma}\label{lem:uniqueFL} For any smearing \REV{\old{method}mollifier $\delta^1$} of order $p \ge 1$, there exists $\sigma_1, \delta_{1} > 0$, such that for  $|\sigma| < \sigma_1$, the equation
  $\widetilde{\cN}(\cdot,\sigma)=N$ has a unique
  solution $\varepsilon_F^\sigma$ in
  $(\varepsilon_{F}-\delta_{1},\varepsilon_{F}+\delta_{1})$. In
  addition, the function
  $(-\sigma_1,\sigma_1) \ni \sigma \mapsto \varepsilon_F^\sigma \in
  \R$ is smooth, and it holds that
\begin{equation*}
\varepsilon_F^{\sigma = 0} =\varepsilon_F, \quad \forall 1 \le n \le p, \quad \left. \frac{\rd^n\varepsilon_F^\sigma}{\rd \sigma^n}\right|_{\sigma=0} = 0, \quad 
\left. \frac{\rd^{p+1}\varepsilon_F^\sigma}{\rd \sigma^{p+1}}\right|_{\sigma=0} = -\frac{\cD^{(p)}(\varepsilon_F)}{\cD(\varepsilon_F)} M_{p+1}(\delta^1).
\end{equation*}
In particular, it holds that
\begin{equation*} 
   \varepsilon_F^\sigma = \varepsilon_F + \dfrac{1}{(p+1)!} \left( -\frac{\cD^{(p)}(\varepsilon_F)}{\cD(\varepsilon_F)} M_{p+1}(\delta^1) \right) \sigma^{p+1} + O(\sigma^{p+2}).
\end{equation*}
\end{lemma}

\subsubsection{Error on the ground-state energy}

By choosing $A_{n \bk} = \varepsilon_{n \bk}$ in \eqref{eq:Amusigma}, we have that
$E^{\sigma}(\varepsilon) = \widetilde{E}(\varepsilon, \sigma)$, where we set for clarity $\widetilde{E} := \cM_{E,\delta^1}$. Numerically, the true ground state energy $E$ is approximated by $E^\sigma := \widetilde{E}(\varepsilon_F^\sigma, \sigma)$. 
\begin{lemma} \label{lem:smearedE}
For any smearing \REV{\old{method}mollifier $\delta^1$} of order $p \ge 1$, there exists $\sigma_1 > 0$ such that the function $\sigma \mapsto \widetilde{E}$ is well-defined and smooth on $(-\sigma_1,\sigma_1) \times \R$, and satisfies
\begin{equation*}
    \dfrac{\partial \widetilde{E}}{\partial \varepsilon}(\varepsilon, 0)= \varepsilon \cD(\varepsilon), \quad  
    \forall 1 \le n \le p, \ \dfrac{\partial^n \widetilde{E}}{\partial \sigma^n}(\varepsilon, 0) = 0,
    \quad
     \dfrac{\partial^{p+1} \widetilde{E}}{\partial \sigma^{p+1}}(\varepsilon, 0) = \left( \varepsilon \cD^{(p)}(\varepsilon) + p \cD^{(p-1)}(\varepsilon)\right) M_{p+1}(\delta^1)
\end{equation*}
In particular, it holds that  
\begin{align*}
E^{\sigma = 0}=E_0, \quad \forall 1 \le n \le p, \quad \left. \frac{\rd^n E^\sigma}{\rd \sigma^n}\right|_{\sigma=0} = 0, \quad 
\left. \frac{\rd^{p+1}E^\sigma}{\rd \sigma^{p+1}}\right|_{\sigma=0} = p \cD^{(p-1)}(\varepsilon_F) M_{p+1}(\delta^1) .
\end{align*}
Finally, we have
\begin{equation} \label{eq:estimationEsigma}
    E^\sigma = E + \dfrac{1}{(p+1)!} \left( p \cD^{(p-1)}(\varepsilon_F) M_{p+1}(\delta^1) \right) \sigma^{p+1} + O(\sigma^{p+2}).
\end{equation}
\end{lemma}

\begin{proof} The first part is a consequence of Lemma~\ref{lem:propg} applied to $g=E(\cdot)$ (which is smooth on $(\varepsilon_F-\delta_0,\varepsilon_F+\delta_0)$) and $\phi=\delta^1$, together with~\eqref{eq:DoE}. The second part comes from Lemma~\ref{lem:uniqueFL}, and the chain rule.
\end{proof}

\ul{\bf Extrapolation of the energy.}
    Following Marzari~\cite{ThesisMarzari}, we can introduce the entropy
    defined as
    \begin{equation} \label{eq:def:Ssigma}
        S^\sigma = \sum_{n \in \N^\ast} \fint_\cB s\left(
        \frac{\varepsilon_{n\bk}-\varepsilon_F^\sigma}{\sigma} \right) \,
        \rd \bk,
    \end{equation}
    where $s : \R \to \R$ is the unique function of $\cS(\R)$ such that
    $s'(t) = t \delta^1(t)$. 
    Following the same computation as in \eqref{eq:Amusigma} with $f^1 = s$, we see that
    $S^\sigma = \widetilde S(\varepsilon_F^\sigma,\sigma)$, where we set $\widetilde{S} := \cM_{\cN,-t\delta^1}$. Note also that $M_n(t\delta^1) = M_{n+1}(\delta^1)$ for all $0 \le n \le p-1$. As in the proof of Lemma~\ref{lem:smearedE}, we deduce that if $\delta^1$ is of order $p \ge 1$, then $\widetilde{S}$ is smooth on $(-\sigma_1,\sigma_1) \times \R$, with
    \begin{align*}
    \forall 1 \le n \le p-1, \ 
    \dfrac{\partial^n \widetilde{S}}{\partial \varepsilon^n}(\varepsilon,0) = 0, \quad
     \dfrac{\partial^n \widetilde{S}}{\partial \sigma^n}(\varepsilon,0) = 0, \quad
      \dfrac{\partial^{p} \widetilde{S}}{\partial \sigma^{p}}(\varepsilon,0) = - \cD^{(p-1)} M_{p+1}(\delta^1).
    \end{align*}
    Together with Lemma~\ref{lem:uniqueFL}, and the chain rule, this lead to
    \begin{equation*}
       \forall 1 \le n \le p-1, \
       \left. \dfrac{\rd^n S^\sigma}{\rd \sigma^n} \right|_{\sigma = 0} = 0, \quad 
       \left.\frac{\rd^{p}S^\sigma}{\rd \sigma^{p}}\right|_{\sigma=0} = - \cD^{(p-1)}(\varepsilon_F) M_{p+1}(\delta^1).
    \end{equation*}
    Thus,
    \begin{equation} \label{eq:Ssigma}
    S^\sigma = \dfrac{1}{p!} \left( - \cD^{(p-1)}(\varepsilon_F) M_{p+1}(\delta^1) \right) \sigma^p + O(\sigma^{p+1}).
    \end{equation}
    From~\eqref{eq:estimationEsigma} and~\eqref{eq:Ssigma}, we finally obtain that
    \begin{align*}
         E^\sigma - \sigma \frac{p}{p+1}S^{\sigma} = E + O(\sigma^{p+2})
    \end{align*}
    As pointed out in~\cite{ThesisMarzari}, the right-hand side provides an approximation of
    $E$ which is consistent of order $p+2$, and therefore outperforms the
    estimator $E^\sigma$ in the asymptotic regime when $\sigma$ goes to
    zero. This is numerically useful, as, from the definition of $S^\sigma$ in~\eqref{eq:def:Ssigma}, $S^\sigma$ can be easily computed numerically.

\medskip
\subsubsection{Error on the ground-state density}
\begin{lemma} \label{lem:smearingE}
Consider a smearing \REV{\old{method}mollifier $\delta^1$} of order $p \geq 1$.
Under Assumptions 1 and 2 and with the notation of Lemma~\ref{lem:coverBZ}, there exists $\sigma_1 > 0$ and $C \in \R_+$ such that 
\begin{equation} \label{eq:estim_rho}
\forall \sigma \in (-\sigma_1,\sigma_1), \quad \left\| \rho - \rho^{\sigma} \right\|_{H^{s+2}_\per} \le C \sigma^{p+1}.
\end{equation}
\end{lemma}
\begin{proof}
  Let $W \in H^{-(s+2)}_\per$, and
  set
  $W_{n\bk} := \left\langle W, |u_{n\bk}|^2
  \right\rangle_{H^{-(s+2)}_\per, H^{s+2}_\per}$. We also set
  \begin{align*}
    A_W(\varepsilon) :=& \sum_{n \in \N^*} \fint_{\BZ} W_{n\bk} \1(\varepsilon_{n\bk}\le \varepsilon) \rd \bk 
    \quad \text{and} \quad
    A_W^{\sigma}(\varepsilon) :=& \sum_{n \in \N^{*}}\fint_{\BZ} W_{n\bk} f^\sigma\left({\varepsilon_{n\bk}-\varepsilon}\right) \, \rd \bk.
  \end{align*}
  
 In this proof, $C$ denotes a positive constant independent of $W$, $n$, $\bk$ and $\sigma$, for $\sigma$ small enough, but whose value can vary from line to line. 
  Since $H^{s+2}_\per$ is an algebra for
  $s \ge 0$ and $d \le 3$, we have that
  \[
    \left| W_{n \bk} \right| =\left| \left\langle W , | u_{n \bk}
        |^2\right\rangle_{H^{-(s+2)}_\per, H^{s+2}_\per} \right|\le C \| W \|_{H^{-(s+2)}_\per}
    \| u_{n \bk} \|_{H^{s+2}_{\per}}^2.
  \]
  Besides, using the equality
  $H_\bk u_{n \bk} = \varepsilon_{n \bk} u_{n \bk}$,
  \eqref{eq:comparison_eigenvalues}, the continuity of the pointwise
  product $H^{s+2}_\per \times H^s_\per$ to $H^s_\per$, and a
  bootstrap argument, we see that
  \begin{align}
    \label{eq:growth_Wnk}
    \forall n \ge 1, \quad \forall \bk \in \BZ, \quad \| u_{n \bk} \|_{H^{s+2}_{\per}} \le C \left(1 + \varepsilon_{n \bk}^{s+2} \right) \le C n^{\frac{2}{d}(s+2)}.
  \end{align}
  Therefore, we have
  \begin{equation*} \label{eq:estimWnk} \left| W_{n \bk} \right| \le
    C \| W \|_{H^{-(s+2)}_\per} n^{\frac{4}{d}(s+2)}.
  \end{equation*}
  This estimate shows that $A_{W}$ is a tempered distribution (as a
  continuous function of polynomial growth), and
  that the computation in \eqref{eq:Amusigma} is justified. In particular, we have
  \begin{align*}
    A_W^{\sigma}(\varepsilon) = \cM_{A_{W},\delta^1}(\varepsilon,\sigma).
  \end{align*}
  From similar considerations as in Lemma~\ref{lemma:DOS}, $A_W$ is
  smooth on $[\eps_{F}-\delta_{0},\varepsilon_{F}+\delta_{0}]$, and
\begin{equation*} \label{eq:AW'}
  A_W'(\varepsilon) =  \dfrac{1}{|\BZ|}  \sum_{n \le \overline{M}} \int_{\cS_n(\varepsilon)} \dfrac{W_{n\bk}}{| \nabla \varepsilon_{n\bk}|} \rd \sigma(\bk).
\end{equation*}
It follows that
\begin{align*}
  |\left\langle W,\rho - \rho^{\sigma}\right\rangle_{H^{-(s+2)}_\per, H^{s+2}_\per}|  = & |A_W(\varepsilon_F)-A_W^{\sigma}(\varepsilon_F^\sigma)|  \nonumber  \leq | A_W(\varepsilon_F)-A_W(\varepsilon_F^\sigma)| + \left|  A_W(\varepsilon_F^\sigma) - A_W^{\sigma}(\varepsilon_F^\sigma) \right|\\
  \leq& \left(\max_{\varepsilon \in [\varepsilon_F-\delta_0,\varepsilon_F+\delta_0]}  \left|{A_W'}(\varepsilon)\right|\right)|\varepsilon_{F}-\varepsilon_{F}^{\sigma}| + \left|  A_W(\varepsilon_F^\sigma) - A_W^{\sigma}(\varepsilon_F^\sigma) \right|\\
  \leq& \left( C \| W \|_{H^{-(s+2)}_\per}  \right) \sigma^{p+1},
\end{align*}
where we have used Lemma \ref{lem:propg} and Lemma
\ref{lem:uniqueFL}. The result follows.
\end{proof}

\subsubsection{Proof of Lemma~\ref{lem:propg}}
We finally prove Lemma~\ref{lem:propg}. Let $g \in \cS'(\R)$, $\varepsilon \in \R$ and
$\sigma \in \R^*$. We define the shifted and scaled tempered
distribution $g(\varepsilon+\sigma \cdot)$ by duality:
$$
\forall \phi \in \cS(\R), \quad  \langle g(\varepsilon+\sigma \cdot),\phi \rangle_{\cS',\cS} := \langle g,A_{\varepsilon,\sigma}\phi \rangle_{\cS',\cS},
$$
where $A_{\varepsilon,\sigma}$ is the linear map on $\cS(\R)$ defined by 
$$
\forall \phi \in \cS(\R), \quad \forall \varepsilon' \in \R, \quad
(A_{\varepsilon,\sigma}\phi)(\varepsilon') := \frac 1 {|\sigma|} \phi\left( \frac {\varepsilon'-\varepsilon}{\sigma} \right).
$$
This is consistent with the usual shift and scale operation for
functions. If $g$ is a tempered distribution that is continuous at $\varepsilon$, we define $g(\varepsilon + 0 \cdot)$ to be the constant tempered distribution with value $g(\varepsilon)$.

It is easy to check that the family $(A_{\varepsilon,\sigma})_{(\varepsilon,\sigma) \in \R \times \R^\ast}$ forms a group of continuous linear operators on $\cS(\R)$ satisfying for all $(\varepsilon,\sigma)$ and $(\varepsilon',\sigma')$ in $\R \times \R^\ast$,
$$
A_{\varepsilon,\sigma} A_{\varepsilon',\sigma'} = A_{\varepsilon+\sigma\varepsilon',\sigma\sigma'}.
$$
In addition, we have the following properties on the derivatives of $A_{\varepsilon,\sigma}$: for all $\phi \in \cS(\R)$,
\begin{align*}
& A_{\varepsilon',\sigma'}\phi \mathop{\longrightarrow}_{(\varepsilon',\sigma') \to (\varepsilon,\sigma)}A_{\varepsilon,\sigma}\phi , \\ 
& \frac{A_{\varepsilon',\sigma}-A_{\varepsilon,\sigma}}{\varepsilon'-\varepsilon}\phi \mathop{\longrightarrow}_{\varepsilon' \to \varepsilon} - \sigma^{-1} A_{\varepsilon,\sigma}\phi', \\
& \frac{A_{\varepsilon,\sigma'}-A_{\varepsilon,\sigma}}{\sigma'-\sigma}\phi \mathop{\longrightarrow}_{\varepsilon' \to \varepsilon} -\sigma^{-1} A_{\varepsilon,\sigma}L\phi, 
\end{align*}
the convergences holding in $\cS(\R)$, where $L$ is the continuous operator on $\cS(\R)$ defined by
$$
\forall \phi \in \cS(\R), \quad \forall t \in \R, \quad (L\phi)(t) = \frac{\rd}{\rd t} \left( t\phi(t) \right) = t\phi'(t)+\phi(t).
$$
It immediately follows that $\cM_{g,\phi}$ is of class $C^1$ on $(\varepsilon_F-\delta_0,\varepsilon_F+\delta_0) \times \R^\ast$ and that 
$$
\forall (\varepsilon,\sigma) \in \R \times \R^\ast, \quad \frac{\partial\cM_{g,\phi}}{\partial\varepsilon}(\varepsilon,\sigma)  = - \sigma^{-1} \langle g, A_{\varepsilon,\sigma} \phi' \rangle_{\cS',\cS}, \quad 
\frac{\partial\cM_{g,\phi}}{\partial\sigma}(\varepsilon,\sigma) = -\sigma^{-1} \langle g, A_{\varepsilon,\sigma} L\phi \rangle_{\cS',\cS}.
$$
Let us prove that $\cM_{g,\phi}$ is also $C^1$ at $\sigma = 0$. Let $\varepsilon \in (\varepsilon_F-\delta_0,\varepsilon_F+\delta_0)$, and let $\chi \in C^\infty_c(\R)$ be a cut-off function supported in a compact interval $K \subset (\varepsilon_F-\delta_0,\varepsilon_F+\delta_0)$ and equal to $1$ in a neighborhood $\cV$ of $\varepsilon$.
For $\phi \in \cS(\R)$, $\varepsilon' \in \cV$ and $\sigma \in \R^\ast$, we have
\begin{align*}
\langle g, A_{\varepsilon',\sigma}\phi \rangle_{\cS',\cS} &= \langle \chi g, A_{\varepsilon',\sigma}\phi \rangle_{\cS',\cS} + \langle (1-\chi) g, A_{\varepsilon',\sigma} \phi \rangle_{\cS',\cS} \\
&= \langle A_{\varepsilon',\sigma} \phi , \chi g\rangle_{C^{0}(K)',C^{0}(K)} + \langle g, (1-\chi) A_{\varepsilon',\sigma} \phi \rangle_{\cS',\cS}  \mathop{\longrightarrow}_{(\varepsilon',\sigma) \to (\varepsilon,0)} g(\varepsilon) M_0(\phi),
\end{align*}
since when $(\varepsilon',\sigma) \to (\varepsilon,0)$, $A_{\varepsilon',\sigma} \phi  \to M_0(\phi)\delta_\varepsilon$ in the space $C^{0}(K)'$ of bounded Borel measures on $K$, while $(1-\chi) A_{\varepsilon',\sigma} \phi$ goes to zero in $\cS(\R)$. This proves that $\cM_{g, \phi}$ is continuous on $(\varepsilon_F - \delta_0, \varepsilon_F, + \delta_0) \times \R$.

\medskip

If in addition, $g$ is of class $C^1$ on $(\varepsilon_F-\delta_0,\varepsilon_F+\delta_0)$, then for $\sigma \neq 0$,
\begin{align*}
\frac{\cM_{g,\phi}(\varepsilon,\sigma)-\cM_{g,\phi}(\varepsilon,0)}{\sigma} &= \frac{\langle \chi g, A_{\varepsilon,\sigma}\phi \rangle_{\cS',\cS} - M_0(\phi)(\chi g)(\varepsilon)}{\sigma} +   \frac{\langle (1-\chi) g, A_{\varepsilon,\sigma}\phi \rangle_{\cS',\cS}}{\sigma} \\ 
&= \langle  \sigma^{-1}(A_{\varepsilon,\sigma}\phi - M_0(\phi)\delta_\varepsilon), \chi g \rangle_{(C^1(K))',C^1(K)} +   \langle g, (1-\chi)  \sigma^{-1} A_{\varepsilon,\sigma}\phi \rangle_{\cS',\cS} \\
&  \mathop{\longrightarrow}_{\sigma \to 0}  (\chi g)'(\varepsilon) M_1(\phi) = g'(\varepsilon) M_1(\phi),
\end{align*}
since $\sigma^{-1}(A_{\varepsilon,\sigma}\phi - M_0(\phi)\delta_\varepsilon)$
converges to $M_{1}(\phi)\delta_{\varepsilon}'$ in $C^{1}(K)'$ , while
$(1-\chi) \sigma^{-1} A_{\varepsilon,\sigma}\phi$ converges to $0$ in
$\cS(\R)$. Hence, for
$\varepsilon \in (\varepsilon_F-\delta_0,\varepsilon_F+\delta_0)$, we have
$$
\frac{\partial \cM_{g,\phi}}{\partial\varepsilon}(\varepsilon,0) = g'(\varepsilon)M_0(\phi), \quad \frac{\partial \cM_{g,\phi}}{\partial\sigma}(\varepsilon,0) = g'(\varepsilon)M_1(\phi).
$$
Observing that $M_0(\phi')= 0$ and $M_0(L\phi)=0$, so that $M_1(\phi') = -M_0(\phi)$, and that, $M_1(L\phi) = - M_1(\phi)$ with an integration by part, and reasoning as above, we obtain that for $\varepsilon' \in \cV$ and $\sigma \in \R^\ast$,
\begin{align*}
&\frac{\partial \cM_{g,\phi}}{\partial\varepsilon}(\varepsilon,\sigma)= - \sigma^{-1} \langle g, A_{\varepsilon,\sigma} \phi' \rangle_{\cS',\cS} \mathop{\longrightarrow}_{\sigma \to 0} - g'(\varepsilon) M_1(\phi') = g'(\varepsilon) M_0(\phi), \\
&\frac{\partial\cM_{g,\phi}}{\partial\sigma}(\varepsilon,\sigma) = - \sigma^{-1} \langle g, A_{\varepsilon,\sigma} L\phi \rangle_{\cS',\cS}\mathop{\longrightarrow}_{\sigma \to 0} - g'(\varepsilon) M_1(L\phi) = g'(\varepsilon) M_1(\phi).
\end{align*}
It follows that $\cM_{g,\varepsilon}$ is of class $C^1$ in $(\varepsilon_F-\delta_0,\varepsilon_F+\delta_0) \times \R$. Similar arguments allow one to show that $\cM_{g,\varepsilon}$ is of class $C^k$ in $(\varepsilon_F-\delta_0,\varepsilon_F+\delta_0) \times \R$ whenever $g$ is of class $C^k$ in $(\varepsilon_F-\delta_0,\varepsilon_F+\delta_0)$. The proof follows by a straightforward induction.

\subsection{Error between smeared quantities and corresponding Riemann sums}
\label{sec:error-betw-smear}
We now investigate the convergence of $A^{\sigma,L}$ to $A^{\sigma}$,
for $\sigma$ fixed. As we already mentioned, the advantage of using
smearing functions is that the quantities ${A}^{\sigma}$ are defined
as the integral of smooth periodic functions. It is therefore natural
to approximate numerically the integral by a regular Riemann sum, for
which we can expect exponential convergence, depending on the analytic
properties of the integrand (see for instance
\cite{trefethen2014exponentially} for a review). For $Y > 0$, we
introduce the (closed) complex strip
\[
S_Y := \R^d + \ri [-Y, Y]^d =  \left\{ \bz \in \CC^d, \quad | \Im(\bz) |_\infty \le Y \right\}.
\]
We recall the following classical result, which is proved as in
\cite[Lemma 5.1]{Gontier2016_M2AN}.
\begin{lemma}
\label{lem:analyticity_implies_exp_decay}
There exists $C \in \R_+$ and $\eta > 0$ such that, for all $Y > 0$ and all $F : \CC^d \to \CC$ that is analytic on $S_Y$ and $\RLat$-periodic on $\R^{d}$, we have
\begin{align*}
\forall L \in \N^*, \quad \left| \fint_\BZ F(\bk) \rd \bk - \dfrac{1}{L^d} \sum_{\bk \in
\cB_L} F(\bk) \right| \le C \left(\max_{\bz \in S_Y} |F(\bz)|\right)
\dfrac{\re^{-\eta Y L}}{Y^d}.
\end{align*}
\end{lemma}

We see that\REV{, for a fixed value of $\sigma$}, the greater the region of analyticity of the integrand,
the faster the convergence as $L\to\infty$. We distinguish here the
different smearing functions: the Fermi-Dirac function
$f_{\rm FD}^\sigma$ has poles at $(2\Z+1)\pi\sigma\ri$ and displays
exponential convergence, while the Gaussian-type smearing functions
are entire, leading to super-exponential convergence. We collect in
Lemmas~\ref{lem:nearRealLine}-\ref{lem:farRealLine} estimates on the
analytic behavior of the integrands of interest, from which the
results in this section proceed.

\REV{We would like to emphasize at this point that, if the value of $\sigma$ is not fixed but depends on $L$, obtaining rates of convergence becomes a more subtle and intricate task. 
We address this issue in more details in Remark~\ref{rem:sigmaadapt}.}

\subsubsection{Error for the integrated density of states}

 \begin{lemma}[Convergence of the integrated density of states] \label{lem:NsigmaL} It holds that \\
 $\bullet$ If $f^1$ is any of the smearing functions
 (\ref{smearing_function_1}-\ref{smearing_function_1}$'''$), there
 exists $C \in \R_{+}$ and $\eta > 0$ such that for all $0 < \sigma \le \sigma_0$ and all $L \in \N^*$,
 \begin{equation} \label{eq:expCVN}
 \max_{\varepsilon \in [\varepsilon_F - \delta_0, \varepsilon_F +
 \delta_0]} \left|\cN^{\sigma,L} (\varepsilon) - \cN^{\sigma} (\varepsilon) \right| \leq C \sigma^{-(d+1)}\re^{-\eta \sigma L}.
 \end{equation}
 $\bullet$ If $f^1$ is a Gaussian-type smearing function (\ref{smearing_function_1}'-\ref{smearing_function_1}$'''$), then there exists $C' \in \R_+$ and $\eta' > 0$ such that, for all $0 < \sigma \le \sigma_0$ and all $L \in \N^*$, such that $\sigma^2 L \ge 4$, it holds that
 \begin{equation} \label{eq:superExpCVN}
 \max_{\varepsilon \in [\varepsilon_F - \delta_0, \varepsilon_F +
 \delta_0]} \left|\cN^{\sigma,L} (\varepsilon) - \cN^{\sigma} (\varepsilon) \right| \leq C' \sigma^{-\tfrac{5d}3} L^{-\tfrac d3} \re^{- \eta' \sigma^{2/3} L^{4/3}}.
 \end{equation}
 \end{lemma}

\begin{proof}
We want to compare $\cN^{\sigma,L}(\varepsilon)$ with $\cN^\sigma(\varepsilon)$. This is
exactly the framework of Lemma
\ref{lem:analyticity_implies_exp_decay}, with integrand
\begin{align*}
F_{\varepsilon,\sigma}^{\cN}(\bk) := \sum_{n \in \N^*} f^{\sigma}(\varepsilon_{n\bk} - \varepsilon)
&= \Tr_{L^2_\per} \left[ f^{\sigma} \left( H_\bk - \varepsilon \right) \right].
\end{align*}
In Appendix~\ref{sec:appendixFanalytic}, we study the analytic property of such functions. From Lemma~\ref{lem:nearRealLine}, there exists $C \in \R_+$ and $Y > 0$ such that,  for all $\varepsilon \in [\varepsilon_F - \delta_0, \varepsilon_F + \delta_0]$ and all $0 < \sigma \le \sigma_0$, the map $F_{\varepsilon,\sigma}^{\cN}$ admits an analytic continuation on $S_{\sigma Y}$, and it holds that
\begin{equation*}
\sup_{\bz \in S_{\sigma Y}} \left| F_{\varepsilon,\sigma}^{\cN}(\bz) \right| \le C \sigma^{-1}.
\end{equation*}
Together with Lemma~\ref{lem:analyticity_implies_exp_decay}, we deduce that there exists $C, C' \in \R^+$ and $\eta' > 0$ such that
\[
\left|\cN^{\sigma,L} (\varepsilon) - \cN^{\sigma} (\varepsilon) \right| \leq C \sigma^{-1} (\sigma Y)^{-d} \re^{-\eta \sigma Y L} \leq C' \sigma^{-(d+1)}  \re^{-\eta' \sigma L}.
\]
This proves the first part~\eqref{eq:expCVN}. For the second part, we
use Lemma~\ref{lem:farRealLine}, and get in a similar way that if
$f^1$ is a Gaussian-type smearing function, then there exists $C \in \R^+$ and $\eta > 0$ such that for all $Y \ge 1$, all $\varepsilon \in [\varepsilon_F - \delta_0, \varepsilon_F + \delta_0]$ and all $0 < \sigma \le \sigma_0$,
\[
\left|\cN^{\sigma,L} (\varepsilon) - \cN^{\sigma} (\varepsilon) \right| \leq C (\sigma Y)^{-d} \re^{\eta \left(  \tfrac{Y^4}{\sigma^{2}} - YL \right)}.
\]
Taking $Y = Y(\sigma,L) = \left( \tfrac14 \sigma^2 L \right)^{1/3}$ leads to the result. \REV{The condition $\sigma^2 L \ge 4$ comes from the fact that we need $Y \ge 1$ for this result to be valid.}
\end{proof}

\subsubsection{Error for the Fermi level, the total energy, and the density}
We now turn to the Fermi energy, the total energy, and the density. As
above, $\cN^{\sigma,L}$ is a continuous function that satisfies
$\cN^{\sigma,L}(-\infty) = 0$ and $\cN^{\sigma,L}(+\infty) = +\infty$,
hence there exists $\varepsilon_F^{\sigma,L} \in \R$ so that
$\cN^{\sigma,L}(\varepsilon_F^{\sigma,L}) = N$. As $\cN^{\sigma,L}$ is
not necessarily increasing, $\varepsilon_F^{\sigma,L}$ may be non
unique. However, since $\cN^{\sigma,L}$ is continuous, close to
$\cN^{\sigma}$ (in the sense of the lemma above), and
$\cN^{\sigma}(\varepsilon_{F}^{\sigma})=N$ with
$\partial_{\varepsilon} \cN^{\sigma}(\varepsilon_{F}^{\sigma}) > 0$,
it follows from the intermediary value theorem that there exists an
$\varepsilon_{F}^{\sigma,L}$ close to $\varepsilon_{F}^{\sigma}$ so
that $\cN^{\sigma,L}(\varepsilon_F^{\sigma,L}) = N$. It is this
$\varepsilon_{F}^{\sigma,L}$ that we assume to be chosen for the rest
of this paper.

\medskip


 \begin{lemma}[Convergence of the Fermi energy, the total energy, and the density] \label{lem:epsSigmaL} It holds that \\
 $\bullet$ If $f^1$ is any of the smearing functions in
 (\ref{smearing_function_1}-\ref{smearing_function_1}$'''$), there
 exists $C \in \R_{+}$ and $\eta > 0$ such that for all $0 < \sigma \le \sigma_0$ and all $L \in \N^*$,
 \begin{equation} \label{eq:expCVepsF}
  \left| \varepsilon_F^\sigma - \varepsilon_F^{\sigma, L} \right|
  + \left| E^\sigma - E^{\sigma, L} \right|
   + \left\| \rho^\sigma - \rho^{\sigma, L} \right\|_{H^{s+2}_{\per}}
    \leq C \sigma^{-(d+1)}\re^{-\eta \sigma L}.
 \end{equation}
 $\bullet$ If $f^1$ is a Gaussian-type smearing function
 (\ref{smearing_function_1}'-\ref{smearing_function_1}$'''$), there
 exists $C \in \R_+$ and $\eta > 0$ such that, for all
 $0 < \sigma \le \sigma_0$ and all $L \in \N^*$ \REV{with $\sigma^2 L \ge 4$}, it holds that
 \begin{equation} \label{eq:superExpCVepsF}
 \left| \varepsilon_F^\sigma - \varepsilon_F^{\sigma, L} \right|
  + \left| E^\sigma - E^{\sigma, L} \right|
   + \left\| \rho^\sigma - \rho^{\sigma, L} \right\|_{H^{s+2}_{\per}}
    \leq C \sigma^{-\tfrac{5d}3} L^{-\tfrac d3} \re^{- \eta \sigma^{2/3} L^{4/3}}.
 \end{equation}
 \end{lemma}

\begin{proof}
  The Fermi level is estimated as outlined above. For the energy (the
  proof is similar for the density), we have
\begin{align} \label{eq:Esigma-EsigmaL}
\left| E^\sigma - E^{\sigma, L} \right| \le \left| E^\sigma(\varepsilon_F^\sigma) - E^{\sigma}(\varepsilon_F^{\sigma,L}) \right| + \left| E^{\sigma}(\varepsilon_F^{\sigma,L}) - E^{\sigma, L}(\varepsilon_{F}^{\sigma, L} ) \right|.
\end{align}
The second term of~\eqref{eq:Esigma-EsigmaL} is exactly in the scope of Lemma~\ref{lem:analyticity_implies_exp_decay} when we consider the function
\[
F_{\sigma}^{E}(\bk) := \sum_{n \in \N^*} \varepsilon_{n\bk} f^{\sigma}(\varepsilon_{n\bk} - \varepsilon^{\sigma, L})
= \Tr_{L^2_\per} \left[ H_\bk f^{\sigma} \left( H_\bk - \varepsilon^{\sigma, L} \right) \right].
\]
Following the lines of Lemma~\ref{lem:NsigmaL} together with
Lemmas~\ref{lem:nearRealLine}-\ref{lem:farRealLine}, we see that the
first term is exponentially (resp. superexponentially) small. The
first term of~\eqref{eq:Esigma-EsigmaL} can be evaluated by noticing that
\[
	\left| E^\sigma(\varepsilon_F^\sigma) - E^{\sigma}(\varepsilon_F^{\sigma,L}) \right| \le \left( \max_{\varepsilon \in [\varepsilon_F - \delta_0/2, \varepsilon_F + \delta_0/2]} \left| \partial_{\varepsilon} E^\sigma \right| \right)\left| \eps_F^\sigma - \eps_F^{\sigma,L} \right|.
\]
The proof follows.
\end{proof}


\subsection{Total error}

We finally combine Lemma~\ref{lem:uniqueFL}, Lemma~\ref{lem:smearedE}, Lemma~\ref{lem:smearingE} and Lemma~\ref{lem:epsSigmaL}  to obtain our final result.
\begin{theorem}
\label{esigmaL-e}
Consider a smearing method of order $p \geq 1$. 
Under Assumption 1 and 2, there exist $C \in \R_{+}$ and $\eta > 0$ such that, for all $0 < \sigma \leq
\sigma_{0}$ and all $L \in \N^{*}$, it holds that
\begin{align}
  \label{eq:final_estimate}
  |\eps_{F}^{\sigma,L} - \eps_{F}| + |E^{\sigma,L} - E| + \|\rho^{\sigma,L} - \rho\|_{H^{s+2}_{\per}} &\leq C(\sigma^{p+1} + \sigma^{-(d+1)}\re^{-\eta \sigma L}).
\end{align}  
\end{theorem}

\begin{remark}[Choosing $\sigma$ adaptively]\label{rem:sigmaadapt}
  In practice, only the parameter $L$ is relevant when considering CPU
  time. The numerical parameter $\sigma$ can be chosen freely with no
  extra numerical cost, and can be optimized with respect to $L$. For
  instance, the choice $\sigma \propto L^{-1}$ has been recommended
  for practical calculations in \cite{bjorkman2011adaptive}, based on
  a heuristic argument. According to our previous theorem, this is not
  enough: the right-hand side of \eqref{eq:final_estimate} does not
  tend to zero when $L \to \infty$ and $\sigma = C/L$. \REV{\old{Still, choosing
  the slightly different scaling $\sigma \propto L^{-1+\eps}$ for some
  $\eps > 0$ leads to a quasi-optimal decay of the error proportional
  to $L^{-(p+1)(1-\eps)}$.} 
  Still, choosing the slightly different scaling $\sigma \propto \log(L) L^{-1}$ leads to a decay of the error proportional to $L^{-(p+1)}$, up to log factors.
  } Our results therefore broadly support those
  of \cite{bjorkman2011adaptive}.
\end{remark}
\begin{remark}[Superexponential convergence]
    \REV{The choice $\sigma \propto \log(L)/L$ gives $\sigma^2 L \propto \log(L)^2/L$, which goes to $0$ as $L \to infty$. In particular, the condition $\sigma^2 L \ge 4$ is not satisfied for large $L$.
  \old{In order to balance the second term in \eqref{eq:final_estimate}, we
  need to choose $\sigma$ small. In particular, we are looking at the
  analytic property of the different functions of interest near the
  real axis.}} The super-exponential scaling result is therefore
  irrelevant for numerical purposes when we want to compute
  zero-temperature quantities.
\end{remark}

\section{Numerical results}\label{sec:resnum}

The aim of this section is to present some numerical results on toy
test cases to illustrate the convergence properties of the methods
analyzed in the previous sections. We also present some examples where
Assumption~1 or Assumption~2 are not valid. In these cases, we
numerically observe a degraded convergence rate.

\medskip

In Section~\ref{sec:interp}, some results about interpolation methods
are presented for different interpolation orders.
In Section~\ref{subsec:smearing}, numerical tests are
performed using some of the smearing methods described in
Section~\ref{sec:smearing-methods}.

\subsection{Interpolation methods}

We consider two-dimensional toy test cases, with $\BZ = (-1/2,1/2)^2$.
For a given $L\in \N^*$, we consider a uniform $L\times L$
discretization grid of $\BZ$ as described in Section~\ref{sec:interp},
and use B-splines of order $1$ and $2$ as interpolation operators. In
all computations in this section the numerical quadratures are
performed up to an error of $10^{-6}$, and so we display error curves
only above that threshold.

\medskip

\paragraph{\textbf{Case 1}}
Let us first consider a two-dimensional example where Assumptions~1
and~2 are satisfied (case 1 in the following). We consider one analytical band:
\begin{equation}\label{eq:oneband}
\forall \bk:=(k_1, k_2)\in \BZ, \quad \varepsilon_{1 \bk} = 3\cos(2\pi k_1)\cos(2\pi k_2) + \sin(4\pi k_1)\cos(4\pi k_2)
\end{equation}
represented on Figure~\ref{fig:contour}. The number of electrons per unit
cell is chosen to be equal to $N = 0.85$ so that the exact Fermi level
is approximately $\varepsilon_F \approx 1.7275$. We plot the error on
the energy $\left| E - E^{L,p,q} \right|$ and the Fermi level
$\left| \varepsilon_F - \varepsilon_F^{L,q} \right|$ as a function of
$L$ on Figure~\ref{fig:interpCos1}, for different interpolation
schemes (we recall that $p$ is the order used to interpolate the
integrand, and $q$ the order used to interpolate $\varepsilon_{n\bk}$
for the purposes of determining the integration region and Fermi level).

\begin{figure}[h!]
\includegraphics[width=0.5\textwidth]{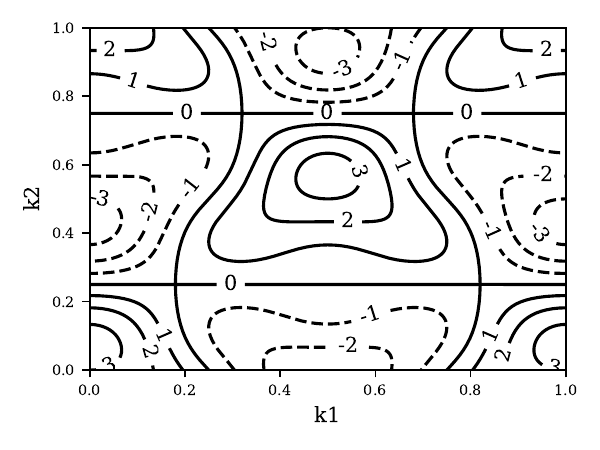}
\caption{Isolines of the band structure \eqref{eq:oneband} used in
  cases 1 and 2 (Fermi level $\approx 1.7275$ and $0$ respectively).}
\label{fig:contour}
\end{figure}

\begin{figure}[h!]
\begin{tabular}{cc}
\includegraphics[width=0.5\textwidth]{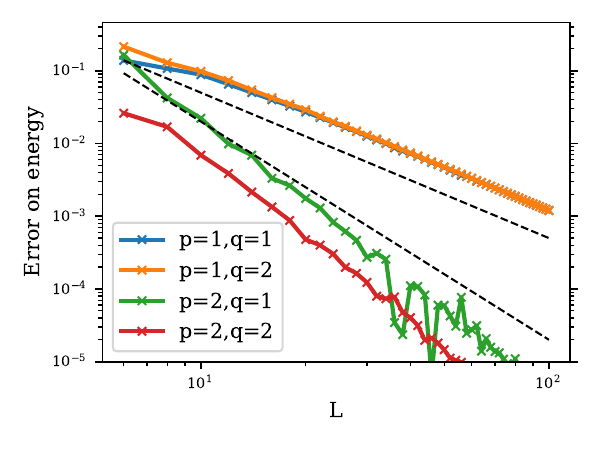} &
 \includegraphics[width=0.5\textwidth]{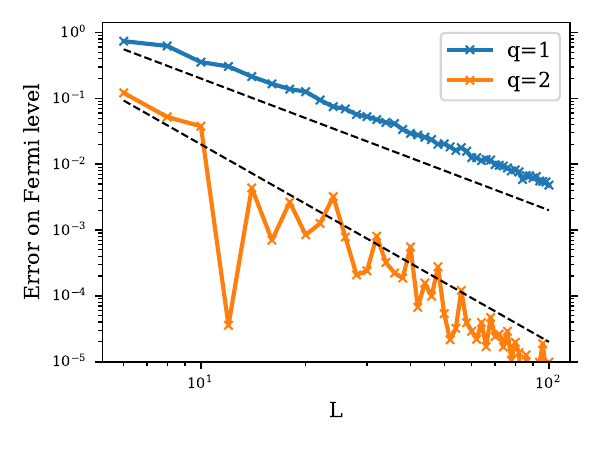}\\ 
\end{tabular}
\caption{Case 1: errors made on the energy (left) and the Fermi level
  (right). The two dashed lines are proportional to $L^{-2}$ and
  $L^{-3}$ respectively. The $p=1$, $q=1$ and $p=1$, $q=2$ curves are almost identical
  in the left plot.}
\label{fig:interpCos1}
\end{figure}

We see that the energy for the methods with $p=1$ converge converges
as $1/L^{2}$, irrespective of $q$, and that the Fermi level converges
as $1/L^{2}$ for $q=1$, as expected from our estimates. For the higher-order
methods, the convergence is more erratic, and it is difficult to
determine the order precisely. However, it can be seen that the $p=2$, $q=2$
method only marginally improves the error on the energy compared to
the $p=2$, $q=1$ method, as expected from our estimates.

\medskip
 
\paragraph{\textbf{Cases 2 and 3} }
We now consider two other two-dimensional examples where either
Assumption~1 or Assumption~2 is violated.

In the test violating Assumption~\REV{2\old{1}} ({case~2} in the following), only
one band is considered, with the same analytic expression
(\ref{eq:oneband}) as in case 1. However, the number of electrons is
now chosen to be equal to $N=0.5$ so that the exact Fermi level is
equal to $\varepsilon_F = 0$. There are saddle points of the function
$\varepsilon_{1\bk}$ at level
$\varepsilon_F = 0$, as can be seen in
Figure~\ref{fig:contour}. This produces a van Hove singularity in the
density of states at the Fermi level, violating Assumption~1. Note
however that the singularity for a saddle point is relatively mild in
2D: $\cD(\varepsilon)$ diverges logarithmically near
$\varepsilon_{F}$.
 
The test case violating Assumption~\REV{1\old{2}} ({case~3} in the following) is
the standard tight-binding model of graphene \cite{neto2009electronic}, where band crossings
occur on the Fermi surface.

We denote for all $\bk = (k_1,k_2)\in \BZ$
\begin{align*}
  c_{1}(\bk) = \frac{k_{1}+k_{2}}{3}, \quad c_{2}(\bk) = \frac{k_{1}-2k_{2}}3, \quad c_{3}(\bk) = \frac{-2k_{1}+k_{2}}3.
\end{align*}
For all $\bk\in \BZ$, we then define the $2\times 2$ Hermitian matrix $H(\bk)$ as follows
$$
H(\bk):= \left( \begin{array}{cc}
                 0 & \sum_{j=1}^3 \re^{-2  \pi \ri c_j(\bk)} \\
                 \sum_{j=1}^3 \re^{2  \pi\ri c_j(\bk)} & 0\\
                \end{array}\right).
$$
It can be checked that, while $H$ is not periodic, for all
$\bK \in \mathcal R^{*}$, $H(\bk+\bK)$ is unitarily equivalent to
$H(\bk)$, so that the eigenvalues of $H(\bk)$ are periodic on $\BZ$.
In this test case, the number of electrons per unit cell is chosen to
be equal to $N=1$ so that the exact Fermi level is
$\varepsilon_F = 0$. The two bands $\varepsilon_{1\bk}$ and
$\varepsilon_{2\bk}$ touch at the Fermi level at two non-equivalent
points of the Brillouin zone (the Dirac points $\bK$ and $\bK'$),
violating Assumption~2.

In Figure~\ref{fig:interppath} we plot the errors between the
exact and approximate energies $\left| E - E^{L,p,q} \right|$ as a
function of $L$ using the same interpolation schemes as described
above.

\begin{figure}[h!]

\begin{tabular}{cc}
\includegraphics[width=0.5\textwidth]{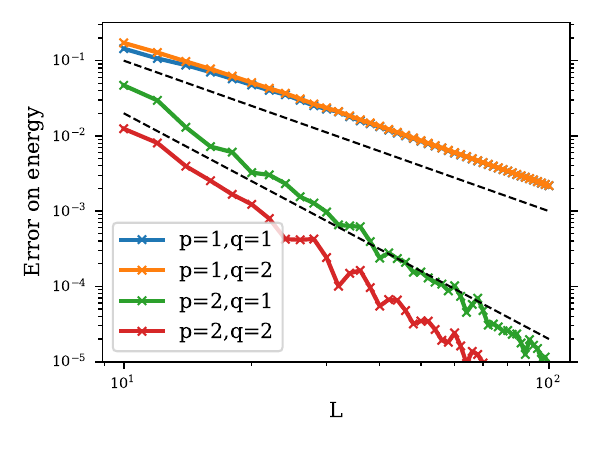} &
 \includegraphics[width=0.5\textwidth]{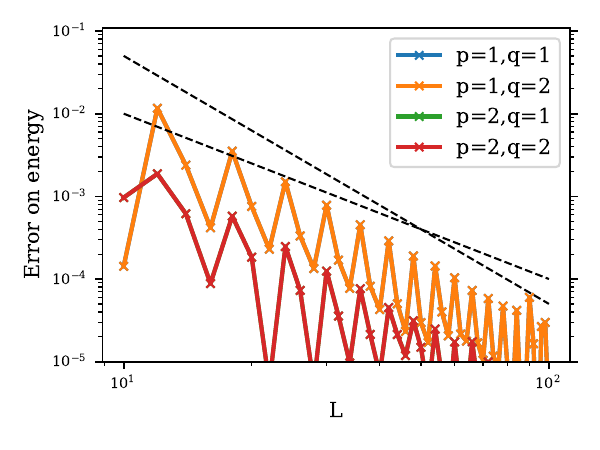}\\ 
\end{tabular}
\caption{Errors made on the energy for interpolation methods on cases
  2 (left) and 3 (right). The two dashed lines are proportional to $L^{-2}$ and
  $L^{-3}$ respectively. On the right plot, the results are
  independent of $q$.}
\label{fig:interppath}
 \end{figure}
 
 The different errors for case 2 seem to decay at a rate similar to
 that in case 1. We attribute this to the relatively mild singularity
 of this case. Case 3 however displays a very different behavior, the
 results depending on the position of the Dirac points $\bK$ and $\bK'$ on
 the grid and therefore displaying an oscillating pattern.

\subsection{Smearing methods}\label{subsec:smearing}

The aim of this section is to present numerical results for some
smearing methods presented and analyzed in
Section~\ref{sec:smearing-methods}, for the three two-dimensional test
cases presented in Section~\ref{sec:interp}. We consider here the
Gaussian (denoted by GA, of order $1$) and the first Methfessel-Paxton
method (denoted by MP, of order $3$) smearing schemes.

\medskip

Let us begin with the first test case, namely the one-band
model~(\ref{eq:oneband}) for which Assumptions~1 and~2 are satisfied.
The errors on the energy and on the Fermi level are plotted on Figure~\ref{fig:smea1} as a function of $L$ and for different values of $\sigma$. 

\begin{figure}[h!]
\begin{tabular}{cc}
 \includegraphics[width=.5\textwidth]{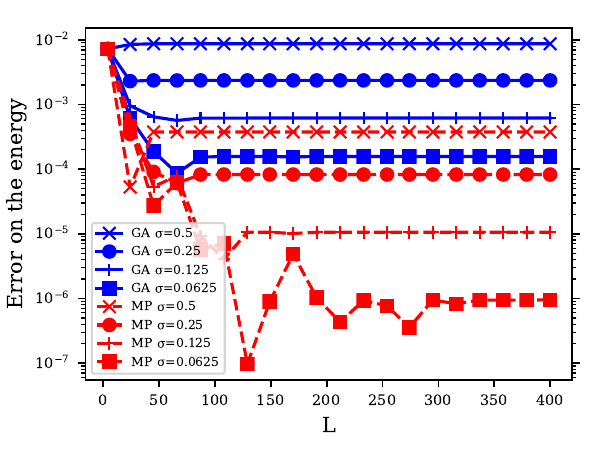}&
 \includegraphics[width=.5\textwidth]{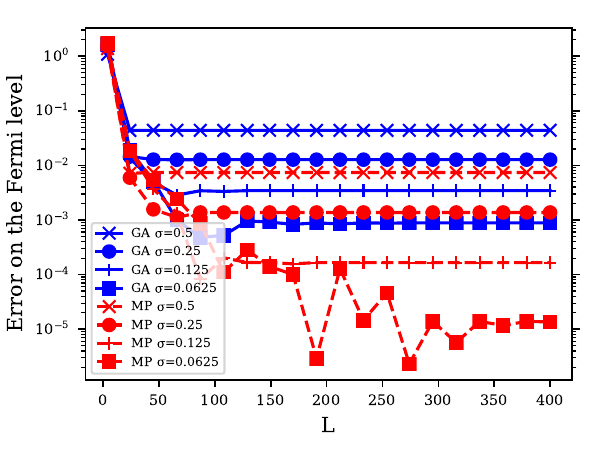}
\end{tabular}
\caption{
  Error on the energy $\left| E - E^{\sigma,L}\right|$ (left) and
  Fermi level $\left| \varepsilon_{F}
    - \varepsilon_{F}^{\sigma,L}\right|$ (right) for case 1 as a function of $L$ for different smearing schemes and different values of $\sigma$.}
\label{fig:smea1}
\end{figure}

From this we obtain the following conclusions:
\begin{itemize}
\item The error as $L\to\infty$ decreases as $\sigma$
  decreases. As $\sigma$ is reduced by a factor of $2$, the error
  $|E^{\sigma,\infty}-E|$ is reduced by a factor of about $4$ (for
  GA) and by a factor of about $8$ (for MP1), suggesting that the
  asymptotic regime is not yet reached for the MP1 method (we would
  expect a factor $16$ since the MP1 method is of order 3).
\item As $\sigma$ is reduced, so is the speed of convergence of
  $E^{\sigma,L}$ to $E^{\sigma,\infty}$. The number of points $L$
  required to achieve convergence of $E^{\sigma,L}$ to
  $E^{\sigma,\infty}$ scales approximately linearly as $1/\sigma$, as
  expected from the $\rme^{-\eta\sigma L}$ term in
  Theorem~\ref{esigmaL-e}.
\item On this toy example and in the parameter regimes
  considered here, the method that gives the lowest error for a given
  $L$ seems to be that with lowest smearing and highest order, giving
  the impression that smearing is not advantageous. This is of course
  not the case in more realistic examples with lower values of $L$,
  where there is a non-zero optimal smearing.
\end{itemize}

 \medskip

 We also present numerical results obtained on the two degenerate
 cases presented in Section~\ref{sec:interp} where Assumption~1 or
 Assumption~2 are violated. Errors $\left| E^{\sigma,L}-E\right|$ are
 plotted as a function of $L$ for cases
 1 and 2  in Figure~\ref{fig:smea1path}. In contrast to before, we see that the
 Methfessel-Paxton scheme is not able to achieve a higher order
 than the Gaussian smearing. This is because of the
 lack of regularity of the density of states at the Fermi level in
 these two cases.

 \begin{figure}[h!]
\begin{tabular}{cc}
\includegraphics[width=.5\textwidth]{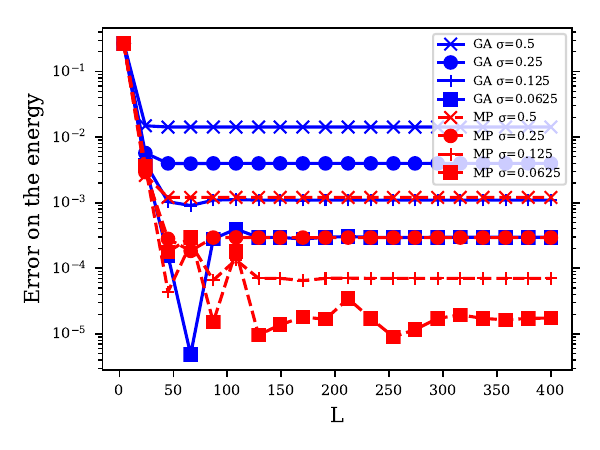}&
 \includegraphics[width=.5\textwidth]{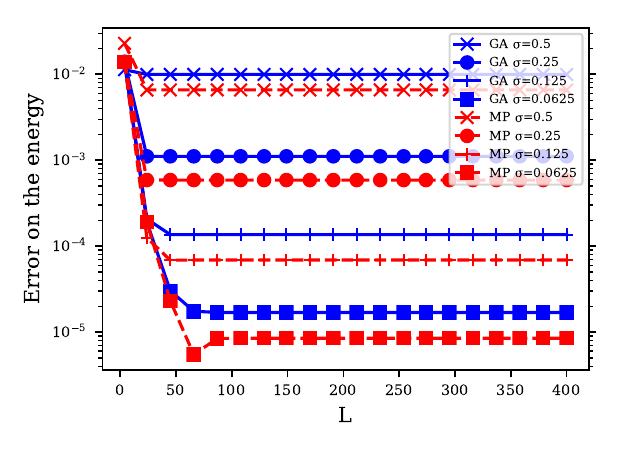}\\ 
\end{tabular}
\caption{Error on the energy $\left| E - E^{\sigma,L}\right|$ for cases 2 (left) and 3 (right) as a function of $L$ for different smearing schemes and different values of $\sigma$.}
\label{fig:smea1path}
 \end{figure}

\section{Conclusion}
We have presented an {\it a priori} error analysis of quadrature rules and smearing
methods for $\bk$-point integration in the Brillouin zone. Our
conclusions justify several non-obvious schemes, and give rigorous bounds allowing one to choose the smearing
parameter optimally.

\medskip

Our analysis is concerned with linear periodic Schr\"odinger operators; as such,
a number of extensions are necessary to covers the framework of density functional theory, which is the main
motivation underlying this work. The nonlinearity of the
Kohn-Sham equations is expected to give rise to difficulties, which can probably be addressed using the tools developed in~\cite{CCM2012}. Another source of 
error not considered here is the space (or plane-wave) discretization used in numerical simulations of periodic Schr\"odinger operators. This has a strong impact 
on our estimates, which rely on the fact that the map $\R^d \ni \bk \mapsto (\ri+H_{\bk})^{-1} \in \cB(L^2_\per)$ is smooth and pseudo-periodic. This cannot be ensured at the discrete level 
for a standard Galerkin discretization, see Remark \ref{rem:discretization}. We plan to explore all these issues in a forthcoming paper.

\medskip

In this work, we only considered simple ground state properties.
Several quantities of interest do not fit into this framework, such as
response functions, which involve integrals
over Fermi surfaces, derivatives with respect to $\bk$ of occupied Bloch states, or unoccupied Bloch states. The error
estimates in this case are expected to be substantially different, in
accordance to the common observation that these quantities
converge slowly in practice as functions of $L$.

\medskip

To establish our results, we have relied on Assumptions 1 and 2, which
mathematically define what is a ``simple'' metal (at least for the
properties we have considered). For such metals, the asymptotic
behavior is universal, and the scaling laws only depend on the
interpolation or smearing used. It would be interesting to explore
what happens in the case of a semimetal, or when a van Hove
singularity is present at (or close to) the Fermi surface.

\appendix

\section{Analytic properties of the integrand}
\label{sec:appendixFanalytic}
The goal of this appendix is to study the analytic properties of the functions $F_{\varepsilon,\sigma}^{\cN}$, $F_{\varepsilon,\sigma}^{E}$ and $F_{\varepsilon,\sigma}^{\rho, W}$ defined respectively by
\begin{align*}
F_{\varepsilon,\sigma}^{\cN}(\bk) & := \Tr_{L^2_\per} \left[ f^{\sigma} \left( H_\bk - \varepsilon \right) \right], \\
F_{\varepsilon,\sigma}^{E}(\bk) & := \Tr_{L^2_\per} \left[ H_\bk f^{\sigma} \left( H_\bk - \varepsilon \right) \right], \\
F_{\varepsilon,\sigma}^{\rho, W}(\bk) & := \Tr_{L^2_\per} \left[ W f^{\sigma} \left( H_\bk - \varepsilon \right) \right].
\end{align*}

We prove the following two results. The first one studies the analytic properties of these functions near the real line.
\begin{lemma}[Analyticity near the real line] \label{lem:nearRealLine}
Let $f^1$ be any of the smearing functions $\mbox{(\ref{smearing_function_1})-(\ref{smearing_function_1}$\,'''\!$).}$ Then, there exist $C \in \R_+$ and ${Y} > 0$ such that, for all $\varepsilon \in [\varepsilon_F - \delta_0, \varepsilon_F + \delta_0]$ and all $0 < \sigma \le \sigma_0$, the maps $F_{\varepsilon,\sigma}^{X}$ admits an analytic continuation on $S_{\sigma Y}$, and it holds that 
\begin{equation*}
  \sup_{\bz \in S_{\sigma Y}} \left| F_{\varepsilon,\sigma}^{\cN}(\bz) \right| + \sup_{\bz \in S_{\sigma Y}} \left| F_{\varepsilon,\sigma}^{E}(\bz) \right|  + \sup_{{\bz \in S_{\sigma Y}, \; \|W\|_{H^{-(s+2)}_\per=1}}} \left| F_{\varepsilon,\sigma}^{\rho, W}(\bz) \right|\le C \sigma^{-1}.
\end{equation*}
\end{lemma}

Our second result studies the analytic properties on the entire complex plane.
\begin{lemma}[Analyticity on $\CC$] \label{lem:farRealLine}
  Let $f^1$ be one of the Gaussian-type smearing functions (\ref{smearing_function_1}')-(\ref{smearing_function_1}$\,'''\!$). Then, the maps $F_{\varepsilon,\sigma}^{X}$ are entire, and there exists $C \in \R_+$ and $\eta > 0$ such that, for all $Y \ge 1$, all $\varepsilon \in [\varepsilon_F - \delta_0, \varepsilon_F + \delta_0]$ and all $0 < \sigma \le \sigma_0$,
\begin{equation*}
\sup_{\bz \in S_{\sigma Y}} \left| F_{\varepsilon,\sigma}^{\cN}(\bz) \right| + \sup_{\bz \in S_{\sigma Y}} \left| F_{\varepsilon,\sigma}^{E}(\bz) \right| + \sup_{{\bz \in S_{\sigma Y}, \;\|W\|_{H^{-(s+2)}_\per} = 1}} \left| F_{\varepsilon,\sigma}^{\rho, W}(\bz) \right|\le C \re^{\eta \frac{Y^4}{\sigma^{2}}}.
\end{equation*}
\end{lemma}

Let us first highlight the idea of the proofs of these lemmas. First, we will obtain bounds that only depend on
$\|V_{\per}\|_{L^\infty}$, so that we can absorb
$\varepsilon \in (\eps_{F}-\delta_{0},\eps_{F}+\delta_{0})$ in
$V_{\per}$. Without loss of generality, we take
$\varepsilon = 0$ and drop the subscript $\varepsilon$. In addition, to shorten the presentation, we only do the proof for the energy. In
the following, $F_{\sigma}$ denotes $F_{\varepsilon = 0, \sigma}^{E}$.

\medskip

We wish to study the analytic properties of the function
\begin{align*}
F_{\sigma}(\bk) = \Tr_{L^2_\per} \left[ H_{\bk} f^1 \left( H_\bk/\sigma \right) \right],
\end{align*}
where $f^1$ is one of the smearing functions~$\mbox{(\ref{smearing_function_1})-(\ref{smearing_function_1}$'''$)}$.
Formally, $H_{\bk}$ admits an analytic continuation to the whole
complex plane: when $\bz = \bk +\ri \by \in \CC^{d}$, we set
\begin{align*}
H_{\bz} = - \frac 1 2 \Delta_{\bz} + V_{\per},
\end{align*}
where
\begin{align} \label{eq:def:Deltaz}
- \Delta_{\bz} &=  (-\ri \nabla + \bz)^{2} =  (-\ri \nabla + \bk)^{2} + 2 \ri \by \cdot (-\ri \nabla + \bk) - {|\by|^{2}}.
\end{align}
The operator $H_{\bz}$ is not self-adjoint (it is not even a normal
operator), so it is difficult to compute the analytical extension $F_\sigma(\bz)$ of $F_{\sigma}(\bk)$. We
will use a representation in terms of contour integrals. Formally, it holds that
\begin{align*}
\Tr_{L^2_\per} \left[H_{\bz} f^{1} \left( H_\bz/\sigma \right) \right] 
& = \Tr_{L^{2}_{\per}} \left[ \dfrac{1}{2 \pi \ri} \oint_{\sC}  \lambda f^{1}(\lambda/\sigma) \left(\dfrac{1}{\lambda - H_\bz} \right) \rd \lambda \right] \\
& = \dfrac{1}{2 \pi \ri} \oint_{\sC}  \lambda f^{1}(\lambda/\sigma) \Tr_{L^2_\per} \left[ \dfrac{1}{\lambda - H_\bz} \right] \rd \lambda
\end{align*}
for some (infinite) contour $\sC$ enclosing the spectrum of $H_{\bz}$. Unfortunately, we cannot commute the trace and the
integral in the last line whenever the dimension $d$ is greater than
$1$. The reason is that the operator $(\lambda - H_\bz)^{-1}$ is not
trace-class when $d \geq 2$ (see \eqref{eq:comparison_eigenvalues}). Instead, we consider the contour integral\footnote{
	In the case of the density, we can take for instance
	\[
	G_{\sigma}^W(\bz) := \dfrac{1}{2 \pi \ri} \int_{\sC} \left[ f^{1}(\lambda/\sigma)(\lambda + \Sigma)^2 \right] \Tr_{L^2_\per} \left( \dfrac{1}{(H_\bz + \Sigma)^{(s+4)/2}}W \dfrac{1}{(H_\bz + \Sigma)^{(s+4)/2}} \dfrac{1}{\lambda - H_\bz} \right) \rd \lambda,
      \]
      with the integrand trace-class from \eqref{eq:growth_Wnk}.
}
\begin{equation} \label{eq:TrAnalytic}
G_{\sigma}(\bz) = \dfrac{1}{2 \pi \ri} \int_{\sC} \left[\lambda f^{1}(\lambda/\sigma)(\lambda + \Sigma)^2 \right] \Tr_{L^2_\per} \left( \dfrac{1}{(H_\bz + \Sigma)^2} \dfrac{1}{\lambda - H_\bz} \right) \rd \lambda,
\end{equation}
where $\Sigma \in \R$ is a well-chosen shift, and prove that $G_{\sigma}$ is
an analytic continuation of $F_{\sigma}$.

\medskip

We prove Lemma~\ref{lem:nearRealLine} and Lemma~\ref{lem:farRealLine} in the following two sections. In both cases
we prove that there exists appropriate contours such that $G_{\sigma}
= F_\sigma$ using a perturbation argument. When $\bz = \bk + \ri \by
\in \CC^3$ with $| \by |$ small, we see $H_\bz$ as a perturbation of
$H_\bk$, while when $\by$ is large, we see it as a perturbation of the
free operator $-\frac12 \Delta_\bz$.

\subsection{Proof of Lemma~\ref{lem:nearRealLine}}
We introduce $\sC$ the parabolic contour in the complex plane defined by
\begin{align} \label{eq:def:sC}
\sC := \left\{\lambda \in \CC, | \Im \ \lambda |^2 = \sigma^{2}\left(1 + \frac{\Re \ \lambda}{\Sigma}\right)\right\}, 
\quad \text{where we set} \quad
\Sigma := \|V\|_{L^\infty}+1.
\end{align}
It is the (unique) parabola that passes through the points $-\Sigma$ and
$\pm \ri \sigma$. In particular, it does not encounter the poles of the
Fermi-Dirac function at $(2\Z + 1)\pi \sigma \ri$. Let us prove that $\sC$ encloses the spectrum of $H_{\bk + \ri \by}$ for $\by$ small
enough (see Figure \ref{fig:contourFD_1}).

\begin{figure}[H]
\centering
\begin{tikzpicture}

\draw[->] (-7, 0) -> (6, 0) node[right] {$\R$};
\draw[->] (-1.5, -2) -> (-1.5, 2) node[above] {$\ri \R$};

\draw[red, line width=3] (-4,0) -- (-1,0);
\draw[red, line width=3] (0,0) -- (2,0);
\draw[red, line width=3] (3,0) -- (5,0) node[below] {$\sigma(H)$};


\draw [black, fill] (0, 0.3) circle(2pt) ;     \draw [black, fill] (-4.3, 0.05) circle(2pt) ;     \draw [black, fill] (1, -0.4) circle(2pt) ;
\draw [black, fill] (-3, 0.1) circle(2pt) ;     \draw [black, fill] (3, 0.3) circle(2pt) ;     \draw [black, fill] (3.5, 0.4) circle(2pt) node[right] {$\sigma(H_{\bz})$};

\draw [black, fill] (-1.5, 1.5) node[cross]{} node[above right] {$\pi \sigma\ri$};
\draw [black, fill] (-1.5, .7) node[cross]{} node[above right] {$\sigma\ri$};
\draw [black, fill] (-6., 0.) node[cross]{} node[above left] {$-\Sigma$};

\draw [domain=0:3.4]  plot ({\x^2-6}, {\x/3} ) node[above right] {$\sC$};
\draw [ domain=0:3.4]  plot ({\x^2-6}, {-\x/3} );

\end{tikzpicture}
\caption{Spectrum of the operator $H_{\bz}$ for $\bz = \bk + \ri \by$, and the contour $\sC$
that encloses it, while avoiding the poles of the Fermi-Dirac
function at $(2\Z+1)\sigma\ri$. 
}
\label{fig:contourFD_1}
\end{figure}

\begin{lemma}
\label{lem:res_ineg}
There exist $C \in \R^+$ and $Y > 0$ such that, for all $0<\sigma \leq
\sigma_{0}$, all $\bz = \bk + \ri \by \in S_{\sigma Y}$, and all $\lambda \in \sC$, we have the following estimates:
\begin{align}
\label{res_ineg_1}
\|(H_{\bz}-\lambda)^{-1}\|_{\sB} &\leq C \sigma^{-1},\\
\label{res_ineg_2}
\|(H_{\bz}+\Sigma)^{-1}\|_{\fS_{2}} &\leq C,\\
\label{res_ineg_3}
\Tr_{L^2_\per} \left( \dfrac{1}{(H_\bz + \Sigma)^2} \dfrac{1}{\lambda - H_\bz} \right) &\leq C \sigma^{-1},\\
\label{res_ineg_4}
\partial_{\bz} \Tr_{L^2_\per} \left( \dfrac{1}{(H_\bz + \Sigma)^2} \dfrac{1}{\lambda - H_\bz} \right) &\leq C \sigma^{-1}.
\end{align}
\end{lemma}
\begin{proof}

From~\eqref{eq:def:Deltaz}, it holds that
\begin{align*}
H_{\bz} = \frac12 (-\ri \nabla + \bk)^{2} +  \ri \by \cdot (-\ri \nabla + \bk) - \frac12{|\by|^{2}} + V_\per = \left( H_{\bk} - \tfrac 1 2 |\by|^{2} \right) + \ri \by \cdot (-\ri \nabla + \bk).
\end{align*}
For $Y \le 1$, the spectrum of the self-adjoint operator $H_{\bk} - \frac 1 2 |\by|^{2}$ is contained in $ [-\Sigma+ \tfrac 1 2, +\infty)$, which is disjoint from $\sC$. In particular, for $\lambda \in \sC$, we have
\begin{equation} \label{eq:lambda-Hz}
\left( \lambda - H_\bz \right) =  \left( \lambda - (H_{\bk} - \tfrac 1 2 |\by|^{2}) \right) \left( 1 + \ri \by \cdot \left( \lambda - (H_{\bk} - \tfrac 1 2 |\by|^{2}) \right)^{-1}  (\ri \nabla + \bk) \right).
\end{equation}
Let us evaluate the norm of the normal operator $\left( \lambda - (H_{\bk} - \tfrac 1 2 |\by|^{2}) \right)^{-1} $. From~\eqref{eq:def:sC}, we have
\begin{align*}
\left\|\left(  \lambda - (H_{\bk} - \tfrac 1 2 |\by|^{2})\right)^{-1}\right\|_{\cB} &=
\dist\left(\lambda, \Spec(H_{\bk} - \tfrac 1 2 |\by|^{2})\right)^{-1} 
\leq \left[ {(\Im \lambda)^{2} + (-\Sigma + \tfrac 1 2 - \Re \lambda)_{+}^{2}} \right]^{-1/2}\\
&= \left[{\sigma^{2}(1+ \tfrac{\Re \lambda} \Sigma) + (-\Sigma + \tfrac 1 2 - \Re \lambda)_{+}^{2}}\right]^{-1/2}
\leq C \sigma^{-1},
\end{align*}
where $x_{+} := \max(x,0)$ and where $C \in \R_{+}$ is a constant that depends only on
$\Sigma$ and $\sigma_{0}$. The last inequality is obtained by optimizing
over all $\Re \lambda \geq -\Sigma$.

\medskip

From the fact that $H_{\bk}$ is a bounded perturbation of $\tfrac
1 2 (-\ri
\nabla + \bk)^{2}$, and using similar calculations, we easily get that there exists $C \in \R_+$ that depends only on $\Sigma$ and $\sigma_0$ such that
\begin{align}
\label{res_times_nabla}
\|\ri\by \cdot (\lambda - (H_{\bk} - \tfrac 1 2 |\by|^{2}))^{-1} \cdot (-\ri \nabla + \bk)\|_{\cB}\leq C \sigma^{-1} | \by |.
\end{align}

As a consequence, for ${Y} \leq 1/(2C)$ and $| \by | \le \sigma Y$, the operator on the right parenthesis of~\eqref{eq:lambda-Hz} is invertible, and its inverse is bounded in norm by $2$. Inverting~\eqref{eq:lambda-Hz} leads to \eqref{res_ineg_1}.

\medskip

Inequality \eqref{res_ineg_2} is proved in a similar way (notice that the operator
$(\Sigma - (H_{\bk} - \frac 1 2 |\by|^{2}))^{-1}$ is Hilbert-Schmidt
by \eqref{eq:comparison_eigenvalues}). Inequality \eqref{res_ineg_3} is a consequence of \eqref{res_ineg_1} and \eqref{res_ineg_2}, together with the operator inequality $\left| \Tr(B^2A) \right|  \le \| B \|_{\fS_2}^2 \| A \|_{\sB}$.

\medskip

We finally prove \eqref{res_ineg_4}. For all
$\mu$ in the resolvent set of $H_{\bz}$, we have
\begin{align*}
\partial_{\bz} \left( \frac{1}{H_{\bz} - \mu} \right) &= - \left( \frac{1}{H_{\bz} - \mu} \right) \partial_{\bz} H_{\bz} \left( \frac{1}{H_{\bz} - \mu} \right)
= - \frac{1}{H_{\bz} - \mu}  (-\ri \nabla + \bz) \frac{1}{H_{\bz} - \mu}.
\end{align*}
We then use similar arguments and the fact that the operator $(-\ri \nabla + \bz) \left(H_{\bz} - \lambda \right)^{-1}$ is bounded.
\end{proof}

We can now prove the analyticity of $G_\sigma$ defined in~\eqref{eq:TrAnalytic}.
\begin{lemma}
\label{lem:G_analytic}
There exist $C \in \R_+$ and $Y > 0$ such that, for all $0<\sigma \leq \sigma_{0}$,
the function $G_{\sigma}$ is analytic on $S_{\sigma Y}$, and
\begin{align*}
\sup_{\bz \in S_{\sigma Y }} |G_{\sigma}(\bz)| \leq C \sigma^{-1}.
\end{align*}
\end{lemma}
\begin{proof}
From the previous Lemma~\ref{lem:res_ineg}, we already see that $G_\sigma$ is analytic. Let us prove the bound. It holds that
\begin{align*}
|G_{\sigma}(\bz)| &\leq C \sigma^{-1} \int_{\sC} |f^{1}(\sigma^{-1} \lambda) \lambda(\lambda + \Sigma)^2| |\rd \lambda|.
\end{align*}
To evaluate the last integral, we parametrize the contour $\sC$ with $\lambda(t) := \Sigma(t^{2} - 1) + \ri \sigma
t$ for $t \in \R$, so that 
\begin{equation} \label{eq:estimatelambda}
| \lambda | = \left( \Sigma^2 (t^2 - 1)^2 + \sigma^2 t^2 \right)^{1/2} \le C(t^2 + 1)
\quad \text{and} \quad
|\rd \lambda| = \sqrt{4 t^{2} \Sigma^{2} + \sigma^{2}} \rd t \le C (| t | + 1) \rd t,
\end{equation}
for some $C \in \R_{+}$ that depends only on $\Sigma$ and $\sigma_0$.
We obtain
\begin{align*}
\left| G_{\sigma}(\bz) \right| &
\leq C \sigma^{-1} \left(  \int_{\R} \left|f^{1}\left( \sigma^{-1} \Sigma(t^{2} - 1) + \ri t \right)\right| (1+| t |^{7})\rd t \right).
\end{align*}

Let us prove that the last integral is uniformly bounded for
$0<\sigma \leq \sigma_{0}$. We prove this result in full details for the Fermi-Dirac smearing $f^{1}(x) =
(1+\re^{x})^{-1}$, the other cases being similar. We
split the integral in the regions $|t| \leq \pi/2$ and
$|t| > \pi/2$. For $| t | \leq \pi/2$, it holds that $\cos t \ge 0$, so that
\begin{align*}
\left| f^1 \left( \sigma^{-1}  \Sigma (t^{2} - 1) + \ri t \right) \right|
& = \left| 1 + \re^{\sigma^{-1} \Sigma(t^{2} - 1)} \re^{\ri t} \right|^{-1}
\le \left| \Re \left( 1 + \re^{\sigma^{-1} \Sigma(t^{2} - 1)} \re^{\ri t} \right) \right|^{-1} \\
& =  \left| 1 + \re^{\sigma^{-1} \Sigma(t^{2} - 1)} \cos t \right|^{-1} \le 1.
\end{align*}
We deduce that the integral over $| t | \leq \pi/2$ is uniformly bounded for $0<\sigma \le \sigma_0$. For $t \ge \pi/2 > 1$, it holds that $(t^2 - 1) > 0$, so that
\[
\left| 1 + \re^{\sigma^{-1} \Sigma(t^{2} - 1)} \re^{\ri t} \right|^{-1} \le \left| \re^{\sigma^{-1} \Sigma(t^{2} - 1)} -1 \right|^{-1} \le \left| \re^{\sigma_0^{-1} \Sigma( t^2  - 1)} -1 \right|^{-1},
\]
where the right-hand side no longer depends on $0<\sigma \le \sigma_0$. Finally, we check that the integral
\[
\int_{| t | \ge \frac{\pi}{2}}\left| \re^{\sigma_0^{-1} \Sigma( t^2  - 1)} -1 \right|^{-1} (1 + | t |^7) \rd t
\]
is absolutely convergent, and independent of $0<\sigma \le \sigma_0$. This ends the proof of Lemma~\ref{lem:G_analytic}.


\end{proof}

  We now prove, as claimed, that $G_{\sigma}$ is an analytic extension
  of $F_{\sigma}$.
  \begin{lemma}
  \label{lem:G_eq_F}
  For all $0 < \sigma \leq \sigma_{0}$ and all $\bk \in \R^d$, it holds that $G_{\sigma}(\bk) = F_{\sigma}(\bk)$.
  \end{lemma}

  \begin{proof}
  We first approximate $H_{\bk}$ by a finite-rank
  operator, then apply the Cauchy residual formula, and pass to the
  limit. Recall that $H_\bk$ is a self-adjoint operator with spectral
  decomposition~\eqref{eq:def:Hk}. For $Q \in \N^*$, we introduce
  the truncated operator
  \[
  H_\bk^Q := \sum_{n =1}^Q \varepsilon_{n \bk}  |
  u_{n\bk} \ket \bra u_{n\bk} |.
  \]
  We have
  \begin{align}
  \left| G_\sigma(\bk) - F_\sigma(\bk) \right| 
  \le & \left| G_\sigma(\bk) - \Tr_{L^2_\per} \left( H_\bk^Q f^\sigma (H_\bk^Q) \right) \right| \label{ineq:G-F1} \\
  & + \left|  \Tr_{L^2_\per} \left( H_\bk^Q f^\sigma (H_\bk^Q) \right) -  \Tr_{L^2_\per} \left( H_\bk f^\sigma (H_\bk) \right) \right|. \label{ineq:G-F2}
  \end{align}
  We first focus on~\eqref{ineq:G-F2}. Using the asymptotic~\eqref{eq:comparison_eigenvalues} and the decay properties of $f^\sigma$, it holds that
  \begin{equation} \label{eq:convTLtoT}
  \Tr_{L^2_\per} \left( H_\bk f^\sigma(H_\bk ) \right) - \Tr_{L^2_\per} \left(H_\bk^Q  f^\sigma(H_\bk^Q) \right) = \sum_{n \ge Q} \varepsilon_{n \bk} f^{\sigma}(\varepsilon_{n\bk}) \xrightarrow[Q \to \infty]{} 0.
  \end{equation}
  We now focus on the right-hand side of~\eqref{ineq:G-F1}. For $M \ge \varepsilon_{Q \bk} + 1$ we denote by $\sC_{M}$ the positively oriented {\em closed} contour defined by
  \[
  \sC_M := \left\{ \lambda \in \sC, \ \Re \lambda \le M \right\} \bigcup \left[  M + \ri \sigma \left( 1 + \frac{M}{\Sigma} \right),  M - \ri \sigma \left( 1 + \frac{ M}{\Sigma} \right)\right].
  \]
  The contour $\sC_M$ is obtained by truncating the parabola $\sC$ to the region
  $\Re \lambda \leq  M$ and closing the contour by a segment. For all $M \ge \varepsilon_{Q \bk} + 1$, this contour encloses the spectrum of
  $H_\bk^Q$, so that, from the Cauchy residual formula,
  \begin{align*}
  \Tr_{L^2_\per} \left( H_\bk^Q f^\sigma(H_\bk^Q) \right) = \frac{1}{2 \pi \ri} \oint_{\sC_M}  \left[f^\sigma(\lambda) \lambda(\lambda+ \Sigma)^2 \right] \Tr_{L^2_\per} \left( \dfrac{1}{(H_\bk^Q + \Sigma)^2}  \dfrac{1}{\lambda - H_\bk^Q} \right) \rd \lambda.
  \end{align*}
As $M \to \infty$, and using the same arguments as in the proofs of Lemmas \ref{lem:res_ineg} and \ref{lem:G_analytic}, we see that the right-hand side
  converges to the integral over the full contour $\sC$. Moreover, we have the point-wise convergence
  \begin{align*}
  \forall \lambda \in \sC, \quad & \lim_{Q \to \infty} \Tr_{L^2_\per} \left( \dfrac{1}{(H_\bk^Q + \Sigma)^2} \dfrac{1}{\lambda - H_\bk^Q} \right) 
  = \lim_{Q \to \infty} \sum_{n =1}^{Q} \frac{1}{(\eps_{n\bk} + \Sigma)^{2}(\lambda-\eps_{n\bk})} \\
  &\qquad =\sum_{n=1}^{\infty} \frac{1}{(\eps_{n\bk} + \Sigma)^{2}(\lambda-\eps_{n\bk})}
  = \Tr_{L^2_\per} \left( \dfrac{1}{(H_\bk + \Sigma)^2} \dfrac{1}{\lambda - H_\bk} \right).
  \end{align*}
  We conclude from the dominated convergence theorem that
  $\Tr_{L^2_\per} \left( H_\bk^Q f^\sigma(H_\bk^Q) \right)$ converges to $G_{\sigma}(\bk)$ as $Q \to \infty$. The proof of Lemma~\ref{lem:G_eq_F} follows.
  \end{proof}
  Combining Lemma~\ref{lem:G_eq_F} together with Lemma~\ref{lem:G_analytic} ends the proof of Lemma~\ref{lem:nearRealLine}.

  \subsection{Proof of Lemma~\ref{lem:farRealLine}}
  We now focus on Gaussian-type smearing functions. The idea of the proof is similar to previously. We need however to re-define $G_\sigma$ for large $| \by |$ (\ie choose an appropriate contour). In the sequel, we fix $Y \ge 1$ and provide estimates uniform in $\bz \in S_Y$.

  \medskip

Looking at the operator $-\frac 1 2 \Delta_{\bz}$ in Fourier basis, we see that its  spectrum is
  \[
  \sigma \left( -\tfrac12 \Delta_\bz \right) := \left\{  \tfrac 1 2(\bK + \bz)^{2} \right\}_{\bK \in \RLat} 
  \left\{ \frac 1 2 (|\bK + \bk|^{2} - |\by|^{2}) + \ri \by \cdot (\bK+\bk) \right\}_{\bK \in \RLat},
  \]
  hence is contained in the parabolic set $(\Im \lambda)^{2} \leq Y^{2}(2\Re \lambda + Y^{2})$. In the sequel, we take an parabolic integration contour that encloses this region. More specifically, for $\alpha > 1$  and $Y \ge 1$, we introduce the dilated contour
  \begin{align*}
  \sC_{\alpha,Y} := \left\{\lambda \in \CC, \ \Im \lambda^{2} = \alpha^{2} Y^{2}(2\Re \lambda + \alpha^{2} Y^{2})\right\}.
  \end{align*}
  As $\alpha$ increases, the distance between $\sC_{\alpha,Y}$ and $\sigma (-\tfrac12 \Delta_\bz)$ goes to $+ \infty$. Since the operator $-\tfrac12 \Delta_\bz$ is normal, this implies that there exists $\alpha_c \ge 1$ such that
  \[
  \forall Y \ge 1, \quad \forall \bz \in S_Y, \quad \forall \lambda \in \sC_{\alpha_c, Y}, \quad  \left\| (\lambda +\tfrac12 \Delta_\bz )^{-1} \right\|_\sB \le \left( 2 \| V \|_{L^\infty} \right)^{-1}.
  \]
The contour $\sC := \sC_{\alpha_c, Y}$ is our integration contour for $\bz \in S_Y$ (see Figure \ref{fig:contourFD_2}).
  \begin{figure}[H]
  \centering
  \begin{tikzpicture}
\draw[->] (-7.5, 0) -> (6, 0) node[right] {$\R$};
\draw[->] (-1.5, -2) -> (-1.5, 2) node[above] {$\ri \R$};

\draw[red, line width=3] (-4,0) -- (-1,0);
\draw[red, line width=3] (0,0) -- (2,0);
\draw[red, line width=3] (3,0) -- (5,0) node[below] {$\Spec(H)$};

\draw [blue, domain=0:3.3]  plot ({\x^2-5}, {\x/5} ) node[above
left] {$\Spec(-\tfrac 1 2 \Delta_{\bz})$};
\draw [blue, domain=0:3.3]  plot ({\x^2-5}, {-\x/5} );

\draw [black, fill] (0, 0.5) circle(2pt) ;     \draw [black, fill] (-4.3, 0.3) circle(2pt) ;     \draw [black, fill] (1, -0.4) circle(2pt) ;
\draw [black, fill] (-3, 0.1) circle(2pt) ;     \draw [black, fill] (3, 0.3) circle(2pt) ;     \draw [black, fill] (3.5, -1) circle(2pt) node[right] {$\Spec(H_\bz)$};


\draw [domain=0:3.23]  plot ({1.2*\x^2-1.1*6}, {1.4*\x/3} ) node[above right] {$\sC$};
\draw [domain=0:3.23]  plot ({1.2*\x^2-1.1*6}, {-1.4*\x/3} );

\draw[black,fill](-1.5,1) node[cross]{} node[above right]{$\ri \alpha_c^{2}Y^{2}$};
\draw[blue,fill](-1.5,.4) node[cross,blue]{} node[above right]{$\ri Y^{2}$};

\draw[black,fill](-1.1*6,0) node[cross]{} node[above]{$-\Sigma=-\tfrac{\alpha_c^{2}Y^{2}}2$};
\draw[blue,fill](-5,0) node[cross,blue]{} node[below left]{$-\tfrac{Y^{2}}2$};

\end{tikzpicture}
\caption{The spectra of the operator $-\tfrac12 \Delta_\bz$ and $H_{\bz}$, and the contour $\sC$
that encloses them.}
\label{fig:contourFD_2}
\end{figure}

For $\lambda \in \sC$, it holds that (compare with~\eqref{eq:lambda-Hz})
\begin{equation*}
\left( \lambda - H_{\bz}\right) = \left( \lambda + \tfrac 1 2 \Delta_{\bz} - V \right) = \left( \lambda + \tfrac12 \Delta_\bz \right) \left( 1 - (\lambda + \tfrac12 \Delta_\bz)^{-1} V \right).
\end{equation*}
As in Lemma~\ref{lem:res_ineg} we deduce the following inequalities. We do not repeat the proof, as it is similar. We denote by $\Sigma := \frac{\alpha_c^{2}Y^{2}}2$ for the sake of clarity.
\begin{lemma} \label{lem:boundRes2}
There exists $C \in \R_+$ such that, for all $Y \geq 1$, all $\bz \in S_Y$, and all $\lambda \in \sC$, it holds that
\begin{align*}{}
& \|(H_{\bz}-\lambda)^{-1}\|_{\sB} \leq C, 
\quad  
\|(H_{\bz}+\Sigma)^{-1}\|_{\fS_{2}} \leq C,\\
&  \Tr_{L^2_\per} \left( \dfrac{1}{(H_\bz + \Sigma)^2} \dfrac{1}{\lambda - H_\bz} \right) \leq C,
\quad \text{and} \quad
\partial_{\bz} \Tr_{L^2_\per} \left( \dfrac{1}{(H_\bz + \Sigma)^2} \dfrac{1}{\lambda - H_\bz} \right) \leq C.
\end{align*}
\end{lemma}


We now define, for $0<\sigma \le \sigma_0$ and $Y \ge 1$, the function defined on $S_Y$ by
\begin{equation}  \label{eq:def:GsigmaY}
G_{\sigma,Y}(\bz) := \dfrac{1}{2 \pi \ri} \int_{\sC} \left[\lambda f^{1}(\lambda/\sigma)(\lambda + \Sigma)^2 \right] \Tr_{L^2_\per} \left( \dfrac{1}{(H_\bz + \Sigma)^2} \dfrac{1}{\lambda - H_\bz} \right) \rd \lambda.
\end{equation}
Before stating a bound on $G_{\sigma,Y}$, we need the following
technical lemma on the growth of $f^{1}$:
\begin{lemma}[Growth of $f^{1}$ in the complex plane] \label{lem:propf1} When $f^{1}$ is one of the
Gaussian-type smearing functions
(\ref{smearing_function_1}$''$-\ref{smearing_function_1}$'''$),
then $f^1$ is entire, and there exists $C \in \R_{+}$ and $q \ge 0$ such that
\begin{equation} \label{eq:ineq_erfc}
\forall x,y \in \R, \quad \left| f^1(x + \ri y) \right| \le
\left\{ \begin{aligned}
& C (1+(x^2 + y^2)^q) \left(\re^{ y^2 - x^2} \right) \quad \text{if} \quad x \ge 0, \\
& C (1+(x^2 + y^2)^q)  \left( 1 + \re^{ y^2  - x^2} \right) \quad \text{if} \quad x < 0.
\end{aligned} \right.
\end{equation}
\end{lemma}

\begin{proof}
We first handle the case when $f^{1}$ is the Gaussian smearing function (in which case we can choose $q = 0$). We have
\begin{equation*}
f^1(x) = \frac12 \left( 1 - {\rm erf} \left( x \right) \right) =
\frac{1}{\sqrt{\pi}} \int_{x}^\infty \re^{-t^2} \rd t.
\end{equation*}
First recall that for $x \in \R$, we have $0 < f^1(x) < 1$. Moreover, for $x \ge 1$, it holds that
\[
0 \le f^1(x) \le \ \frac{1}{\sqrt{\pi}} \int_{x}^\infty t \re^{-t^2} \rd t = \frac{1}{2 \sqrt{\pi}} \re^{-x^2}.
\]
This already proves~\eqref{eq:ineq_erfc} for $y = 0$. 
Then, we notice that an analytic continuation of $f^1$ is given by
\begin{equation*}
f^1(x + \ri y) = \dfrac{1}{\sqrt{\pi}} \left( \int_x^\infty \re^{-t^2} \rd t + \ri
\int_{0}^{y} \re^{-(x + \ri t)^2} \rd t\right) = f^1(x) + \ri \dfrac{\re^{-x^2}}{\sqrt{\pi}} \int_0^y \re^{- 2 \ri x t} \re^{t^2} \rd t.
\end{equation*}
Lemma~\ref{lem:propf1} then follows from the inequalities
\[
\left| f^1(x + \ri y) \right| \le \left| f^1(x) \right| +  \dfrac{\re^{-x^2}}{\sqrt{\pi}} \int_{[0,y]}  \re^{t^2} \rd t,
\]
together with the fact that
\[
\int_{[0,y]}  \re^{t^2} \rd t = \int_{[0,1]} \re^{t^2} \rd t + \int_{[1,y]} \re^{t^2} \rd t  \le \int_{[0,1]} \re^{t^2} \rd t + \int_{[1,y]} t \re^{t^2} \rd t \le C \re^{y^2}.
\]
The case of the Methfessel-Paxton and cold smearing schemes follows immediately by noting
that they differ from the Gaussian smearing function by terms of the
form $x^n \re^{-x^{2}}$, and the fact that
\[
\left| z^n \re^{-z^2} \right| = ( x^2 + y^2 )^{n/2} \re^{y^2 - x^2}.
\]
\end{proof}

We are now in position to prove estimates on $G_{\sigma, Y}$ defined on~\eqref{eq:def:GsigmaY}.

\begin{lemma} \label{lem:GsigmaY}
There exists $C \in \R_+$ and $\eta > 0$ such that, for all
$0 < \sigma \leq \sigma_{0}$ and all $Y \geq 1$, the function $G_{\sigma,Y}$ is analytic on $S_Y$, and
\begin{align*}
\sup_{\bz \in S_{Y}} |G_{\sigma,Y}(\bz)| \leq C \re^{\eta \frac{Y^{4}}{\sigma^{2}}}.
\end{align*}

\end{lemma}
\begin{proof}
For the sake of clarity, we do the proof when $f^1$ is the Gaussian smearing function (\ie $q = 0$). From Lemma~\ref{lem:boundRes2}, it holds that
\begin{align*}
|G_{\sigma,Y}(\bz)| &\leq C \int_{\sC} | \lambda f^{1}(\lambda/\sigma) (\lambda + \Sigma)^2| |\rd \lambda|.
\end{align*}
We parametrize the contour $\sC$ with $\lambda(t) := \tfrac{\alpha_c^{2} Y^{2}}{2} ( (t^{2}-1) + \ri 2t)$, so that (compare with~\eqref{eq:estimatelambda})
\[
| \lambda | \le C Y^2 (t^2 + 1 ) \quad \text{and} \quad | \rd \lambda | \le C Y^2 (| t | + 1) \rd t.
\]
Then,
\begin{align*}
\left| G_{\sigma,Y}(\bz) \right| &\leq C Y^{8} \int_{\R} \left|f^{1} \left(\tfrac{\alpha_c^{2}Y^{2}}{2 \sigma} \left(t^{2}-1+\ri 2t\right)\right)\right|(1+|t|^{7}) \rd t.
\end{align*}
We split this integral in regions where $| t | \ge 1$, and $| t | \le 1$. When $| t | \le 1$, it holds that $\Re \lambda(t) \le 0$. Together with the second inequality of Lemma~\ref{lem:propf1} and the inequalities
\[
	\forall -1 \le t \le 1, \quad (1 + | t |^7) \le 2, \quad \text{and} \quad 4t^2 - (t^2 - 1)^2 \le 4, 
\] 
we get
\begin{align*}
\int_{[-1, 1]} \left|f^{1} \left(\tfrac{\alpha_c^{2}Y^{2}}{2 \sigma} \left(t^{2}-1+\ri 2 t\right)\right)\right|(1+|t|^{7}) \rd t 
& \le 2 C \int_{[-1,1]} \left[ 1 + \exp\left( \tfrac{\alpha_c^4 Y^4}{4 \sigma^2}(4 t^2 - (t^2 - 1)^2) \right) \right] \rd t  \\
&  \le 4 C \left[ 1 + \re^{\frac{\alpha_c^4 Y^4}{\sigma^2} } \right] \le 8C \re^{\frac{\alpha_c^4 Y^4}{\sigma^2}}.
\end{align*}
For $| t | \ge 1$, we use the first inequality of Lemma~\ref{lem:propf1}, and obtain
\begin{align*}
& \int_{[-1, 1]^c} \left|f^{1} \left(\tfrac{\alpha_c^{2}Y^{2}}{2\sigma} \left(t^{2}-1+\ri 2 t\right)\right)\right|(1+|t|^{7}) \rd t 
\le 2C \int_1^\infty \exp\left( \tfrac{\alpha_c^4 Y^4}{4 \sigma^2}(4 t^2 - (t^2-1)^2)  \right) (1 + | t |^7)\rd t.
\end{align*}
When $t \ge 1$, we have the inequalities
\[
(1 + | t |^7) \le 2 | t |^7 \quad \text{while} \quad 6 t^2 \le \frac12(t^4 + 6^2)
\quad \text{so that} \quad
4 t^2 -  (t^2 - 1)^2 \le  - \tfrac{t^4}{2} + 17.
\]
 As a result,
\[
\int_1^\infty \exp\left( \tfrac{\alpha_c^4 Y^4}{4 \sigma^2}(4 t^2 - (t^2-1)^2)  \right) (1 + | t |^7)\rd t \le  2 \exp\left( \tfrac{17 \alpha_c^4 Y^4}{4 \sigma^2 } \right) \int_0^\infty  \exp\left( \tfrac{- \alpha_c^4 Y^4}{8\sigma^2}t^4 \right)  t^7 \rd t.
\]
We finally perform the change of variable $u = \frac{\alpha_c Y}{8^{1/4} \sqrt{\sigma}}t$ and get
\[
\int_0^\infty  \exp\left( \tfrac{- \alpha_c^4 Y^4}{8\sigma^2}t^4 \right)  t^7 \rd t = \left( \frac{8^{1/4} \sqrt{\sigma}}{\alpha_c Y} \right)^8 \int_0^\infty \re^{-u^4} u^7 \rd u,
\]
where the right-hand side is uniformly bounded for $0<\sigma \le \sigma_0$ and $Y \ge 1$. Combining all the inequalities concludes the proof of Lemma~\ref{lem:GsigmaY}
\end{proof}

Lemma~\ref{lem:farRealLine} is a consequence of the Lemma~\ref{lem:GsigmaY} and the following one, whose proof is similar to the one of Lemma~\ref{lem:G_eq_F}
\begin{lemma}
\label{lem:G_eq_F_superexp}
For all $0 < \sigma \leq \sigma_{0}$, all $Y \geq 1$,  and all $\bk \in \R^d$, it holds that $G_{\sigma,Y}(\bk) = F_{\sigma}(\bk)$.
\end{lemma}

\section*{Acknowledgements}

The authors are grateful to Gus Hart, Volker Blum and Nicola Marzari for interesting discussions. This work was supported in part by ARO MURI Award W911NF-14-1-0247. This project has received funding from the European Research Council (ERC) under the European Union’s Horizon 2020 research and innovation programme (grant agreement No 810367).


\bibliography{kpts}
\bibliographystyle{plain}
{\footnotesize

\begin{tabular}{rl}
(E. Canc\`es) & \textsc{Universit\'e Paris-Est, CERMICS (ENPC)} \\
 &  F-77455 Marne-la-Vall\'ee, France\\
 &  \textsl{E-mail address}: \href{mailto:eric.cances@enpc.fr}{\texttt{eric.cances@enpc.fr}} \\
 \\
(V. Ehrlacher) & \textsc{Universit\'e Paris-Est, CERMICS (ENPC)} \\
 &  F-77455 Marne-la-Vall\'ee, France\\
 &  \textsl{E-mail address}: \href{mailto:virginie.ehrlacher@enpc.fr}{\texttt{virginie.ehrlacher@enpc.fr}} \\
 \\
(D. Gontier) & \textsc{Universit\'e Paris-Dauphine, PSL Research University, CEREMADE} \\
& 75775 Paris, France \\
& \textsl{E-mail address}: \href{mailto:gontier@ceremade.dauphine.fr}{\texttt{gontier@ceremade.dauphine.fr}} \\
\\
  (A. Levitt) & \textsc{Inria Paris and Universit\'e Paris-Est, CERMICS (ENPC)} \\
 &  F-75589 Paris Cedex 12, France\\
 &  \textsl{E-mail address}: \href{mailto:antoine.levitt@inria.fr}{\texttt{antoine.levitt@inria.fr}} \\
  \\
(D. Lombardi) & \textsc{Inria Paris and Sorbonne Universit\'es, UPMC Univ Paris 6, UMR 7598 LJLL} \\
 &  F-75589 Paris Cedex 12, France\\
 &  \textsl{E-mail address}: \href{mailto:damiano.lombardi@inria.fr}{\texttt{damiano.lombardi@inria.fr}} \\
 \\
\end{tabular}

}

\end{document}